\documentclass[a4paper,10pt]{article}

\usepackage{amssymb, amsmath, mathtools, amsfonts, times, amsthm, amscd, setspace, leftidx, enumerate, xfrac}

\input xy 
\xyoption{all}

\usepackage[titletoc,title]{appendix}
\usepackage{graphicx}
\graphicspath{ {Graphics/} }

\theoremstyle{plain}

\newtheorem{proposition}{Proposition}[section]
\newtheorem{defi}[proposition]{Definition}
\newtheorem{lemma}[proposition]{Lemma}
\newtheorem{thm}[proposition]{Theorem}

\newtheorem{rmk}[proposition]{Remark}
\newtheorem{prop}[proposition]{Proposition}
\newtheorem{ex}[proposition]{Example}

\newcommand{\cnc}[0]{\nabla}
\newcommand{\ct}[1]{\mathcal{#1}}
\newcommand{\ov}[1]{\overline{#1}}
\newcommand{\ovl}[1]{\overleftarrow{#1}}
\newcommand{\ovr}[1]{\overrightarrow{#1}}
\newcommand{\pr}[1]{\prescript{\vee}{}{#1}}
\newcommand{\un}[0]{\mathtt{1}}
\newcommand{\field}[0]{\mathbb{K}}
\newcommand{\biml}[0]{\prescript{l}{A}{\mathcal{E}}_{A}}
\newcommand{\Ibim}[0]{\prescript{}{A}{\mathcal{IE}}_{A}}
\newcommand{\bim}[0]{\prescript{}{A}{\mathcal{M}}_{A}}
\newcommand{\XA}[0]{\mathfrak{X}} 
\newcommand{\omx}[3]{\prescript{}{#1}{#2}_{#3}}
\newcommand{\bm}[1]{\boldsymbol{#1}}
\newcommand{\fr}[1]{\mathfrak{#1}}
\newcommand{\sbt}[0]{{\scriptstyle\bullet}}
\newcommand{\und}[1]{\underline{#1}}
\allowdisplaybreaks
\begin{document}

\title{Hopf Algebroids, Bimodule Connections and Noncommutative Geometry}
\author{Aryan Ghobadi \\ \small{Queen Mary University of London }\\\small{ School of Mathematics, Mile End Road}\\\small{ London E1 4NS, UK }\\ \small{Email: a.ghobadi@qmul.ac.uk}}
\date{}

\maketitle
\begin{abstract}
We construct new examples of left bialgebroids and Hopf algebroids, arising from noncommutative geometry. Given a first order differential calculus $\Omega^{1}$ on an algebra $A$, with the space of left vector fields $\XA^{1}$, we construct a left $A$-bialgeroid $B\XA^{1}$, whose category of left modules is isomorphic to the category of left bimodule connections over the calculus. When $\Omega^{1}$ is a pivotal bimodule, we construct a Hopf algebroid $H\XA^{1}$ over $A$, by restricting to a subcategory of bimodule connections which intertwine with both $\Omega^{1}$ and $\XA^{1}$ in a compatible manner. Assuming the space of 2-forms $\Omega^{2}$ is pivotal as well, we construct the corresponding Hopf algebroid, $\ct{D}\XA$, for flat bimodule connections, and recover Lie-Rinehart Hopf algebroids as a quotient of our construction in the commutative case. We use these constructions to provide explicit examples of Hopf algebroids over noncommutative bases.
\end{abstract}\begin{footnotesize}2010\textit{ Mathematics Subject Classification}: Primary 16T99, 58B32 18D10; Secondary 20L05, 16T05 ,22A22
\end{footnotesize}
\section{Introduction}
The relationship between Hopf algebroids and Hopf algebras is analogous to that of groupoids and groups and since the discovery of significant Hopf algebras or \emph{quantum groups} in the 1980s, there were several attempts to define an analogous notion of Hopf algebroids or quantum groupoids \cite{lu1996hopf,xu1998quantum,xu2001quantum}. Today, the different formulations of these structures are well understood \cite{bohm2018hopf}, and there exists an extensive literature \cite{BOHM2009173,bohm2004hopf,hai2008tannaka,el2016finite,schauenburg2000algebras,szlachanyi2003monoidal},   generalising various properties of Hopf algebras to the setting of Hopf algebroids. Despite this, there continues to be a shortage of examples of Hopf algebroids with noncommutative base algebras. Since classically Lie algebroids and groupoids arise naturally in differential geometry \cite{mackenzie2005general}, we choose to tackle this problem in the setting of noncommutative differential geometry \cite{beggs2019quantum}. In the same spirit, Lie-Rinehart algebras \cite{rinehart1963differential} are regarded as algebraic generalisations of Lie algebroids and their universal enveloping algebras provide a family of example of Hopf algebroids \cite{kowalzig2009hopf,kowalzig2010duality,moerdijk2008enveloping}, with applications to differential geometry \cite{huebschmann1990poisson,huebschmann1998lie} and differential equations \cite{el2016finite}. However, the Lie-Rinehart construction is limited to the commutative setting, whereas we obtain a Hopf algebroid associated to any unital algebra $A$ equipped with a pivotal first order differential calculus of 1-forms. From an algebraic perspective, Hopf algebroids and bialgebroids \cite{schauenburg2000algebras,takeuchi1977groups} lift the closed monoidal structure of the category of $A$-bimodules, over a possibly noncommutative algebra $A$. Meanwhile, bimodule connections \cite{bresser1996non} were introduced to provide a subcategory of connections over a noncommutative space, with a monoidal structure, which lifts that of $A$-bimodules. We construct the left bialgebroid representing this category in Theorem \ref{TBXAbiml}. Assuming the space of 1-forms is a pivotal bimodule, as defined in Section \ref{SPiv}, we construct the Hopf algebroid representing a subcategory of bimodule connections which lifts the closed monoidal structure of the category of $A$-bimodules. When provided with a space of 2-forms which has a pivotal structure compatible with that of the space of 1-forms, we extend this construction to a Hopf algebroid representing flat bimodule connection. In the commutative setting, a Lie-Rinehart algebra with a finitely generated and projective space of vector fields contains all the aforementioned data and its corresponding Hopf algebroid can be recovered as a quotient of our construction in the flat case. In Sections \ref{SExHopf} and \ref{SExFlat}, we utilise our constructions to provide several explicit examples of Hopf algebroids over noncommutative algebras. 

While a single definition for bialgebroids has now been accepted, several definitions of Hopf algebroids have been explored. A bialgebra is called a Hopf algebra if it admits a linear endomorphism called the \emph{antipode}, which lifts the inner homs of $\ct{VEC}$ to its module category. The corresponding generalisation for Hopf algebroids is that of Schauenburg \cite{schauenburg2000algebras}. However, a Schauenburg Hopf algebroid does not need to admit an antipode and an example of such a Hopf algebroid was presented in \cite{krahmer2015lie}. The alternative versions, which involve an antipode are that of Lu \cite{lu1996hopf} and B{\"o}hm, and Szlach{\'a}nyi \cite{bohm2004hopf}, the latter of which are examples of Schauenburg Hopf algebroids. The Hopf algebroids constructed here satisfy Schauenburg's axioms, however in Theorems \ref{THXaAnti} and \ref{TDXAAnti}, we present a criterion for each of our examples to admit invertible antipodes, in the sense of B{\"o}hm, and Szlach{\'a}nyi. In Section \ref{SBial}, we review the relevant definitions of closed monoidal categories and the theory of Hopf algebroids. 

The flavour of noncommutative geometry we employ here is that of noncommutative Riemannian geometry, as presented in \cite{beggs2019quantum}, which is somewhat different from, but not incompatible with, Connes' more well known approach \cite{connes1994noncommutative} coming out of spectral triples and cyclic homology. The algebra of continuous functions on a manifold is replaced by an arbitrary algebra $A$ and the additional data of 1-forms on the manifold is replaced by an $A$-bimodule $\Omega^{1}$ and a linear map $d:A\rightarrow \Omega^{1}$ satisfying the Leibnitz rule, as in Definition \ref{DDifCal}. To capture a more complete picture of geometry, we would require the additional data of higher differential forms. However, our constructions up to Section \ref{SFlat}, only require a first order differential calculus. We review the relevant definitions and provide several examples of such structures in Section \ref{SNCG}. 
 
An important tool in geometry is to understand vector bundles over a manifold. The Serre-Swan theorem tells us that this is the same as looking at finitely generated projective modules over the algebra of smooth functions on the manifold. In differential geometry, one would like to understand differentiation on smooth bundles, which translates to viewing covariant derivatives on these modules. The algebra of smooth functions on a manifold is commutative, and any left module over this algebra can be viewed as a bimodule, with the same left action acting on the right. In particular, one can tensor connections over the algebra. Over a noncommutative algebra however, one must distinguish between left and right connections and there is no natural monoidal structure on either category. To overcome this issue, one must look at left (or right) bimodule connection which consist of a bimodule $M$, instead of a left module, with a left connection $\cnc: M\rightarrow \Omega^{1}\otimes M$ and a bimodule map $\sigma :M\otimes \Omega^{1} \rightarrow \Omega^{1}\otimes M$ called an $\Omega^{1}$-intertwining, satisfying compatibility conditions which are presented in Definition \ref{Dbimcnc}. As demonstrated in \cite{bresser1996non}, the category of left bimodule connections has a monoidal structure, with the tensor of two bimodules with such data having a natural left connection and a compatible $\Omega^{1}$-intertwining. Bimodule connections originally arose in \cite{dubois1996first,DUBOISVIOLETTE1996218} and have continued to be of interest in noncommutative geometry \cite{beggs2017spectral,beggs2014noncommutative,beggs2019quantum,fiore2000leibniz,majid2013noncommutative,majid2018generalised}.

Classically, vector fields over the manifold are dual to the space of 1-forms. However, in the noncommutative case, the bimodule $\Omega^{1}$ can have a left dual bimodule $\XA^{1}$ or a right dual bimodule $\mathfrak{Y}^{1} $. In \cite{beggs2014noncommutative}, given compatible bimodule connections on $\Omega^{1}$ and $\XA^{1}$, the algebra $T\XA^{1}_{\sbt}$ is defined by an associative product on $T_{A}\XA^{1}$, such that the action of elements in $\XA^{\otimes n}$, captures local geometry and the action of vector fields. In Proposition 6.15 of \cite{beggs2019quantum}, it is demonstrated that the category of left $T\XA^{1}_{\sbt}$-modules is isomorphic to the category of left connections over the calculus. Hence, as an algebra $T\XA^{1}_{\sbt}$ is independent of the choice of bimodule connection on $\Omega^{1}$, up to isomorphism. We review this construction and the relevant definitions in Section \ref{Scnc}.

In Section \ref{SBXA}, we construct a left $A$-bialgebroid $B\XA^{1}$ whose category of left modules is isomorphic to the category of left bimodule connections, $\biml$. We first construct a smaller bialgebroid in Section \ref{SInter}, whose category of modules is isomorphic to the category of $A$-bimodules with $\Omega^{1}$-intertwinings, $\bim^{\Omega^{1}}$. We denote this algebra by $B(\Omega^{1} )$ and construct $B\XA^{1}$ as a quotient of the free product of $B(\Omega^{1} )$ and $T\XA^{1}_{\sbt}$ by the relevant relations. In Section \ref{SExBia}, we describe $B\XA^{1}$ by generators and relations for several differential calculi. 

The authors of \cite{beggs2014noncommutative} conclude by stating that a bialgebroid or Hopf algebroid structure on $T\XA^{1}_{\sbt}$ would be desirable, while a coproduct does not seem to be available. It is well-known that a Hopf algebra $H$, comes equipped with the structure of a commutative algebra in the center of the category of left $H$-modules. A similar phenomenon was conjectured in \cite{beggs2014noncommutative}, since $T\XA^{1}_{\sbt}$ was found to have a commutative algebra structure in the center of the monoidal category $\biml$. While $T\XA^{1}_{\sbt}$ does not admit a bialgebroid structure, it is a subalgebra of the bialgebroid $B\XA^{1}$ whose representations form $\biml$. In Section \ref{SCCATX}, we recover the lax braiding making $T\XA^{1}_{\sbt}$ an object in the monoidal center in \cite{beggs2014noncommutative}, by restricting the coproduct of $B\XA^{1}$ to $T\XA^{1}_{\sbt}$. 

Although the category of left bimodule connections is monoidal, it does not lift the closed structure of $\bim$. In Section \ref{SInv}, we consider bimodule connections with invertible $\Omega^{1}$-intertwinings. In this case, left and right bimodule connections correspond (Remark \ref{RTXop}) and it is the first step towards obtaining a closed monoidal category of connections. Consequently, we construct the bialgebroids $IB(\Omega^{1})$ and $IB\XA^{1}$, which represent the category of bimodules with invertible $\Omega^{1}$-intertwinings and that of invertible bimodule connections, respectively. To obtain a closed monoidal category, we require $\Omega^{1}$ to be \emph{pivotal}. We say a bimodule is pivotal if its left and right dual bimodules are isomorphic. In other words, the space of left vectorfields, $\XA^{1}$, and that of right vectorfields, $\mathfrak{Y}^{1} $, are isomorphic. For a commutative algebra, any left module is a pivotal bimodule when considered as a bimodule. In Section \ref{SPiv}, we show that several examples of differential calculi which are of interest, such as quiver calculi, bicovariant calculi on Hopf algebras and calculi admitting a quantum Riemannian metric, all have a pivotal structure. 

In Section \ref{SPivHpf}, we construct a quotient of $IB(\Omega^{1})$, $H(\Omega^{1})$ so that it admits a bijective antipode. Any bimodule with an invertible $\Omega^{1}$-intertwining map has two induced intertwinings with bimodules $\XA^{1}$ and $\mathfrak{Y}^{1} $, (\ref{EqXsig}) and (\ref{EqsigX}). Since $\Omega^{1}$ is pivotal, the additional relations present in $H(\Omega^{1} )$ make the induced $\XA^{1}$-intertwinings on $H(\Omega^{1})$-modules, inverses. We construct $H\XA^{1}$ as the quotient of $IB\XA^{1}$ by the same relations and observe that it admits a Hopf algebroid structure, Theorem \ref{THXAHopf}. In Theorem \ref{THXaAnti}, we demonstrate that $H\XA^{1}$ admitting an antipode is equivalent to the existence of a suitable linear map $\Upsilon:\XA^{1}\rightarrow A$, which satisfies condition (\ref{EqSpade}).

When provided with the space of 2-forms, $\Omega^{2}$, one can define the notion of curvature for connections and what it means for a connection to have zero curvature or to be \emph{flat}. In Chapter 6 of \cite{beggs2019quantum}, a quotient of $T\XA^{1}_{\sbt}$ called $\ct{D}_{A}$ is constructed to represent the category of flat connections. However, to obtain a monoidal category of flat bimodule connections one needs to assume that the $\Omega^{1}$-intertwinings of the connections extend to $\Omega^{2}$-intertwinings. After briefly reviewing this theory in Section \ref{SFlat}, we construct the corresponding quotient of $H\XA^{1}$ for flat bimodule connections and denote it by $\ct{D}\XA$, Theorem \ref{TDXAHpf}. In Theorem \ref{TDXAAnti}, we provide a criterion for when $\ct{D}\XA$ admits an invertible antipode in the sense of B{\"o}hm, and Szlach{\'a}nyi.

In Section \ref{SCommLie}, we review our construction in the commutative setting. A Lie-Rinehart algebra consists of a commutative algebra $A$ and a Lie algebra $(\XA^{1} ,[,])$, such that $\XA^{1}$ is an $A$-module and $A$ is a $\XA^{1}$-module satisfying additional compatibility conditions. When $\XA^{1}$ is finitely generated and projective, with $\Omega^{1}$ as its dual module, the data of a Lie-Rinehart algebra translates exactly to $\Omega^{1}$ being a first order calculus over $A$ and the calculus extending to $\Omega^{2}=\bigwedge^{2}(\Omega^{1})$, where $\bigwedge^{2}(\Omega^{1})$ is the exterior power of $\Omega^{1}$ as an $A$-module. We review this correspondence and show that the universal enveloping algebra of $(A,\XA^{1})$ is isomorphic to $\ct{D}_{A}$. More generally, if $A$ is commutative and $\Omega^{1}$ is a symmetric bimodule, $T\XA^{1}_{\sbt}$ has a natural Hopf algebroid structure. We remark that both Hopf algebroid structures of $T\XA^{1}_{\sbt}$ and $\ct{D}_{A}$, can be recovered as quotients of $H\XA^{1}$ and $\ct{D}\XA$, in the commutative and Lie-Rinehart settings, respectively. 

In Sections \ref{SExBia} and \ref{SExHopf}, we provide several examples of left bialgebroids and Hopf algebroids, respectively, in terms of generators and relations. For any finite quiver $\Gamma =(V,E)$, we construct a Hopf algebroid over the algebra $\field (V)$, which contains the quiver path algebra as a subalgebra. We describe the structure of $H\XA^{1}$ over a base Hopf algebra, for an arbitrary bicovariant calculus and calculate an explicit example for the group algebra of the Dihedral group of order 6, $\mathbb{C}D_{6}$. Other examples include derivation calculi on any algebra and a specific inner calculus over the algebra of complex 2-by-2 matrices $M_{2}(\mathbb{C})$. In Section \ref{SExFlat}, we construct $\ct{D}\XA$ explicitly in the cases of finite quivers with no loops and $\mathbb{C}D_{6}$.

\section{Preliminaries}
\textbf{Notation.} Throughout this work, $\field$ will denote a field and $A$ an algebra over this field. When necessary we denote the multiplication of $A$ by $.:A\otimes_{\field} A\rightarrow A$ and otherwise we denote $a.b$ by $ab$ for brevity, where $a,b\in A$. We use the notation $[a,b]=ab-ba$ for the commutator of two elements $a,b$. For a vectorspace $V$, $TV$ will denote the free associative algebra $\field\oplus V\oplus V\otimes_{\field}V \oplus\dots $ over the vectorspace $V$. If $R$ and $S$ are two algebras, $R\star S$ will denote the free product of associative algebras $R$ and $S$. We will denote actions of an algebra $A$ on its (left) module $M$, by $am$, where $a\in A$ and $m\in M$, unless otherwise noted. We denote the category of $A$-bimodules by $\bim$ and the category of vectorspaces by $\ct{VEC}$. For any algebra $R$ and an $R$-bimodule $M$, $T_{R}M$ will denote the free monoid generated by $M$ in $\bim$, which is defined on the vectorspace 
$$T_{R}M= R\oplus M\oplus (M\otimes_{R} M)\oplus (M\otimes_{R} M \otimes_{R} M)\oplus\dots $$
For a natural number $n$, we denote $M\otimes_{R}M\otimes_{R}\cdots \otimes_{R} M$ for n copies of $M$, by $M^{\otimes_{R} n}$. Throughout this work $\otimes$ will denote the tensor product over the algebra $A$ and $\otimes_{\field}$ the tensor product over $\field$. We use Sweedler's notation for coproducts of coalgebras and $R|R$-corings $(C,\Delta ,\epsilon)$: for an element $c$, $\Delta (c)=c_{(1)}\otimes c_{(2)}$ where the right hand side is a sum of elements of the form $c_{(1)}\otimes c_{(2)}$ in $C\otimes C$. We denote the 2-by-2 matrix with $1$ in the $(i,j)$-th position and zeros elsewhere by $E_{ij}$. All sums $\sum_{i}$ will be taken over a free index $i$ with values in a finite set. We have omitted $\sum_{i}$ when the sum is taking place over dual bases arising from coevaluation maps and whenever such terms appear with free indeces, summation is implicit.  
\subsection{Bialgebroids and Hopf Algebroids}\label{SBial}
We briefly recall the theory of monoidal categories and refer the reader to \cite{mac2013categories} for additional details. We call $(\mathcal{C}, \otimes, 1_{\otimes},\alpha,l,r)$ a monoidal category where $\mathcal{C}$ is a category, $1_{\otimes}$ an object of $\mathcal{C}$, $\otimes:\mathcal{C}\times \mathcal{C}\rightarrow\mathcal{C}$ a bifunctor and $\alpha: (\mathrm{id}_{\mathcal{C}}\otimes  \mathrm{id}_{\mathcal{C}})\otimes \mathrm{id}_{\mathcal{C}}\rightarrow \mathrm{id}_{\mathcal{C}}\otimes(\mathrm{id}_{\mathcal{C}}\otimes \mathrm{id}_{\mathcal{C}}), l:1_{\otimes}\otimes \mathrm{id}_{\mathcal{C}}\rightarrow \mathrm{id}_{\mathcal{C}}$ and $r:\mathrm{id}_{\mathcal{C}}\otimes 1_{\otimes} \rightarrow \mathrm{id}_{\mathcal{C}}$ natural isomorphisms satisfying coherence axioms as presented in Chapter VII of \cite{mac2013categories}. In what follows $\alpha,l,r$ will all be trivial isomorphisms, hence we will avoid discussing them. The main examples of monoidal categories which we consider here, are the category of vectorspaces over a field and the category of bimodules over an algebra. 

A functor $F:\mathcal{C}\rightarrow\mathcal{D}$ between monoidal categories is said to be \emph{(strong) monoidal} if the exists a natural (isomorphism) transformation $F_{2}(-,-): F(-)\otimes_{\mathcal{D}}F(-)\rightarrow F(-\otimes_{\mathcal{C}}-)$ and a (isomorphism) morphism $F_{0} : 1_{\otimes}\rightarrow F(1_{\otimes})$ satisfying 
\begin{align*}
F_{2}(X\otimes Y,Z)( F_{2}(X,Y)\otimes \mathrm{id}_{F(Z)})=&
\\F_{2}(X, Y\otimes Z)&(\mathrm{id}_{F(X)}\otimes  F_{2}(Y,Z))\alpha_{F(X),F(Y),F(Z)}
\\F(r)F_{2}(X,1_{\otimes})(\mathrm{id}_{F(X)}\otimes F_{0})=&\mathrm{id}_{F(X)}=F(l)F_{2}(1_{\otimes},X)(F_{0}\otimes \mathrm{id}_{F(X)})
\end{align*}
where we have omitted the subscripts denoting the ambient categories, since they are clear from context. If $F$ has a left adjoint, it is said to be part of a \emph{comonoidal adjunction}, and the resulting monad on $\ct{D}$ is called a \emph{bimonad}. Although, we do not use bimonads directly, we are viewing bialgebroids as an example of bimonads and refer the reader to \cite{bohm2018hopf,bruguieres2011hopf}. 

An \emph{algebra} or \emph{monoid} in a monoidal category $\ct{C}$ consists of a triple $(M,\mu ,\eta)$, where $M$ is an object of $\ct{C}$ and $\mu : M\otimes M\rightarrow M$ and $\eta :1_{\otimes} \rightarrow M$ are morphisms in $\ct{C}$ satisfying $\mu (\mathrm{id}_{M}\otimes \eta ) =\mathrm{id}_{M}= \mu (\eta\otimes \mathrm{id}_{M} )$ and $\mu (\mathrm{id}_{M}\otimes\mu )= \mu ( \mu\otimes \mathrm{id}_{M}) \alpha_{M,M,M} $. A \emph{coalgebra} or \emph{comonoid} in $\ct{C}$ can be defined by simply reversing the arrows in the definition of a monoid. 

For an object $X$ in a monoidal category $\mathcal{C}$, we say an object $\pr{X}$ is a \emph{left dual} of $X$, if there exist morphisms $\mathrm{ev}_{X}:\pr{X}\otimes X\rightarrow \un_{\otimes}$ and $\mathrm{coev}_{X}:\un_{\otimes}\rightarrow X\otimes \pr{X}$ such that 
$$(\mathrm{ev}_{X}\otimes \mathrm{id}_{\pr{X}})(\mathrm{id}_{\pr{X}}\otimes \mathrm{coev}_{X})=\mathrm{id}_{\pr{X}} , \quad (\mathrm{id}_{X}\otimes \mathrm{ev}_{X} )(\mathrm{coev}_{X}\otimes \mathrm{id}_{X})=\mathrm{id}_X$$ 
In such a case, we call $X$ a \emph{right dual} for $\pr{X}$. Furthermore, a right dual of an object $X$ is denoted by $X^{\vee}$, with evalutation and coevaluation maps denoted by $\underline{\mathrm{ev}}_{X}:X\otimes X^{\vee}\rightarrow \un_{\otimes}$ and $\underline{\mathrm{coev}}_{X}:\un_{\otimes}\rightarrow X^{\vee}\otimes X$, respectively. The category $\mathcal{C}$ is said to be left (right) \emph{rigid} or \emph{autonomous} if all objects have left (right) duals. If a category is both left and right rigid, we simply call it \emph{rigid}. We call a category $\mathcal{C}$ \emph{left (right) closed} if for any object $X$ there exists an endofunctor $[X,-]^{l}$ (resp. $[X,-]^{r}$) on $\mathcal{C}$ which is right adjoint to $-\otimes X$ (resp. $X\otimes -$). By definition $[-,-]^{l},[-,-]^{r}:\mathcal{C}^{op}\times\mathcal{C}\rightarrow \mathcal{C}$ are bifunctors. If a category is left and right closed, we call it \emph{closed}. Observe that if $X$ has a left (right) dual $\pr{X}$ (resp. $X^{\vee}$), the functor $-\otimes \pr{X}$ (resp. $X^{\vee}\otimes -$) is left adjoint to $- \otimes X$ (resp. $X\otimes -$) and $\pr{X}$ (resp. $X^{\vee}$) is unique up to isomorphism. Furthermore,  if $X$ has a left (right) dual, $\pr{X}\cong [X,1_{\otimes}]^{l}$ (resp. $X^{\vee}\cong [X,1_{\otimes}]^{r}$). We have adopted the notation of \cite{bruguieres2011hopf} here, and what we refer to as a left closed structure is referred to as a right closed structure in various other sources \cite{el2016finite,schauenburg2000algebras}.

It is well known that strong monoidal functors preserve dual objects i.e. $F(\pr{X})\cong \pr{F(X)}$ with $F_{0}F(\mathrm{ev})F_{2}(\pr{X},X)$ and $ F_{2}^{-1}(X,\pr{X}) F(\mathrm{coev})F_{0}^{-1}$ acting as the evaluation and coevaluation morphisms for $F(\pr{X})$. For left (right) closed monoidal categories $\ct{C}$ and $\ct{D}$, we say a monoidal functor $F:\ct{C}\rightarrow \ct{D}$ is left (right) closed if the canonical morphism $F[X,Y]^{l (r)}_{\ct{C}}\rightarrow [F(X),F(Y)]^{l (r)}_{\ct{D}}$ is an isomorphism for any pair of objects $X,Y$ in $\ct{C}$. 

Before introducing bialgebroids, we briefly recall the theory of Hopf algebras. An algebra $A$ is said to have a bialgebra structure if $(A,\delta , \nu)$ is a coalgebra in the category of vectorspaces satisfying 
$a_{(1)}a'_{(1)}\otimes_{\field} a_{(2)}a'_{(2)}= (aa')_{(1)}\otimes_{\field}  (aa')_{(2)}$ for any $a,a'\in A$, where $\delta (a)=a_{(1)}\otimes_{\field} a_{(2)}$ by Sweedler's notation. There are three additional axioms involving $1$ and $\nu$, which can be found in Chapter 4 \cite{bohm2018hopf}. The coproduct $\delta$ of a Hopf algebra is usually denoted by $\Delta$, but we choose to reserve $\Delta$ for the coproduct of bialgebroids. The category of left $A$-modules for a bialgebra $A$ has a natural monoidal structure which makes the forgetful functor $\prescript{}{A}{\ct{M}}\rightarrow \ct{VEC}$ a strong monoidal functor. A bialgebra is called a \emph{Hopf algebra}, if there exists an anti-multiplicative linear map $S:A\rightarrow A$ satisfying $S(a_{(1)})a_{(2)}= a_{(1)}S(a_{(2)})=\nu (a)1_{A}$ for any $a\in A$. The map $S$ is called the \emph{antipode} and exists if and only if the forgetful functor $\prescript{}{A}{\ct{M}}\rightarrow \ct{VEC}$ is left closed. Moreover, $S$ is bijective if and only if the forgetful functor is closed. 

For an algebra $A$, the \emph{opposite algebra} $A^{op}$ is the algebra structure defined on $A$ by $(\ov{a})(\ov{b})=\ov{ba}$, where we denote elements of the opposite ring with a line above i.e $a,b\in A$ and $\ov{a},\ov{b}\in A^{op}$. It is a well-known fact that $A$-bimodules correspond to left $A\otimes_{\field} A^{op}$-modules, where $A^{e}=A\otimes_{\field} A^{op}$ is called the \emph{enveloping algebra} of $A$. More concretely, there exists an equivalence of categories, between the category of $A$-bimodules $\prescript{}{A}{\ct{M}}_{A} $ and that of left $A^{e}$-modules $\prescript{}{A^{e}}{\ct{M}} $. Hence, we use $\prescript{}{A^{e}}{\ct{M}} $ and $\prescript{}{A}{\ct{M}}_{A} $ interchangeably. We will denote elements of $A^{e}=A\otimes_{\field}A^{op}$ by $a\ov{b}$ where $a\in A$ and $\ov{b}\in A^{op}$.

The category of $A$-bimodules has a natural monoidal structure by tensoring bimodules over the algebra $A$, denoted by $\otimes$, and the algebra $A$ regarded as an $A$-bimodule acting as the unit object. It is well known that a bimodule has a left (right) dual in the monoidal category $\bim$ if and only if it is finitely generated and projective, \emph{fgp} for short, as a right (left) $A$-module. A straight forward proof is presented in Proposition 3.8 of \cite{beggs2019quantum}. In particular, $\bim$ is closed with 
\begin{align*}
[M,N]^{l}:=\mathrm{Hom}_{A}(M,N), \hspace{1cm} &[a\ov{b}f](m)= af(bm),\quad f\in \mathrm{Hom}_{A}(M,N)
\\ [M,N]^{r}:=\prescript{}{A}{\mathrm{Hom}}(M,N),  \hspace{1cm} & [a\ov{b}g](m)= g(ma)b,\quad g\in \prescript{}{A}{\mathrm{Hom}}(M,N)
\end{align*}
where $a\ov{b}\in A^{e}$ and $\mathrm{Hom}_{A}(M,N)$ and $\prescript{}{A}{\mathrm{Hom}(M,N)}$ denote the vectorspaces of right and left $A$-module morphisms from $M$ to $N$, respectively. Explicitly, the units and counits of the adjunctions for the left and right closed structures, are given by 
\begin{align}\label{EqAdj}
&\xymatrix@C+3pt@R-26pt{\varrho^{M}_{N}: N\longrightarrow \mathrm{Hom}_{A}(M,N\otimes M),& \varepsilon^{M}_{N}: \mathrm{Hom}_{A}(M,N)\otimes M\longrightarrow N
\\ \hspace{0.5cm}n\longmapsto f_{n}:(m\mapsto n\otimes m) & \hspace{1.3cm} f \otimes m\longmapsto f(m)}
\\&\xymatrix@C+8pt@R-26pt{\Theta^{M}_{N}: N\longrightarrow \prescript{}{A}{\mathrm{Hom}}(M,M\otimes N),&\hspace{-0.25cm} \Pi^{M}_{N}: M\otimes \prescript{}{A}{\mathrm{Hom}}(M,N)\longrightarrow N
\\ \hspace{0.5cm}n\longmapsto g_{n}:(m\mapsto m\otimes n) & \hspace{1.1cm}m \otimes g\longmapsto g(m)}\nonumber
\end{align}
for any pair of $A$-bimodules $M$ and $N$. Consequently, for a right or left fgp bimodule $M$, we identify $\pr{M}$ by $\mathrm{Hom}_{A}(M,A)$ and $M^{\vee}$ by $\prescript{}{A}{\mathrm{Hom}(M,A)}$. 

The notation for Hopf algebroids and bialgebroids varies quite a bit depending on the reference, but here we refer to \cite{bohm2018hopf}. The Eilenberg-Watts theorem \cite{watts1960intrinsic} tells us that any additive left adjoint functor $F: \prescript{}{A^{e}}{\ct{M}}\rightarrow \prescript{}{A^{e}}{\ct{M}} $ is isomorphic to a functor $\prescript{}{A^{e}}{B}\otimes_{A^{e}} -  $, where $B$ is an $A^{e}$-bimodule.  For an $A^{e}$-bimodule $B$ we denote the functor $\prescript{}{A^{e}}{B}\otimes_{A^{e}} -  $ by $B\boxtimes -:\bim \rightarrow \bim $. This functor absorbs the bimodule structure via its right  $ A^{e} $-action and produces new bimodule actions via its left $ A^{e} $-action. Explicitly, for an $A$-bimodule $M$
\begin{align*}
B \boxtimes M = B\otimes_{\field} M/ \lbrace (br\ov{s})\otimes_{\field} m &- b\otimes_{\field} (rms)\mid m\in M,\  r,s\in A,\   b\in B \rbrace
\\ r (b\boxtimes m)s= ( r\ov{s}b ) \boxtimes m &\quad  \forall m\in M,\  \forall r,s\in A,\  \forall b\in B
\end{align*}
An $A^{e}$-bimodule $B$, can be considered as an $A$-bimodule either by its right or left $A^{e}$-action, and we denote the latter $A$-bimodule by $|B$. We continue to adapt the notation of \cite{bohm2018hopf} and recall the following definitions from Chapter 5. 
\begin{defi}\label{Dalgebroid} Let $A$ be an algebra and $B$ an $A^{e}$-bimodule. 
\begin{enumerate}[(I)]
\item An $A^{e}$-\emph{ring} structure on $B$ consists of a $\field$-algebra structure $(\mu ,1_{B})$ on $B$ with an algebra homomorphism $\eta : A^{e}\rightarrow B$, such that the $A^{e}$-bimodule structure on $B$ is induced by the algebra homomorphism  i.e. $ \mu(\eta\otimes_{\field} \mathrm{id}_{B})$ coincides with the left action of $A^{e}$ and $ \mu(\mathrm{id}_{B}\otimes_{\field}\eta )$ with the right action of $A^{e}$. Equivalently, an $A^{e}$-\emph{ring} structure on $B$ consists of $A^{e}$-bimodule maps $\mu_{A^{e}}: B\otimes_{A^{e}} B\rightarrow B$ and $\eta_{A^{e}}: A^{e} \rightarrow B$, which provide $B$ with the structure of a monoid in the category of $A^{e}$-bimodules. 
\item An $A|A$-\emph{coring} structure on $B$ consists of bimodule maps $\Delta :|B\rightarrow |B\otimes |B$ and $\epsilon :|B\rightarrow A$ satisfying
\begin{align} b_{(1)}\otimes(b_{(2)})_{(1)}\otimes(b_{(2)})_{(2)}=&(b_{(1)})_{(1)}\otimes(b_{(1)})_{(2)}\otimes b_{(2)} \label{EqDelAss} \\ 
\epsilon (b_{(1)})b_{(2)} =b&= \ov{\epsilon (b_{(2)})}b_{(1)} \label{EqCoun}\\ 
\Delta (br\ov{s})=& b_{(1)} r\otimes b_{(2)} \ov{s} \label{EqDelrs}  \\ 
\epsilon (br) =& \epsilon (b\ov{r})\label{EqCounrs}
\end{align} 
for any $b\in B$ and $r,s\in A$, where $\Delta (b)= b_{(1)}\otimes b_{(2)}$ is denoted by Sweedler's notation. Conditions (\ref{EqDelAss}), (\ref{EqCoun}) are equivalent to $(|B,\Delta, \epsilon )$ being a comonoid in the category of $A$-bimodules.
\item A \emph{left} $A$-\emph{bialgebroid} structure on $B$ consists of an $A^{e}$-ring structure $(\mu ,\eta)$ and an $A|A$-coring structure $(\Delta, \epsilon )$ on $B$ satisfying  
\begin{align}
(bb')_{(1)}\otimes (bb')_{(2)} &=b_{(1)}b_{(1)}' \otimes b_{(2)}b_{(2)}' , \label{EqDelMul}
\\ \Delta (1_{B})&=1_{B}\otimes 1_{B}
\\\epsilon (1_{B} )&=1_{A}
\\ \epsilon (b b')= \epsilon &( b\epsilon (b') )=\epsilon \left( b\ov{\epsilon (b')} \right)\label{EqCounMul}
\end{align}
\end{enumerate}
for any $b,b'\in B$, where $1_{B}=\eta (1_{A^{e}})$.
\end{defi}
From the above axioms for an $A|A$-coring $B$, one can deduce that the image of $\Delta$ lands in 
\begin{equation*}\label{Takeuchi}
B\times_{A} B:=\left\lbrace \sum_{i} b_{i}\otimes b'_{i} \in |B\otimes |B \middle\vert \sum_{i} b_{i}\ov{a}\otimes b'_{i} =\sum_{i} b_{i}\otimes b'_{i}a,\  a\in A \right\rbrace
\end{equation*}
Bialgebroids are often defined with reference to $B\times_{A} B$, the \emph{Takeuchi $\times$-product} \cite{takeuchi1977groups}, and often called $\times$-bialgebras. The equivalence of the above definition and the more popular variation is present in both \cite{bohm2018hopf,bohm2004hopf}.

Any $A^{e}$-ring $B$ comes equipped with an algebra map $\eta :A^{e}\rightarrow B$, therefore by restriction of scalars, any $B$-module is equipped with an $A$-bimodule structure and there exists a forgetful functor $U: \prescript {}{B}{\ct{M}} \rightarrow \bim$. In fact, $B\boxtimes -\dashv U:\prescript {}{B}{\ct{M}} \leftrightarrows \bim$ form a free/forgetful adjunction and a left action of $B$ on a bimodule $M$, $B\otimes_{\field} M\rightarrow M$ factors through an $A$-bimodule map $B\boxtimes M\rightarrow M$. In this setting, $B$ has the additional structure of an $A$-bialgebroid, if and only if $U$ is strong monoidal. In particular, the map  
\begin{equation}\label{EqDelMN}
\xymatrix@R-27pt{\Delta_{M,N}:B\boxtimes (M\otimes N)\longrightarrow (B\boxtimes M)\otimes (B\boxtimes N)
\\ \hspace{1.85cm}b\boxtimes (m\otimes n) \longmapsto (b_{(1)}\boxtimes m)\otimes (b_{(2)}\boxtimes n )}
\end{equation}
is well-defined and a bimodule map, for any pair of bimodules $M,N$. Hence, if $(M,\triangleright_{M})$ and $(N,\triangleright_{N})$ are $B$-modules, the $B$-action on $M\otimes N$ is defined by the composition $(\triangleright_{M}\otimes\triangleright_{N})\Delta_{M,N}$. Moreover, the counit $\epsilon$ provides the monoidal unit $A$, with a $B$-action $\epsilon_{0}:B\boxtimes A\rightarrow A$ defined by $b\boxtimes a \mapsto \epsilon (ba)$. 

We must point out that the theory described above is not symmetric. A \emph{right} $A$-bialgebroid structure on $B$ arises when we ask the category of right $B$-modules to be monoidal so that the forgetful functor $\ct{M}_{B}\rightarrow \bim $ is strong monoidal. 

There have been several variations of the Hopf condition for bialgebroids to mimic the Hopf condition for bialgebras. The choice which interests us, is to say a bialgebroid $B$ is Hopf when the forgetful functor $\prescript {}{B}{\ct{M}} \rightarrow \bim $ is closed. This would be the case for Schauenburg Hopf algebroids as introduced in \cite{schauenburg2000algebras}. A class of such Hopf algebroids are those introduced by B{\"o}hm-Szlach{\'a}nyi \cite{bohm2004hopf}, which admit an antipode-like map. 
\begin{defi}\label{DHgebroid} \begin{enumerate}[(I)]
\item A \emph{Schauenburg Hopf algebroid} or $\times$-\emph{Hopf algebra} structure on $B$ consists of an $A$-bialgebroid structure as above, such that the maps 
\begin{align}
\xymatrix@R-26pt{\beta:B\otimes_{A^{op}} B\longrightarrow  B\diamond B & \vartheta:B\odot B\longrightarrow  B\diamond B
\\ \hspace{1.2cm}b\otimes_{A^{op}}b' \mapsto b_{(1)}\diamond b_{(2)}b' &\hspace{1.2cm}b\odot b' \mapsto b_{(1)}b'\diamond b_{(2)} }
\end{align}
where we define
\begin{align*}B\otimes_{A^{op}} B&= B\otimes_{\field} B/ \lbrace b\ov{s}\otimes_{\field}b'- b\otimes_{\field}\ov{s} b' \mid b, b'\in B,\ov{s}\in A^{op} \rbrace  
\\B\odot B&= B\otimes_{\field} B/ \lbrace br\otimes_{\field}b'- b\otimes_{\field}r b' \mid b, b'\in B, r\in A \rbrace 
\\ B\diamond B&= B\otimes_{\field} B/ \lbrace \ov{s}b \otimes_{\field}b'- b\otimes_{\field}s b' \mid b, b'\in B, s\in A \rbrace  \end{align*} 
are invertible.
\item A \emph{B{\"o}hm-Szlach{\'a}nyi Hopf algebroid} structure on $B$ consists of an $A$-bialgebroid structure as above and an anti-algebra automorphism $S:B\rightarrow B$ satisfying
\begin{align}
S(\eta(\ov{a}))=\eta(a)\quad &
\\S(b_{(1)})_{(1)}b_{(2)}\diamond S(b_{(1)})_{(2)}&= 1\diamond S(b) \label{EqSCond}
\\ S^{-1}(b_{(2)})_{(1)}\diamond S^{-1}(b_{(2)})_{(2)}b_{(1)}&=S^{-1}(b)\diamond 1 \label{EqSinvCond}
\end{align}
for all $b\in B$ and $\ov{a}\in A^{op}$.
\end{enumerate}
\end{defi}
\textbf{Notation.} In what follows, we will simply write $a$ and $\ov{a}$ to refer to $\eta (a)$ or $\eta (\ov{a})$ as the images of elements $a\in A$ and $\ov{a}\in A^{op}$ in an $A^{e}$-ring $B$. This is not an abuse of notation, since the multiplication and the module structure of an $A^{e}$-ring $B$ coincide.

If $B$ is a Schauenburg Hopf algebroid and $\beta, \vartheta$ are invertible, we denote $\beta^{-1} (b \diamond 1) =b_{(+)}\otimes_{A^{op}} b_{(-)}$ and $\vartheta^{-1} (1\diamond b) =b_{[+]}\odot b_{[-]} $. In this case, the closed structure of $\bim$ is lifted to $\prescript{}{B}{\ct{M}}$ via the following $B$-actions:
\begin{equation}\label{EqClosHpf}
\xymatrix@R-26pt@C-10pt{B\boxtimes \mathrm{Hom}_{A}(M,N)\rightarrow \mathrm{Hom}_{A}(M,N) & B\boxtimes \prescript{}{A}{\mathrm{Hom}}(M,N)\rightarrow \prescript{}{A}{\mathrm{Hom}}(M,N)
\\ b\boxtimes f\mapsto (m\mapsto b_{(+)}f( b_{(-)}m)) & b\boxtimes g\mapsto (m\mapsto b_{[+]}g( b_{[-]}m))}
\end{equation}
for any pair of $A$-bimodules $M,N$. Equivalently, $\vartheta^{-1}$ and $\beta^{-1}$ can be recovered, if one has a well-defined actions of $B$ on the inner homs, such that the units and counits presented in (\ref{EqAdj}) are $B$-module morphisms. For a left bialgebroid $B$, this is precisely what it means for the forgetful functor $\prescript{}{B}{\ct{M}}\rightarrow \bim$ to be closed.

If $B$ is a B{\"o}hm-Szlach{\'a}nyi Hopf algebroid with an invertible antipode $S:B\rightarrow B$ then the inverses of $\beta,\vartheta$ are given by 
\begin{align}\label{EqClosAntip}
\beta^{-1} (b \diamond b') &=S^{-1}(S(b)_{(2)})\otimes_{A^{op}} S(b)_{(1)}b'
\\ \vartheta^{-1} (b\diamond b') &=S\left( S^{-1}(b')_{(2)}\right)\odot S^{-1}(b')_{(1)}b 
\end{align}

Finally, we refer the reader to Chapter 5 of \cite{bohm2018hopf} and \cite{bohm2004hopf} for further details on these elementary facts. We conclude by presenting the following Theorem which motivates our work when looking at the category of bimodule connections: 
\begin{thm}\label{TSzl}\cite{szlachanyi2003monoidal} For an algebra $A$ and an abelian monoidal category $\ct{C}$, if $F:\ct{C}\rightarrow \bim$ is an additive functor with a left adjoint $G$, such that $FG:\bim\rightarrow \bim$ has a right adjoint , then $F$ is (closed) strong monoidal if and only if $\ct{C}$ is equivalent to $\prescript {}{B}{\ct{M}}$ for a left (Hopf) bialgebroid $B$. 
\end{thm}
From this point onwards, we only consider left bialgebroids and left Hopf algebroids, when refering to bialgebroids or Hopf algebroids. 
\subsection{Noncommutative Geometry Framework and Examples}\label{SNCG}
Here we provide a brief introduction to noncommutative Riemmanian geometry as presented in \cite{beggs2019quantum}. In particular, all details and proofs relating to the examples presented here can be found in Chapter 1 of \cite{beggs2019quantum}. 
\begin{defi}\label{DDifCal} By a \emph{(first order) differential calculus} over an algebra $A$, we refer to an $A$-bimodule $\Omega^{1}$ along with a linear map $d:A\rightarrow \Omega^{1}$ satisfying $d(ab)=(da)b+a(db)$, for any $a,b\in A$. 
\end{defi} 
In \cite{beggs2019quantum} and most of the literature, the additional condition $\Omega^{1} = \mathrm{Span}_{\field}\lbrace adb\mid a,b\in A\rbrace$ (the \emph{surjectivity condition}) is also required. If this property does not hold, $(\Omega^{1},d)$ is often called a \emph{generalized calculus} \cite{majid2018generalised}. However, in what will follow, we do not require the surjectivity condition. If $\mathrm{ker} (d) = \field .1 $, where $1$ is the unit of algebra $A$, we say the calculus is \emph{connected}. Every algebra has a natural largest connected differential calculus, namely the universal calculus $\Omega^{1}_{uni}=\mathrm{ker}(.)\subseteq A\otimes_{\field} A$, with differential $da= 1\otimes a -a\otimes 1 $. Any first order differential calculus satisfying the surjectivity condition arises as a quotient of the universal calculus. 
\begin{ex}\label{Eclassical} [Classical Example] Let $\mathsf{M}$ be a smooth manifold, $A=\ct{C}^{\infty}(\mathsf{M})$ the algebra of smooth functions on $\mathsf{M}$, $\Omega^{1}$ the space of 1-forms and $d:A\rightarrow \Omega^{1}$ the usual differential on smooth functions. In this case, $A$ is commutative and $\Omega^{1}$ has a bimodule structure where the left and right module structure agree. 
\end{ex}
We say a differential calculus is called \emph{inner} if there exists an element $\theta \in \Omega^{1}$ such that $da= [\theta ,a]$. Notice that even over a commutative algebra $A$, inner calculi are only possible because we are not requiring $\Omega^{1}$ to have the same left and right module structure. 
\begin{ex}\label{EGrph} [Finite Quivers \cite{connes1994noncommutative,majid2013noncommutative}] Let $V$ be a finite set, and $A= \field (V)=\lbrace f:V\rightarrow \field \rbrace$ be the algebra of functions on $V$. There exists a natural basis for $A$, namely $\lbrace f_p \mid p\in V\rbrace$, where $f_{p}(q)=\delta_{p,q}$ for any $p,q\in V$. In fact, $A$ is the finite dimensional algebra with a complete set of idempotents $T:=\lbrace f_p \mid p\in V\rbrace$ as its basis and is thereby semisimple. Any $A$-bimodule $M$ decomposes as $M=\bigoplus_{p,q\in V} \prescript{}{p}{M}_{q}$ such that $f_{x}mf_{y}=\delta_{x,p}\delta_{y,q} m$, for $m\in \prescript{}{p}{M}_{q}$. Hence, a bimodule over $A$ corresponds to the choice of a \emph{directed graph} or \emph{quiver}, on the set of points $V$: for a set of edges $E\subset V\times V$, and an edge $e\in E$, we denote its corresponding basis element in $\Omega^{1}$ by $\ovr{e}$, so that 
$$ \prescript{}{p}{\Omega^{1}}_{q} = \mathrm{Span}_{\field} \left\lbrace \ovr{e}\mid s(e)=p,\ t(e)=q\right\rbrace $$
where $s,t:E\rightarrow V$ are the usual source and target maps. The differential structure is defined by
$$df = \sum_{e\in E} [f(t(e))-f(s(e))]\ovr{e}$$
The calculus is inner with $\theta = \sum_{e\in E}\ovr{e}$. The surjectivity condition holds if and only if no edge has the same source and target and two points have at most one edge between them.
\end{ex}

If $\Omega^{1}$ is a left (right) free module over $A$ with a basis of cardinality $n$, we say $\Omega^{1}$ is left (right) \emph{parallelised} with left (right) \emph{cotangent dimension} $n$. If $\Omega^{1}$ is both left and right parallelised, we call it simply parallelised. Although our work does not require $\Omega^{1}$ to be parallelised, such bimodules facilitate our calculations when producing examples.

\begin{ex}\label{EDrv}[\emph{Derivation Calculus}] First order differential calculi on $\Omega^{1}=A$, regarded as an $A$-bimodule, are just \emph{derivations} $d: A\rightarrow A$ i.e. endomorphisms $d$ satisfying the Leibnitz rule as presented in Definition \ref{DDifCal}.
\end{ex}
\begin{ex}\label{EM2C}[$M_{2}(\mathbb{C})$] The complete moduli of surjective first order calculi for the algebra of $2$-by-$2$ matrices $A=M_{2}(\mathbb{C})$ has been described in Example 1.8 of \cite{beggs2019quantum}. An example of such calculi is $\Omega^{1}=M_{2}(\mathbb{C})\oplus M_{2}(\mathbb{C})$ as a free bimodule, equipped with an inner calculus by $\theta= E_{12}\oplus E_{21}$.  
\end{ex}

It is well known that \emph{bicovariant calculi} \cite{woronowicz1989differential} or \emph{Hopf bimodules} over Hopf algebras are parallelised. In particular, a Hopf module $\Omega^{1}$ over a Hopf algebra $(A,\delta,\nu , s)$ is free as a right $A$-module and decomposes as $\Omega^{1} \cong \Lambda\otimes_{\field} A$, for a particular subspace $\Lambda\subseteq \Omega^{1}$. Under this decomposition, the right $A$-action arises from $A$, solely. The left $A$-action on $\Omega^{1}$ arises by considering $\Lambda\otimes_{\field} A$ as the tensor of two left $A$-modules, where $\Lambda$ has an induced left $A$-action $\triangleright$ defined by $a\triangleright\lambda = a_{(1)}\lambda s(a_{(2)})$, for any $a\in A$ and $ \lambda\in \Lambda$. Consequently, the left action of $\Omega^{1}$ translates to
$$b(\lambda\otimes_{\field} a ) = b_{(1)}\triangleright \lambda \otimes_{\field} b_{(2)}a =b_{(1)}\lambda s(b_{(2)}) \otimes_{\field}b_{(3)}a$$
As we will see in Section \ref{SPiv}, $\Omega^{1}$ is free as a left $A$-module, in a symmetric manner. A Hopf bimodule has compatible $A$-bimodule and $A$-bicomodule structures, which give rise to the above structure. In particular, $\Lambda = \lbrace \omega\in\Omega^{1} \mid \delta_{R}(\omega)=\omega\otimes_{\field}1 \rbrace$, where $\delta_{R}$ denotes the right $A$-coaction. The left coaction of $\Omega^{1}$ restricts to $\Lambda$ and along with the left action $\triangleright$, make $\Lambda$ a left \emph{Yetter-Drinfeld module}. A first order differential calculus $\Omega^{1}$ over a Hopf algebra is called a bicovariant differential calculus if $\Omega^{1}$ is a Hopf bimodule. Bicovariant differential calculi which satisfy the surjectivity condition are in bijection with \emph{Ad-stable} left ideals of $A^{+}=\mathrm{ker}(\nu)$. For further detail on bicovariant calculi, we refer the reader to Section 2.3 of \cite{beggs2019quantum} and conclude with a particular example of bicovariant calculi over a Hopf algebra.  

\begin{ex}\label{EGrpAlg}[\emph{Group Algebra} \cite{majid1998classification}] Given a group $G$, a left module $G$-module $(\Lambda,\triangleright )$ and a 1-cocycle $\zeta\in Z^{1}(G,\Lambda)$ i.e. a map $\zeta:\field G \rightarrow \Lambda $ such that $\zeta (gh)=g \triangleright\zeta (h)+ \zeta (g)$, there is a corresponding differential calculus $\Omega^{1} =\Lambda \otimes_{\field} \field G $ over the group algebra $\field G$ with the differential defined by 
$$d(g)= \zeta (g)\otimes_{\field} g$$ 
The calculus is inner if and only if $\zeta$ is exact i.e. there exists an element $\theta \in \Lambda$ such that $\zeta(g) =g\triangleright \theta -\theta $. When $G$ is finite and $|G|$ is invertible in $\field $, then the calculus is always inner with $\theta =\frac{1}{|G|}\sum_{g\in G} \zeta (g)$.
\end{ex}
When looking at the classical case, first order differential calculus only contains the data for 1-forms. To capture a true generalisation of classical geometry one must consider the space of all differential forms.
\begin{defi}\label{DDGA} A \emph{differential graded algebra} or \emph{DGA} on an algebra $A$ is a graded algebra $(\Omega^{\sbt}= \oplus_{n\geq 0} \Omega^{n}, \wedge)$ with $\Omega^{1}_{0}=A$ and a differential $d:\Omega^{n}\rightarrow \Omega^{n+1}$ such that $d^{2}=0$ and 
$d(\omega\wedge \rho ) = (d\omega)\wedge\rho + (-1)^{n}\omega\wedge (d\rho )$, 
where $\rho\in\Omega^{\sbt}$ and $\omega\in\Omega^{n}$, hold for all $n\geq 0$.
\end{defi}
If a DGA is generated by $A$ and $dA$, we refer to it as an \emph{exterior algebra} on $A$. Observe that given a DGA $(\Omega^{\sbt}, \wedge)$ on $A$, $(\Omega^{1},d)$ form the data for a first order differential calculus. Conversely, every first order calculus $(\Omega^{1},d)$ can be extended to an exterior algebra on $A$ called its \emph{maximal prolongation}, such that any exterior algebra on $A$, which agrees with $(\Omega^{1},d)$ on its first grading and differential, is a quotient of the maximal prolongation by a differential ideal. Further details can be found in Section 1.5 of \cite{beggs2019quantum}.
\subsection{Connections}\label{Scnc}
\begin{defi}\label{Dcnc} If $(\Omega^{1} ,d)$ is a differential calculus on the algebra $A$, by a \emph{left connection} or \emph{left covariant derivative}, we mean a left $A$-module $M$ and a linear map $\cnc :M\rightarrow\Omega^{1}\otimes M$ satisfying 
$$\cnc (am) = a \cnc(m) +da\otimes m $$
for all $a\in A$ and $m\in M$. 
\end{defi}
A right connection can be described similarly, as a right $A$-module $M$ with a linear map $\cnc :M\rightarrow M\otimes\Omega^{1}$ satisfying $\cnc (ma) =  \cnc(m)a + m\otimes da$. The category of left (right) connections on a differential calculus which has left (right) connections $(M,\cnc_{M})$ as objects and left (right) module maps $f:M\rightarrow N$ satisfying $(\mathrm{id}_{\Omega^{1}}\otimes f)\cnc_{M} = \cnc_{N} f$ (resp. $(f \otimes\mathrm{id}_{\Omega^{1}} )\cnc_{M} = \cnc_{N} f$), as morphisms between $f:(M,\cnc_{M})\rightarrow (N,\cnc_{N}) $, is denoted by $\prescript{}{A}{\mathcal{E}}{}$ (resp. $\mathcal{E}_{A}$).

A natural question which arises is when can one describe $\prescript{}{A}{\mathcal{E}}$ as modules over an algebra. This question was answered in Chapter 6 of \cite{beggs2019quantum}. When $\Omega^{1}$ is right fgp, we denote $\XA^{1}:=\pr{\Omega^{1}}$ with $\mathrm{ev}:\XA^{1}\otimes\Omega^{1} \rightarrow A$ and $\mathrm{coev}:A\rightarrow \Omega^{1}\otimes\XA^{1} $ as the respective evaluation and coevalution maps for dual bimodules, as described in Section \ref{SBial}. The bimodule $\XA^{1}$ can be thought of as the space of vector fields on the noncommutative space, since it is dual to the space of 1-forms. In this setting, $\prescript{}{A}{\mathcal{E}}\cong\prescript{}{T\XA^{1}_{\sbt}}{\mathcal{M}}$, where $T\XA^{1}_{\sbt}$ is the associative algebra defined as 
$$T\XA^{1}_{\sbt}=A*T\XA^{1} / \left\langle a\sbt x -ax,\  x\sbt a-xa -\mathrm{ev}(x,da)\mid a\in A,\  x\in\XA^{1} \right\rangle $$ 
where $\sbt$ denotes the associative product in $A*T\XA^{1}$ and  a left $T\XA^{1}_{\sbt}$-module $M$ has a left $A$-module structure by restriction of scalars. Hence, the action of $T\XA^{1}_{\sbt}$ on $M$ restricts to a map $\triangleright :T\XA^{1}_{\sbt}\otimes M\rightarrow M$ and the corresponding left connection $\cnc :M\rightarrow \Omega^{1}\otimes M$ is defined by 
$$\cnc = (\mathrm{id}_{\Omega^{1}}\otimes \triangleright )(\mathrm{coev} \otimes \mathrm{id}_{M}) $$
Conversely, any left connection $(M,\cnc)$ induces an action of $T\XA^{1}_{\sbt}$ on $M$, with the action of $A$ agreeing with the left $A$-module structure on $M$ and the action of elements of $\XA^{1}$ being defined by $ (\mathrm{ev}\otimes \mathrm{id}_{M} ) (\mathrm{id}_{\XA^{1}}\otimes_{\field}\cnc ) $. 

\begin{rmk}\label{RTXA} The ideal quotiented out from $A*T\XA^{1}$ demonstrates that we can describe $T\XA^{1}_{\sbt}$ via an associative product on $T_{A}\XA^{1}$. We have an isomorphism of vectorspaces  
\begin{equation}\label{EqXtenX}
\XA^{1}\otimes_{\field}\XA^{1}\cong (\XA^{1}\otimes\XA^{1}) \oplus \mathrm{Span}\lbrace xa\otimes_{\field} y- x\otimes_{\field} ay \mid x,y\in\XA^{1}, a\in A \rbrace
\end{equation}
If $x\otimes_{\field} y= \sum_{i} x_{i}\otimes y_{i} \oplus \sum_{j}(w_{j}a_{j}\otimes_{\field} z_{j}- w_{j}\otimes_{\field} a_{j}z_{j})$ by the above decomposition, then $x\sbt y = \sum_{i} x_{i}\sbt y_{i} + \sum_{j}\mathrm{ev}(w_{j},da_{j}) z_{j}$ in $T\XA^{1}_{\sbt}$. Extending this idea to iterated products of elements of $\XA^{1}$, we can organise $T\XA^{1}_{\sbt}$ as an associative product on $T_{A}\XA^{1}$. In \cite{beggs2014noncommutative}, $T\XA^{1}_{\sbt}$ is presented as associative product on the vector space $T_{A}\XA^{1}$ to begin with. However, the multiplication of elements of $\XA^{\otimes m}$ and $\XA^{\otimes n}$ are defined iteratively, by requiring $\Omega^{1}$ and $\XA^{1}$ to have compatible bimodule connections. This description of $T\XA^{1}_{\sbt}$ is meant to encode the classical action of vector fields. Since we are only interested in $T\XA^{1}_{\sbt}$ as an algebra and $T\XA^{1}_{\sbt}$ is independent of the choice of bimodule connection on $\Omega^{1}$, up to isomorphism, the above definition is satisfactory. But we must emphasise that arranging $T\XA^{1}_{\sbt}$ as a product on $T_{A}\XA^{1}$, as above, will not produce the same product as the method of \cite{beggs2014noncommutative} via bimodule connection, but an isomorphic one. 
\end{rmk}
\begin{defi}\label{Dbimcnc} If $(\Omega^{1} ,d)$ is a differential calculus on the algebra $A$, by a \emph{left bimodule connection}, we mean an $A$-bimodule $M$ and a linear map $\cnc :M\rightarrow\Omega^{1}\otimes M$ such that there exists a bimodule map $\sigma :M\otimes\Omega^{1} \rightarrow \Omega^{1}\otimes M$ satisfying 
$$\cnc (am) = a\cnc (m) +da\otimes m, \quad \quad  \cnc (ma) = \cnc (m)a +\sigma(m\otimes da)$$
for all $a\in A$ and $m\in M$. 
\end{defi}
A right bimodule connection is defined symmetrically as a bimodule $M$ with a right connection $\cnc$ and a \emph{left} $\Omega^{1}$-intertwining $\sigma :\Omega^{1}\otimes M\rightarrow M\otimes \Omega^{1}$ satisfying $\cnc (am) = a\cnc (m) +\sigma (da \otimes m)$ for all $m\in M$ and $a\in A$. Observe that a left bimodule connection structure on a bimodule does not imply the existence of a right bimodule connection structure. The category of left bimodule connections on a differential calculus, which has left bimodule connections $(M,\cnc_{M},\sigma)$ as objects and bimodule maps $f:M\rightarrow N$ satisfying $(\mathrm{id}_{\Omega^{1}}\otimes f)\cnc_{M} = \cnc_{N} f$ and $\sigma_{N}(f\otimes \mathrm{id}_{\Omega^{1}})=(\mathrm{id}_{\Omega^{1}}\otimes f) \sigma_{N}$ as morphisms $f:(M,\cnc_{M},\sigma_{M})\rightarrow (N,\cnc_{N},\sigma_{N}) $, is denoted by $\prescript{l}{A}{\mathcal{E}}_{A}$. The category of right bimodule connections is defined symmetrically and denoted by $\prescript{r}{A}{\mathcal{E}}_{A}$. 

For a surjective calculus, a triple $(M,\cnc,\sigma )$ being a left bimodule connection is a property for a given bimodule $M$ with a left connection $\cnc$ and $\sigma$ is not additional data. Although we do not focus on surjective calculi, we comment on the features of our construction in the surjective setting in Remark \ref{RSrjBim}. In the classical setting, where $A$ is a commutative algebra and we regard any left module as a bimodule with the right action coinciding with the left action, every left connection is a left bimodule connection with the flip map as $\sigma$. We look at the classical case in more detail in Section \ref{SCommLie}. 

The benefit of working with bimodule connections is that $\biml$ admits a monoidal structure. If $(M,\cnc_{M},\sigma_{M})$ and $(N,\cnc_{N},\sigma_{N}) $ are left bimodule connections, then one can define $(M,\cnc_{M},\sigma_{M})\otimes (N,\cnc_{N},\sigma_{N}) $ as the triple $(M\otimes N,\cnc_{M\otimes N},\sigma_{M\otimes N})$ where 
$$\cnc_{M\otimes N}= \cnc_{M} \otimes \mathrm{id}_{N} +(\sigma_{M} \otimes \mathrm{id}_{N})(\mathrm{id}_{M}\otimes \cnc_N ) $$
$$\sigma_{M\otimes N}=(\sigma_{M} \otimes \mathrm{id}_{N})(\mathrm{id}_{M} \otimes\sigma_{N})  $$ 
One must of course check that $\cnc_{M\otimes N}: M\otimes N \rightarrow \Omega^{1}\otimes M\otimes N$ is a well defined map, which is demonstrated in Section 3.4.2 of \cite{beggs2019quantum}: 
\begin{prop}\hspace{0.001cm}[Theorem 3.78 \cite{beggs2019quantum}]  The category $\biml$ is monoidal with the tensor product defined as above and the triple $(A,d, \mathrm{id}_{\Omega^{1}})$ as the unit object.  
\end{prop}
The category of bimodule connections comes equipped with a forgetful funtor $U:\biml\rightarrow \bim$ which sends a triple $(M,\cnc_{M},\sigma_{M})$ to its underlying bimodule $M$. Furthermore, the described monoidal structure on $\biml$ applies the usual bimodule tensor product on the underlying bimodules of the bimodules connections. In other words, $U$ is strong monoidal. By Theorem \ref{TSzl}, $\biml$ can be written as the category of modules over a bialgebroid if and only if it is abelian and $U$ is co-continous and has a left adjoint. This is the case when $\Omega^{1}$ is right fgp. 
\section{Bialgebroids Representing Bimodule Connections}
Before we construct the bialgebroid representing $\biml$, we must look at the category of bimodules which intertwine with $\Omega^{1}$ and construct the bialgebroid representing this category. 
\subsection{Category of Intertwining Modules}\label{SInter}
Let $\Omega^{1}$ be a right fgp $A$ bimodule and $\XA^{1}$ be its left dual with $\mathrm{coev}: A\rightarrow \Omega^{1}\otimes\XA^{1}$ and $\mathrm{ev}: \XA^{1} \otimes \Omega^{1}\rightarrow A$ as described in Section \ref{SBial}. Denote $\mathrm{coev}(1)=\sum_{i} \omega_{i}\otimes x_{i}$ so that $\sum_{i} a\omega_{i}\otimes x_{i}=\sum_{i} \omega_{i}\otimes x_{i}a=\mathrm{coev}(a)$ holds for any $a\in A$. 
\begin{defi}\label{Dintcat} For an $A$ bimodule $\Omega^{1}$, we define the category of $\Omega^{1}$\emph{-intertwined bimodules} to have pairs $(M,\sigma_{M} )$, where $M$ is an $A$-bimodule along with a bimodule map $\sigma_M : M\otimes\Omega^{1}\rightarrow\Omega^{1}\otimes M$, as objects and $f:M\rightarrow N$ bimodule maps satisfying $\sigma_{N}(f\otimes \mathrm{id}_{\Omega^{1}})=(\mathrm{id}_{\Omega^{1}}\otimes f) \sigma_{N}$ as morphisms. We denote this category by $\prescript{}{A}{\mathcal{M}}^{\Omega^{1}}_{A}$.
\end{defi}
Let $\mathfrak{M}= \XA^{1}\otimes_{\field}\Omega^{1}$, then $\mathfrak{M}$ has a $A^{e}$-bimodule structure:
$$a\ov{a'}(x,\omega) b\ov{b'}= (axb, b'\omega a') $$
for any $a,a',b,b'\in A$, where we denote arbitrary elements of $\mathfrak{M}$ by $(x,\omega )$. Hence, define $B(\Omega^{1}):=T_{A^{e}}(\XA^{1}\otimes_{\field}\Omega^{1})$ as an algebra and denote its multiplication by $\sbt$ so that
\begin{align}
a\sbt (x,\omega)&= (ax,\omega ) ,\hspace{1.5cm} (x,\omega)\sbt a= (xa,\omega )\label{EqXAOMA}
\\ (x,\omega)\sbt \ov{a}&= (x,a\omega ),\hspace{1.5cm} \ov{a}\sbt (x,\omega)=(x,\omega a ) \label{EqXAOMAop}
\end{align} 
hold for $(x,\omega)\in\mathfrak{M}$ and $a\in A$. Equivalently, $B(\Omega^{1})$ is isomorphic to the quotient of the algebra $T(\mathfrak{M}\oplus A)$ by the ideal generated by the set of relations (\ref{EqXAOMA}) and (\ref{EqXAOMAop}), for all $(x,\omega)\in \mathfrak{M}$ and $a\in A$.

To obtain a bialgebroid structure on $B(\Omega^{1})$, we define the coproduct and counit for elements of $A^{e}$ and $\mathfrak{M}$, and extend them multiplicatively to $B(\Omega^{1})$ by $\Delta (m\sbt n)= m_{(1)}\sbt n_{(1)}\otimes m_{(2)}\sbt n_{(2)}$ and $\epsilon (m\sbt n) = \epsilon (m \sbt \epsilon (n))$. 
\begin{align}
\Delta (a\ov{b})&= a\otimes \ov{b}
\\ \Delta ((x,\omega ))&=  (x,\omega_{i} )\otimes (x_{i},\omega )
\\ \epsilon (a\ov{b})= ba \hspace{1cm} & \hspace{1cm} \epsilon ((x,\omega))=\mathrm{ev}(x,\omega)
\end{align}
for $a\ov{b}\in A^{e}$ and $(x,\omega )\in \mathfrak{M}$.
\begin{prop}\label{PBomega} The algebra $B(\Omega^{1})$ along with $\Delta, \epsilon$ has a left $A^{e}$-bialgebroid structure.  
\end{prop}
\begin{proof} It is easy to see that $\Delta$ and $\epsilon$ are well defined with respect to relations (\ref{EqXAOMA}) and (\ref{EqXAOMAop}). We must also check that $\Delta$ and $\epsilon$ are bimodule maps:
\begin{align*}
\Delta (a\sbt (x,\omega) )=&\Delta((ax,\omega ))= a\sbt(x,\omega_{i} )\otimes (x_{i},\omega )=a\Delta((x,\omega ))
\\ \Delta (\ov{a}\sbt (x,\omega) )=&\Delta((x,\omega a ))= (x,\omega_{i} )\otimes \ov{a}\sbt (x_{i},\omega )=\Delta((x,\omega ))a
\\ \epsilon (a\sbt (x,\omega))= \mathrm{ev}(ax,\omega)&=a\epsilon ((x,\omega)) \hspace{0.6cm} \epsilon (\ov{a}\sbt (x,\omega ))= \mathrm{ev}(x\otimes\omega a)= \epsilon((x,\omega))a
\end{align*} 
where $(x,\omega )\in \mathfrak{M}$ and $a\in A$. Now check that $(B(\Omega^{1}),\Delta ,\epsilon )$ is an $A|A$-coring. Coassociativity (\ref{EqDelAss}) and the counit condition (\ref{EqCoun}) follow easily by the definition of $\Delta,\epsilon$ on the generators and  are left to the reader. We briefly check (\ref{EqDelrs}) and (\ref{EqCounrs}) for $a\ov{b}\in A^{e}$ and $(x,\omega )\in \mathfrak{M}$:
\begin{align*}
\Delta((x,\omega)\sbt a\ov{b})=\Delta ((xa,b\omega ))=  (x,\omega_{i} )\sbt a\otimes (x_{i},\omega ) \sbt  \ov{b} 
\\\epsilon ((x,\omega)\sbt a)=\mathrm{ev} (xa \otimes\omega )=\mathrm{ev} (x\otimes a\omega )=\epsilon ((x,\omega)\sbt \ov{a})
\end{align*}
Since $\Delta$ and $\epsilon$ are well-defined on the generators and (\ref{EqDelrs}) holds, they can be extended multiplicatively to an $A|A$-coring structure on $B(\Omega^{1} )$. By defining the comultiplication and counit multiplicatively, $B(\Omega^{1})$ automatically satisfies the bialgebroid axioms. \end{proof} 
Notice that for $(x_{i},\omega_{i})\in\mathfrak{M}$, where $1\leq i\leq n$, 
$$ \epsilon \big((x_{1},\omega_{1})\sbt (x_{2},\omega_{2})\sbt \cdots \sbt (x_{n},\omega_{n})\big)= \mathrm{ev}^{\langle n\rangle} (x_{1}\otimes x_{2}\otimes \dots \otimes x_{n}\otimes \omega_{n}\otimes\dots \otimes\omega_{1})$$
where $\mathrm{ev}^{\langle n\rangle}$ is defined iteratively by $\mathrm{ev}^{\langle n+1\rangle}= \mathrm{ev}(\mathrm{id}_{\XA^{1}} \otimes \mathrm{ev}^{\langle n\rangle}\otimes \mathrm{id}_{\Omega^{1}})$ and $\mathrm{ev}^{\langle 1\rangle} =\mathrm{ev}$.
\begin{thm}\label{TBOmga} There exists an isomorphism of categories $\prescript{}{B(\Omega^{1} )}{\mathcal{M}}\cong \prescript{}{A}{\mathcal{M}}^{\Omega^{1}}_{A}$. 
\end{thm}
\begin{proof} Any $B(\Omega^{1})$-module $M$ has an induced $A$-bimodule structure, by restriction of scalars to $A^{e}$. Moreover, $M$ has an induced $\Omega^{1}$-intertwining $\sigma$ defined by 
$$\sigma (m\otimes\omega)=  \omega_{i}\otimes (x_{i},\omega)m $$
for $m\in M$ and $\omega\in \Omega^{1}$. The left column of relations in (\ref{EqXAOMA}) and (\ref{EqXAOMAop}) imply that the map $\sigma$ is well defined, while the right column of relations make $\sigma$ a bimodule map. Conversely, an $A$-bimodule $M$ with an $\Omega^{1}$-intertwining map, $\sigma$, has an induced action of $B(\Omega^{1})$ defined on the generators of the algebra by 
$$(x,\omega)m= (\mathrm{ev}\otimes \mathrm{id}_{M})(x\otimes\sigma(m\otimes\omega )), \quad (a\ov{b})m= amb$$ 
where $(x,\omega)\in\mathfrak{M}$, $a\ov{b}\in A^{e}$ and $m\in M$. The two correspondences described are each others inverses and their functoriality follows easily.  \end{proof}
\subsection{Mutation of $T\mathfrak{X}^{1}_{\sbt}$ for Bimodule Connections}\label{SBXA}
In this section we construct the bialgerboid whose category of left modules recovers the category of left bimodule connections, $\biml$. Any left bimodule connection in $\biml$, is a bimodule with an $\Omega^{1}$-intertwining and a left connections. Hence, every left bimodule connection has an induced $B(\Omega^{1})$-action and a $T\XA^{1}_{\sbt}$-action arising from its $\Omega^{1}$-intertwining and left connection, respectively. The only additional data defining a left bimodule connection, is how its left connection and right $A$-action interact. We define $B\XA^{1}$ to be the quotient of algebra $T(\mathfrak{M}\oplus \XA^{1} \oplus A^{e})$ by the ideal generated by the set of relations (\ref{EqXAOMA}), (\ref{EqXAOMAop}) and 
\begin{align}
 a\sbt x&=ax \label{EqAX}
 \\  x\sbt a &=xa+\mathrm{ev}(x,da)\label{EqXA}
 \\ x\sbt\ov{a}&=\ov{a}\sbt x+ (x,da) \label{EqBXA}
\end{align}
for all $x\in \XA^{1}, \omega\in\Omega^{1}, a\in A$. Equivalently, $B\XA^{1}$ is the quotient of the free product of algebras $B(\Omega^{1})\star T\XA^{1}$ by the ideal which the set of relations (\ref{EqAX}), (\ref{EqXA}) and (\ref{EqBXA}) generate. 

We extend the coproduct and counit of $B(\Omega^{1} )$ to $B\XA^{1}$ by defining it on elements of $\XA^{1}$ and extending them multiplicatively to $B\XA^{1}$, by $\Delta (m\sbt n)= m_{(1)}\sbt n_{(1)}\otimes m_{(2)}\sbt n_{(2)}$ and $\epsilon (m\sbt n) = \epsilon (m \epsilon (n))$:
\begin{equation}
\Delta (x) = x\otimes 1 +  (x,\omega_{i})\otimes x_{i} \hspace{2cm} \epsilon (x)=0
\end{equation}
for $x\in\XA^{1}$. 
\begin{lemma}\label{LBXA} The coproduct $\Delta$ and counit $\epsilon$ are well-defined maps on $B\XA^{1}$ and provide $B\XA^{1}$ with a left $A$-bialgebroid structure.
\end{lemma}
\begin{proof} Since we have defined $\Delta$ and $\epsilon$ on the generators of the algebra and extended them multiplicatively to the rest of the algebra, we must first check if they are well-defined:
\begin{align*} 
\Delta (a\sbt x) &=a\sbt x_{(1)}\otimes x_{(2)}= a\sbt x \otimes 1+ a\sbt (x,\omega_{i})\otimes x_{i}=\Delta (ax)
\\ \Delta (x\sbt a) &=x_{(1)}\sbt a\otimes x_{(2)}= x \sbt a\otimes 1+  (x,\omega_{i})\sbt a\otimes x_{i}
\\& = \Delta(xa)  + \mathrm{ev}(x \otimes da )\otimes 1 = \Delta (xa+\mathrm{ev}(x\otimes da ))
\\ \Delta (\ov{a}\sbt x) &=x_{(1)}\otimes x_{(2)}\sbt \ov{a}= x \otimes \ov{a}+  (x,\omega_{i})\otimes x_{i}\sbt \ov{a}
\\ &=\Delta (\ov{a}\sbt x) -  (x,\omega_{i})\otimes (x_{i},da)= \Delta (\ov{a}\sbt x-(x,da))
\end{align*}
The map $\Delta$ being a bimodule morphism follows from the above calculations. Now, we check that $\epsilon$ is well-defined:
\begin{align*}
\epsilon (a\sbt x)&=\epsilon (a \sbt \epsilon(x))=0=\epsilon(ax)
\\ \epsilon(x\sbt a )&= \epsilon (x \sbt \epsilon(a))=\epsilon (xa +\mathrm{ev}(x\otimes da)) = \epsilon (xa) +\epsilon (\mathrm{ev}(x\otimes da))
\\ \epsilon(x\sbt \ov{a} )&=\epsilon (x \sbt \epsilon(\ov{a}))=\epsilon (x \sbt \epsilon(a))= \mathrm{ev}(x\otimes da)= \epsilon ((x,da))
\\& = \epsilon (\ov{a}\sbt \epsilon (x))+\epsilon ((x,da))=  \epsilon ( \ov{a}\sbt x +(x,da))
\end{align*}
It also follows that $\epsilon$ is a bimodule map. Now we demonstrate coassociativity (\ref{EqDelAss}) and the counit condition (\ref{EqCoun})
\begin{align*}
(\Delta \otimes \mathrm{id}_{B\XA^{1}})\Delta (x)= &x\otimes 1\otimes 1+ (x,\omega_{i})\otimes x_{i}\otimes 1 
\\&+ (x,w_{j})\otimes (x_{j},\omega_{i}) \otimes x_{i} =(\mathrm{id}_{B\XA^{1}}\otimes \Delta )\Delta (x)
\\\epsilon (x_{(1)})x_{(2)} &= \epsilon (x)1+ \epsilon((x,\omega_{i}))x_{i}= x
\\ \ov{\epsilon (x_{(2)})}x_{(1)}&= \ov{\epsilon(1)}x+\ov{\epsilon(x_{i})}(x,\omega_{i})= x
\end{align*}
The other coring axioms are easy to check and are left to the reader. The bialgebroid axioms hold since we defined the coproduct and counit multiplicatively.\end{proof}
\begin{thm}\label{TBXAbiml} There exists an isomorphism of categories $\prescript{}{B\XA^{1}}{\mathcal{M}}\cong \biml$. 
\end{thm}
\begin{proof} By restriction of scalars to $B(\Omega^{1})$ and Theorem \ref{TBOmga}, a $B\XA^{1}$-module $M$ has an $A$-bimodule structure and an induced $\Omega^{1}$-intertwining defined by $\sigma (m\otimes\omega )= \omega_{i}\otimes (x_{i}, \omega)m$. As described in Section \ref{Scnc}, by restriction to $T\XA^{1}_{\sbt}$, the module $M$ has a left connection defined by $\cnc (m)= \sum_{i} \omega_{i}\otimes x_{i}m$. The induced connection $\cnc$ is a left bimodule connection 
\begin{align*}
 \cnc (ma)&=  \omega_{i}\otimes x_{i}(ma)=  \omega_{i}\otimes (x_{i}\sbt \ov{a} )m
\\ &=  \omega_{i}\otimes (\ov{a}\sbt x_{i})m +  \omega_{i}\otimes (x_{i}, da)m
\\ &=  \omega_{i}\otimes (x_{i}m)a+ \sigma (m\otimes da )= \cnc (m) a+ \sigma (m\otimes da )
\end{align*}
for all $a\in A$ and $m\in M$. Functoriality follows easily and the functor in the opposite direction is formed by realising that the induced $T\XA^{1}_{\sbt}$ and $B(\Omega^{1})$ actions for a bimodule connection satisfy relation (\ref{EqBXA}) and induce an action of $B\XA^{1}$.
\end{proof}
\begin{rmk}\label{RSrjBim} When the calculus is surjective, a triple $(M,\cnc,\sigma )$ being a left bimodule connection is a property for a given bimodule $M$ with a left connection $\cnc$ i.e. the $\Omega^{1}$-intertwining $\sigma$ is not additional data and either exists or not. We observe that in this case the generators of the form $\XA^{1}\otimes_{\field}\Omega^{1}$ are made redundant in the definition of $B\XA^{1}$, because $\Omega^{1}$ is spanned by elements of the form $adb$, where $a,b\in A$, and for any $(x,adb)\in \XA^{1}\otimes_{\field}\Omega^{1}$ we have
$$ (x,adb)= \big( x\sbt \ov{b} - \ov{b}\sbt x\big)\sbt \ov{a}= [x,\ov{b}]\sbt \ov{a}$$
Thereby, $B\XA^{1}$ reduces to a quotient of $A^{e}\star T\XA^{1}$ with relations of $T\XA^{1}_{\sbt}$, (\ref{EqAX}), (\ref{EqXA}) and relations arising from $[x,\ov{b}]\sbt \ov{a}$ being regarded as elements of $\XA^{1}\otimes_{\field}\Omega^{1}$. We do this reduction for Example \ref{EBiaM2C}.
\end{rmk}
\subsection{$T\mathfrak{X}^{1}_{\sbt}$ as a Central Commutative Algebra in $\biml$}\label{SCCATX}
In this section we consider the $A$-bimodule structure on $T\XA^{1}_{\sbt}$ which arises from $A$ being a subalgebra of $T\XA^{1}_{\sbt}$. In \cite{beggs2014noncommutative}, $T\XA^{1}_{\sbt}$ is presented with the additional structure of a commutative algebra in the lax center of $\biml$. We briefly recall the definition of the center of a monoidal category from \cite{majid1992braided}.

If $(\mathcal{C},\otimes,1_{\otimes},\alpha ,l, r )$ is a monoidal category as described in Section \ref{SBial}, then the \emph{(lax) center} of $\mathcal{C}$ has pairs $(X,\tau)$ as objects, where $X$ is an object in $\mathcal{C}$ and $\tau : X\otimes -\rightarrow -\otimes X$ is a natural (transformation) isomorphism satisfying
\begin{equation}\label{EqBrai}
\tau_{1_{\otimes}}= l_{X}^{-1}r_{X}, \hspace{1cm}(\mathrm{id}_{M}\otimes \tau_{N})(\tau_{M}\otimes \mathrm{id}_{N})\alpha_{X,M,N}^{-1}=\alpha_{M,N,X}\tau_{M\otimes N}
\end{equation}
and morphisms $f: X\rightarrow Y$ of $\mathcal{C}$ satisfying $(\mathrm{id}_{\mathcal{C}}\otimes f)\tau =\nu (f\otimes \mathrm{id}_{\mathcal{C}})$, as morphism $f:(X,\tau)\rightarrow (Y,\nu )$. 
We denote the lax center and center by $Z^{lax}(\mathcal{C})$ and $Z(\mathcal{C})$, respectively. This construction is often referred to as the \emph{Drinfeld-Majid center}. The lax center is also referred to as the \emph{prebraided} or \emph{weak} center. The (lax) center has a monoidal structure via 
$$ (X,\tau )\otimes (Y, \nu):= (X\otimes Y , (\tau \otimes \mathrm{id}_{Y})(\mathrm{id}_{X}\otimes\nu ) )$$ 
and $(1_{\otimes}, l^{-1}r)$ acting as the monoidal unit, so that the forgetful functor to $\mathcal{C}$ is strong monoidal. 

First we observe that if we restrict the coproduct $\Delta$ to $T\XA^{1}_{\sbt}$, we obtain a map 
$$\ov{\Delta} : T\XA^{1}_{\sbt} \rightarrow |\langle \mathfrak{M}\rangle \otimes T\XA^{1}_{\sbt} \oplus T\XA^{1}_{\sbt} \otimes 1 $$
where $\langle \mathfrak{M}\rangle$ is the ideal generated by elements of $\mathfrak{M}$ in $B\XA^{1}$. Notice that we are abusing notation here and should be writing $T\XA^{1}_{\sbt}$ instead of $1$. However, we do this to emphasise that the image of the map is $1\in T\XA^{1}_{\sbt}$. 

For any bimodule $M$, we can restrict $\Delta_{M,A}$, as described in (\ref{EqDelMN}), to $T\XA^{1}_{\sbt}$:
$$\ov{\Delta}_{M} : T\XA^{1}_{\sbt} \otimes M \rightarrow (\langle \mathfrak{M}\rangle\boxtimes M )\otimes T\XA^{1}_{\sbt} \oplus (T\XA^{1}_{\sbt} \otimes M)\otimes 1 $$
Observe that $\ov{\Delta}_{M}$ is in fact an $A$-bimodule morphism. This is because $(B\XA^{1}\boxtimes M)\otimes (B\XA^{1}\boxtimes A)$ is the image of $\Delta_{M,A}$ and $B\XA^{1}\boxtimes A = B\XA^{1}/ \lbrace b\sbt a =b\sbt \ov{a}\mid b\in B\XA^{1},\ a\in A\rbrace$. Therefore, for any $b\in B\XA^{1}$ and $m\in M$ 
$$\Delta_{M,A} (b\boxtimes ma) = \Delta_{M,A} (b\sbt\ov{a}\boxtimes m)= (b_{(1)} \boxtimes m)\otimes b_{(2)}\sbt\ov{a}= (b_{(1)} \boxtimes m)\otimes b_{(2)}\sbt a$$ 
holds and $\ov{\Delta}_{M}$ is an $A$-bimodule morphism.

Consequently, for any $B\XA^{1}$-module $(M,\triangleright :B\XA^{1}\boxtimes M\rightarrow M)$, the composition 
$$\xymatrix@C+3.9pt{\lambda_{M}:T\XA^{1}_{\sbt} \otimes M \ar[r]^{\hspace{0.3cm}\ov{\Delta}_{M}\hspace{2cm}}&(\langle \mathfrak{M}\rangle\boxtimes M )\otimes T\XA^{1}_{\sbt} \oplus (T\XA^{1}_{\sbt} \otimes M)\otimes 1 \ar[r]^{\hspace{2.2cm}\triangleright \otimes \mathrm{id}_{T\XA^{1}_{\sbt}}} &M\otimes T\XA^{1}_{\sbt}}
$$
is an $A$-bimodule map. Recall that the algebra $T\XA^{1}_{\sbt}$ has a natural $A$-bimodule structure due to $A$ being its subalgebra, which makes $T\XA^{1}_{\sbt}$ a left $A^{e}$-module. We can extend this left $A^{e}$-action on $T\XA^{1}_{\sbt}$ to a left $B\XA^{1}$-module structure, where the elements of $T\XA^{1}_{\sbt}$ act by the multiplication of the algebra, and the action of the ideal $\langle \mathfrak{M}\rangle$ is zero. Equivalently, as a left bimodule connection we obtain the triple $(T\XA^{1}_{\sbt} , \sum_{i}\omega_{i} \otimes x_{i}\sbt -, 0)$. Consequently, $\lambda_{M}$ becomes a morphisms of bimodule connections i.e. $\lambda_{M}$ respects the $B\XA^{1}$-action since the coproduct respects multiplication by (\ref{EqDelMul}). Furthermore, for any morphism of left bimodule connections $f:M\rightarrow N$, the right square below commutes 
$$\xymatrix@C+8pt{T\XA^{1}_{\sbt} \otimes M \ar[d]^{\mathrm{id}_{T\XA^{1}_{\sbt}}\otimes f}\ar[r]^{\ov{\Delta}_{M}\hspace{2cm}}&(\langle \mathfrak{M}\rangle\boxtimes M )\otimes T\XA^{1}_{\sbt} \oplus (T\XA^{1}_{\sbt} \otimes M)\otimes 1 \ar[d]^{(\mathrm{id}_{B\XA^{1}}\boxtimes f)\otimes \mathrm{id}_{T\XA^{1}_{\sbt}}}\ar[r]^{\hspace{2.1cm}\triangleright \otimes \mathrm{id}_{T\XA^{1}_{\sbt}}} &M\otimes T\XA^{1}_{\sbt}\ar[d]^{f\otimes \mathrm{id}_{T\XA^{1}_{\sbt}}}
\\ T\XA^{1}_{\sbt} \otimes N \ar[r]^{\ov{\Delta}_{N}\hspace{2cm}}&(\langle \mathfrak{M}\rangle\boxtimes N )\otimes T\XA^{1}_{\sbt} \oplus (T\XA^{1}_{\sbt} \otimes N)\otimes 1 \ar[r]^{\hspace{2.1cm}\triangleright \otimes \mathrm{id}_{T\XA^{1}_{\sbt}}} &N\otimes T\XA^{1}_{\sbt}} $$
and thereby $\lambda_{N}(\mathrm{id}_{T\XA^{1}_{\sbt}}\otimes f)=(f\otimes \mathrm{id}_{T\XA^{1}_{\sbt}})\lambda_{M}$. This implies that 
$$\lambda : T\XA^{1}_{\sbt} \otimes \mathrm{id}_{\biml} \rightarrow \mathrm{id}_{\biml} \otimes T\XA^{1}_{\sbt} $$
is a natural transformation. It follows directly from the definition of $\ov{\Delta}_{M}$, the coassociativity of $\Delta$, (\ref{EqDelAss}), and the counit condition, (\ref{EqCoun}), that $\lambda $ satisfies the braiding conditions (\ref{EqBrai}). 
\begin{thm}\hspace{0.001cm}[Theorem 8.2 \cite{beggs2014noncommutative}] The triple $(T\XA^{1}_{\sbt} ,\sum_{i} \omega_{i} \otimes x_{i}\sbt -, 0)$ along with braiding $\lambda $ becomes an object in the lax center $Z^{lax}\left(\biml \right)$. 
\end{thm}
The braiding presented for the left bimodule connection $(T\XA^{1}_{\sbt} , \sum_{i}\omega_{i} \otimes x_{i}\sbt -, 0)$ in \cite{beggs2014noncommutative}, coincides with our definition of $\lambda$ on the elements of $\XA^{1}$ and $A$, and is extended iteratively for their basis of $T\XA^{1}_{\sbt}$ and ultimately gives the same braiding. Additionally, in \cite{beggs2014noncommutative}, $T\XA^{1}_{\sbt}$ forms a commutative algebra with the braiding $\lambda$ i.e. $\sbt (\lambda_{T\XA^{1}_{\sbt}} )=\sbt$. This follows from the image of $\ov{\Delta}_{M}$ on the right component being the identity i.e. the diagram
$$\xymatrix@C+1.5pt{T\XA^{1}_{\sbt}\otimes M\ar[dr]_-{\ov{\Delta}_{M}}\ar[rr]^{\mathrm{id}_{T\XA^{1}_{\sbt} \otimes M}\otimes 1} & & T\XA^{1}_{\sbt} \otimes M \otimes 1
\\ &(\langle \mathfrak{M}\rangle\boxtimes M )\otimes T\XA^{1}_{\sbt} \oplus (T\XA^{1}_{\sbt} \otimes M)\otimes 1\ar[ur]_-{\hspace{0.8cm} 0 \oplus \mathrm{id}_{T\XA^{1}_{\sbt} \otimes M \otimes 1} }&} $$
commutes. When $M=T\XA^{1}_{\sbt}$, the action of $\langle \mathfrak{M}\rangle$ on $T\XA^{1}_{\sbt}$ is zero and 
$$\xymatrix@C+3pt{T\XA^{1}_{\sbt}\otimes T\XA^{1}_{\sbt}\ar[dr]_-{\ov{\Delta}_{ T\XA^{1}_{\sbt}}}\ar[r]^{\sbt}  & T\XA^{1}_{\sbt} &T\XA^{1}_{\sbt}\otimes 1\ar[l]_{\sbt}
\\ &(\langle \mathfrak{M}\rangle\boxtimes T\XA^{1}_{\sbt} )\otimes T\XA^{1}_{\sbt} \oplus (T\XA^{1}_{\sbt} \otimes T\XA^{1}_{\sbt})\otimes 1\ar[ur]_-{\hspace{0.8cm}0 \otimes \mathrm{id}_{T\XA^{1}} \oplus  \sbt \otimes \mathrm{id}_{T\XA^{1}}}& } $$
commutes. 

The author would like to point out that although the above description answers why $T\XA^{1}_{\sbt}$ appears as a commutative algebra in the lax center of $\biml$ and provides a framework for the work presented in \cite{beggs2014noncommutative}, it does not seem to relate to previous work on bialgebroids. As demonstrated in \cite{bruguieres2011hopf}, central commutative algebras should be viewed equivalent to Hopf comonads. However, the resulting comonad is not a part of the picture below.
$$\xymatrix{\prescript{}{B(\Omega^{1})}{\mathcal{M}}\cong\prescript{}{A}{\mathcal{M}}^{\Omega^{1}}_{A}\ar@<1ex>[r] &\bim\ar[l]\ar[dl] \\ \prescript{}{B\XA^{1}}{\mathcal{M}}\cong\biml\ar[u]\ar@<1ex>[ur]&} $$
The forgetful functor $\biml\rightarrow\prescript{}{A}{\mathcal{M}}^{\Omega^{1}}_{A}$ does not appear to have a left adjoint. In other words, $B\XA^{1}$ does not arise as the composition of two bimonads as defined in \cite{bruguieres2011hopf}. It is also not an extension by a central commutative algebra, as described in Section 3.4.7 of \cite{BOHM2009173}, since $T\XA^{1}_{\sbt}$ is not a commutative algebra in the center of $ \prescript{}{A}{\mathcal{M}}^{\Omega^{1}}_{A}$.
\subsection{Examples of Bialgebroids}\label{SExBia}
Now we present several examples of left bialgebroids by generators and relations, arising from the differential calculi presented in Section \ref{SNCG}. In the examples below we will not repeat how the coproduct and counit are defined on elements of $A^{e}$ in $B\XA^{1}$, since they follow from the bialgebroid axioms.
\begin{ex}\label{EBiaDeriv}[Derivation Calculus] Recall that for any derivation $d$ on an algebra $A$, we regard $\Omega^{1}=A$ as a bimodule, so that $\XA^{1}=A$, where the evaluation morphism is given by multiplication and the coevaluation morphism is given by $\mathrm{coev}(1)=1\otimes 1$. It is easy to see that $T\XA^{1}_{\sbt}$ is isomorphic to
$$T\XA^{1}_\sbt = A\star\field [D]/ \langle D\sbt a =a\sbt D +da \mid a\in A \rangle $$
where $D=1\in\XA^{1}$. In this case we say the algebra factorizes as $A.\field [D]$ under the commuting relations $D\sbt a =a\sbt D +da$, for $a\in A$. The bialgebroid $B\XA^{1}$ has the additional generator $F=1\otimes_{\field} 1\in \XA^{1}\otimes_{\field} \Omega^{1}$ and factorises as $A^{e}.\field\langle D,F\rangle$ with the commutation relations 
$$[D, a]=da, \hspace{1cm} [D, \ov{a}]=\ov{da} \sbt F, \hspace{1cm}[F , a]= [F , \ov{a}]=0 $$
where $a\in A$. The coproduct and counit are are defined on the generators by $\Delta (D)=D\otimes 1+F\otimes D$, $\Delta (F)=F\otimes F$ with $\epsilon (D)=0$ and $\epsilon (F)=1$. 
\end{ex}
\begin{ex}\label{EBiaGrph} [Finite Quivers] Example \ref{EGrph} provided a setting for differential geometry on a finite quiver $\Gamma =(V,E)$, with $A=\field (V)$ and $\Omega^{1}=\oplus_{e\in E}\field\ovr{e}$. Consequently, $\XA^{1} = \oplus_{e\in E}\field\ovl{e}$ where $f_{p}\ovl{e}f_{q}= \delta_{p,t(e)}\delta_{q,s(e)}\ovl{e}$. In this case $T\XA^{1}_{\sbt} =\field \langle f_{p},\ovl{e}\mid p\in V, e\in E\rangle / \mathcal{U}$ where $\mathcal{U}$ is the ideal generated by relations 
\begin{align}
f_{p}\sbt f_{q}&= \delta_{p,q}f_{q} ,\quad\quad f_{p}\sbt \ovl{e} =\delta_{p,t(e)} \ovl{e}
\\  \ovl{e}\sbt f_{p} &= \delta_{p,s(e)} [\ovl{e} -f_{t(e)}] + \delta_{p,t(e)} f_{t(e)} \label{EqGrphTX}
\end{align} 
for all $e\in E$ and $p,q\in V$. In Lemma 4.1 of \cite{majid2018generalised}, it was pointed out that a left connection over this calculus corresponds to a \emph{quiver representation} in the classical sense \cite{auslander1997representation}. We can explain this by observing that the \emph{quiver path algebra} $\field\Gamma$, whose module category recovers the category of quiver representations, is isomorphic to $T\XA^{1}_{\sbt}$. The quiver algebra $\field\Gamma$ has the same generators, however it has $\ovl{e}\sbt f_{p} = \delta_{p,s(e)} \ovl{e}$ as a relation instead of (\ref{EqGrphTX}). There exists an isomorphism of algebras $\field\Gamma \rightarrow T\XA^{1}_{\sbt}$ define by 
$$f_{p}\longmapsto f_{p},\quad \quad \ovl{e}\longmapsto  \ovl{e}-f_{t(e)}$$
Hence, the bialgebroid $B\XA^{1}$ is the quotient of $\field\Gamma\left\langle \ov{f_{p}},\ (\ovl{e_{1}},\ovr{e_{2}})\mid p\in S, \ e_{1},e_{2}\in E \right\rangle$ by the additional relations 
\begin{align*}
\ov{f_{p}}\sbt \ov{f_{q}}= \delta_{p,q}\ov{f_{q}} , \quad &f_{p}\sbt \ov{f_{q}}=\ov{f_{q}}\sbt f_{p}
\\ (\ovl{e_{1}},\ovr{e_{2}})\sbt f_{p}\sbt \ov{f_{q}} = &(\ovl{e_{1}},\ovr{e_{2}})\delta_{p,s(e_{1})}\delta_{q,s(e_{2})}
\\ f_{p}\sbt \ov{f_{q}}\sbt(\ovl{e_{1}},\ovr{e_{2}}) =& (\ovl{e_{1}},\ovr{e_{2}})\delta_{p,t(e_{1})}\delta_{q,t(e_{2})}
\\ \ovl{e_{1}}\sbt \ov{f_{q}}= \ov{f_{q}}\sbt \ovl{e_{1}} +\sum_{e\in E,\  t(e)=q} & (\ovl{e_{1}},\ovr{e})- \sum_{e\in E,\   s(e)=q}  (\ovl{e_{1}},\ovr{e})
\end{align*}
and the coproduct and counit  are defined by 
\begin{align*}
\Delta ((\ovl{e_{1}},\ovr{e_{2}}))= \sum_{e\in E} (\ovl{e_{1}},\ovr{e})\otimes (\ovl{e},\ovr{e_{2}}), \quad\epsilon ((\ovl{e_{1}},\ovr{e_{2}}))= \delta_{e_{1},e_{2}} f_{t(e_{1})}
\\ \Delta (\ovl{e_{1}}) = \ovl{e_{1}}\otimes 1 + \sum_{e\in E} (\ovl{e_{1}},\ovr{e})\otimes (\ovl{e}+f_{t(e)}), \quad \epsilon (\ovl{e_{1}}) =-f_{t(e_{1})} 
\end{align*}
for all $e_{1},e_{2}\in E$ and $p,q\in V$. 
\end{ex}
\begin{ex}\label{EBiaM2C}[$M_{2}(\mathbb{C})$] For the calculus of Example \ref{EM2C}, we denote elements $1\oplus 0$ and $0\oplus 1$ in $\Omega^{1}$ by $\mathsf{s}$ and $\mathsf{t}$, respectively. Hence, $\XA^{1}$ is a free bimodule with $f_{\mathsf{s}}$ and $f_{\mathsf{t}}$ as the dual basis to $\mathsf{s}$ and $\mathsf{t}$. The algebra $T\XA^{1}_{\sbt}$ was described in Chapter 6 of \cite{beggs2019quantum}, and factorises as $A.\mathbb{C}\langle f_{\mathsf{s}},f_{\mathsf{t}}\rangle$ with commutation relations 
$$f_{\mathsf{s}}\sbt a = a \sbt f_{\mathsf{s}}+ [E_{12},a] ,\hspace{0.6cm} f_{\mathsf{t}}\sbt a = a \sbt f_{\mathsf{t}}+ [E_{21},a]$$
The bialgebroid $B\XA^{1}$ factorises as $A^{e}.\mathbb{C}\langle f_{i},\omx{i}{\gamma}{j}\mid i,j\in \lbrace \mathsf{s},\mathsf{t}\rbrace\rangle$ with additional relations 
\begin{align*}
f_{i} \sbt \ov{a} = \ov{a} \sbt f_{i} + &\ov{[E_{12},a]}\omx{i}{\gamma}{s}+\ov{[E_{21},a]}\omx{i}{\gamma}{t}, \quad [\omx{i}{\gamma}{j}, a\ov{b} ]=0
\end{align*}
for $i,j\in \lbrace \mathsf{s},\mathsf{t}\rbrace$. The coproduct and counit are defined by 
\begin{align*}
\Delta (f_{i})&=f_{i}\otimes 1 + \omx{i}{\gamma}{\mathsf{s}}\otimes f_{\mathsf{s}}+ \omx{i}{\gamma}{\mathsf{t}}\otimes f_{\mathsf{t}} \hspace{1cm} \epsilon (f_{i})=0  
\\ \Delta (\omx{i}{\gamma}{j})&= \omx{i}{\gamma}{\mathsf{s}}\otimes\omx{\mathsf{s}}{\gamma}{j}+ \omx{i}{\gamma}{\mathsf{t}}\otimes\omx{\mathsf{t}}{\gamma}{j}\hspace{1.8cm} \epsilon(\omx{i}{\gamma}{j})= \delta_{i,j}
\end{align*}
for $i,j\in \lbrace \mathsf{s},\mathsf{t}\rbrace$. The calculus in this case is surjective with $\mathsf{s}=  (dE_{21})E$ and $\mathsf{t}=  (dE_{12})E$, where $E=E_{11}-E_{22}$. By Remark \ref{RSrjBim} generators of the form $\omx{i}{\gamma}{j}$ become redundant: 
$$\omx{i}{\gamma}{\mathsf{s}}= \ov{E}\sbt [f_{i},\ov{E_{21}} ],\quad  \omx{i}{\gamma}{\mathsf{t}}= -\ov{E}\sbt [f_{i},\ov{E_{12} }]$$
where $i\in \lbrace \mathsf{s},\mathsf{t}\rbrace$. Thereby, $B\XA^{1}$ factorizes as $A^{e}.\mathbb{C}\langle f_{\mathsf{s}},f_{\mathsf{t}}\rangle$ with the $T\XA^{1}_{\sbt}$ relations as above and the additional relations 
$$[f_{i},\ov{a}]= \ov{E[E_{12},a]} \sbt [f_{i},\ov{E_{21}} ]+\ov{E[E_{21},a]} \sbt [f_{i},\ov{E_{12} }] $$ 
for all $i\in \lbrace \mathsf{s},\mathsf{t}\rbrace$ and $a\in A$.
\end{ex}
\begin{ex}\label{EBiaHopf}[Hopf Bimodules] If $(A,\delta ,\nu ,s)$ is a Hopf algebra and $\Omega^{1} = \Lambda \otimes_{\field} A$ a Hopf bimodule, then $\Omega^{1}$ being right fgp is equivalent to $\Lambda$ being a finite dimensional vectorspace, with basis $\lbrace\lambda_{i}\rbrace_{i=1}^{n}$. Hence, $\XA^{1}\cong A\otimes_{\field} \Lambda^{\star}$ is free as a left module, where $\Lambda^{\star}$ is the dual vectorspace to $\Lambda$ with dual basis $\lbrace f_{i}\rbrace_{i=1}^{n}$. Here, $\Lambda^{\star}$ has an induced right $A$-action corresponding to the left $A$-action of $\Lambda$ defined by $f\triangleleft a =f(a\triangleright -)$ for all $f\in\Lambda^{\star}$ and $a\in A$. In this case, $T\XA^{1}_{\sbt}$ was described in Chapter 6 of \cite{beggs2019quantum} and factorises as $A.T\Lambda^{\star}$ with commutation relation 
$$f_{i}\sbt a =a_{(2)}\sbt f_{i}\triangleleft a_{(1)} +\partial^{i}(a) $$ 
where $\partial^{i}(a)=\mathrm{ev}(1\otimes_{\field} f_{i}\otimes da)$. The $A^{e}$-bimodule $\XA^{1}\otimes_{\field}\Omega^{1}$ is free as a left $A^{e}$-module and isomorphic to $ A^{e}\otimes_{\field} (\Lambda^{\star}\otimes_{\field} \Lambda)$. We denote the basis of $\Lambda^{\star}\otimes_{\field} \Lambda$ by $(f_{i},\lambda_{j})$. Hence, the bialgebroid $B\XA^{1}$ factorizes as $A.T\mathfrak{L}$ where $\mathfrak{L}=\Lambda^{\star}\oplus (\Lambda^{\star}\otimes_{\field} \Lambda)$, with additional commutation relations
\begin{align*}
(f_{i},\lambda_{j})\sbt a\ov{b} = & a_{(2)}\ov{b_{(2)}}\sbt(f_{i}\triangleleft a_{(1)} , b_{(1)}\triangleright \lambda_{j} )
\\ [f_{i}, \ov{a}] &= \sum_{j=1}^{n}\ov{\partial^{j}(a)} \sbt (f_{i}, \lambda_{j})
\end{align*}
for all $1\leq i,j\leq n$. The coproduct and counit are given by 
\begin{align*}
\Delta (f_{i})&=f_{i}\otimes 1 + \sum_{j=1}^{n}(f_{i},\lambda_{j})\otimes f_{j},  \hspace{1cm}\epsilon (f_{i})=0
\\ \Delta ((f_{i},\lambda_{j}))&= \sum_{k=1}^{n}(f_{i},\lambda_{k})\otimes (f_{k},\lambda_{j}),\hspace{1cm} \epsilon((f_{i},\lambda_{j}))=\delta_{i,j} 
\end{align*}
for all $1\leq i,j\leq n$.
\end{ex}
\begin{ex}\label{EBiaD6}[$\mathbb{C}D_{6}$] Let $D_{6}$ denote the Dihedral group with 6 elements with presentation $\langle a,b\mid a^{3}=b^{2}=1,\ a^{2}b=ba\rangle$ and $\Lambda$ its 2-dimensional irreducible complex representation with basis $\xi ,\tau$ defined by  
$$ a\triangleright \xi = \frac{1}{2} (\xi +\sqrt{3}\tau),\quad  b\triangleright \xi =\xi, \quad a\triangleright \tau = \frac{1}{2} (-\sqrt{3}\xi +\tau),\quad b\triangleright \tau =-\tau $$
Recall from Example \ref{EGrpAlg} that we obtain an inner calculus on $A=\mathbb{C}D_{6}$, by taking $\theta=\xi +\tau$, so that $d:\mathbb{C}D_{6}\rightarrow \Omega^{1}$ satisfies 
$$ d(a)= \frac{1}{2} [ -(1+\sqrt{3})\xi +(\sqrt{3}-1)\tau]\otimes_{\mathbb{C}} a ,\quad d(b)=-2\tau \otimes_{\mathbb{C}} b $$
Consequently, $\Lambda^{*}$ has a dual basis to $\Lambda$, denoted by $f_{\xi},f_{\tau}$ and $T\XA^{1}_{\sbt}$ factorizes as $A.\mathbb{C}\langle f_{\xi},f_{\tau}\rangle$ with commutation relations 
\begin{align*}
f_{\xi}\sbt a& = \frac{1}{2} a \sbt (f_{\xi} -\sqrt{3}f_{\tau}) -\frac{1}{2}(1+\sqrt{3})a, \hspace{0.6cm}f_{\xi}\sbt b= b\sbt f_{\xi} 
\\ f_{\tau}\sbt a &= \frac{1}{2} a \sbt (\sqrt{3}f_{\xi} +f_{\tau}) +\frac{1}{2}(\sqrt{3}-1)a,\hspace{0.6cm} f_{\tau}\sbt b= -b\sbt f_{\tau} -2b
\end{align*}
The resulting bialgebroid $B\XA^{1}$ factorises as $A^{e}.\mathbb{C}\langle f_{\xi},f_{\tau}, \omx{\xi}{\gamma}{\xi},\omx{\tau}{\gamma}{\xi},\omx{\xi}{\gamma}{\tau},\omx{\tau}{\gamma}{\tau}\rangle$ with additional relations
\begin{align*}
[f_{i},\ov{a}] &= -\frac{1}{2}(1+\sqrt{3})\ov{a}\sbt\omx{i}{\gamma}{\xi}+\frac{1}{2}(\sqrt{3}-1)a\sbt\omx{i}{\gamma}{\tau} ,\quad [f_{i},\ov{b}]=-2\ov{b}\sbt \omx{i}{\gamma}{\tau}
\\ &\omx{\xi}{\gamma}{i}\sbt a = \frac{1}{2} (\omx{\xi}{\gamma}{i} -\sqrt{3}\omx{\tau}{\gamma}{i}), \quad \omx{\tau}{\gamma}{i}\sbt a = \frac{1}{2} (\sqrt{3}\omx{\xi}{\gamma}{i} +\omx{\tau}{\gamma}{i}) 
\\ &\omx{i}{\gamma}{\xi}\sbt \ov{a} = \frac{1}{2} (\omx{i}{\gamma}{\xi} +\sqrt{3}\omx{i}{\gamma}{\tau}), \quad \omx{i}{\gamma}{\tau}\sbt \ov{a} = \frac{1}{2} (-\sqrt{3}\omx{i}{\gamma}{\xi} +\omx{i}{\gamma}{\tau}) 
\\ \omx{\xi}{\gamma}{i}\sbt b &=\omx{\xi}{\gamma}{i}, \quad \omx{\tau}{\gamma}{i}\sbt b =-\omx{\tau}{\gamma}{i} ,\quad  \omx{i}{\gamma}{\xi}\sbt \ov{b}= \omx{i}{\gamma}{\xi},\quad  \omx{i}{\gamma}{\tau}\sbt \ov{b}= -\omx{i}{\gamma}{\tau}
\end{align*}
for $i\in \lbrace \xi,\tau\rbrace$. The coproduct and counit take the form of 
\begin{align*}
\Delta (f_{i})&=f_{i}\otimes 1 + \omx{i}{\gamma}{\xi}\otimes f_{\xi}+ \omx{i}{\gamma}{\tau}\otimes f_{\tau} \hspace{1cm} \epsilon (f_{i})=0  
\\ \Delta (\omx{i}{\gamma}{j})&= \omx{i}{\gamma}{\xi}\otimes\omx{\xi}{\gamma}{j}+ \omx{i}{\gamma}{\tau}\otimes\omx{\tau}{\gamma}{j}\hspace{1.8cm} \epsilon(\omx{i}{\gamma}{j})= \delta_{i,j}
\end{align*}
for $i,j\in \lbrace \xi,\tau\rbrace$.
\end{ex}
\section{Hopf Algebroids for Pivotal Calculi}
We would like the monoidal category of connections which we consider to lift the closed monoidal structure of $\bim$. In this a situation, if a bimodule with such a connection is right (left) fgp, its dual bimodule $\pr{M}$ (resp. $M^{\vee}$) will have an induced connection making it left (right) dual to the original connection in this monoidal category of connections. In Section 3.4.2 of \cite{beggs2019quantum}, several statements are presented, demonstrating that if $M$ is a right (left) fgp bimodule with a left (right) bimodule connection $(M,\cnc ,\sigma)$ such that $\sigma$ is invertible, then $\pr{M}$ (resp. $M^{\vee}$) has a compatible right (left) bimodule connection structure. The subcategory of \emph{invertible bimodule connections}, with left bimodule connections with invertible $\Omega^{1}$-intertwinings as objects, is hence considered as a nicer category to work with. In particular, left and right bimodule connections with invertible intertwining morphisms coincide. However, the category of invertible bimodule connections is not closed: given a right fgp bimodule $M$ with an invertible left bimodule connection $(M,\cnc ,\sigma)$, its left dual bimodule $\pr{M}$ will have a right bimodule connection structure denoted by $(\pr{M},\pr{\cnc} ,\sigma^{\sharp})$, but the $\Omega^{1}$-intertwining $\sigma^{\sharp}$ is not necessarily invertible. In fact, there is a natural way of defining connections on inner homs of invertible bimodule connections, but to obtain the correct closed monoidal category lifting the structure of $\bim$, we must find a subcategory of $\prescript{}{A}{\mathcal{M}}^{\Omega^{1}}_{A}$ which lifts the closed structure of $\bim$.
\subsection{Invertible Bimodule Connections}\label{SInv}
To agree with \cite{beggs2019quantum}, we denote the category of invertible bimodule connections i.e. the subcategory of $\biml$, where objects $(M,\cnc ,\sigma)$ have invertible $\Omega^{1}$-intertwinings $\sigma$, by $\Ibim $. Furthermore, we denote the subcategory of $\bim^{\Omega^{1}}$ of bimodules with invertible $\Omega^{1}$-intertwinings by $\prescript{}{A}{\mathcal{IM}}_{A}^{\Omega^{1}}$. It should be clear that $\prescript{}{A}{\mathcal{IM}}_{A}^{\Omega^{1}}$ is a monoidal subcategory of $\prescript{}{A}{\mathcal{M}}^{\Omega^{1}}_{A}$. 
\begin{lemma}\label{LSigInv} An object of $\prescript{}{A}{\mathcal{M}}^{\Omega^{1}}_{A}$, $(M,\sigma )$ has a (right) left dual, if and only if $M$ is (left) right fgp and $\sigma$ is invertible. 
\end{lemma}
\begin{proof} First, observe that since the forgetful functor from $\prescript{}{A}{\mathcal{M}}^{\Omega^{1}}_{A}$ to $\bim$ is strong monoidal, if $(N,\tau)$ is a left dual of $(M,\sigma)$, then $N\cong \pr{M}$ and $M$ is right fgp. Furthermore, the evaluation and coevaluation morphisms $\mathrm{ev}$ and $\mathrm{coev}$ must commute with the intertwining maps i.e. 
\begin{align*}
 \mathrm{ev}\otimes \mathrm{id}_{\Omega^{1}} &= (\mathrm{id}_{\Omega^{1}}\otimes \mathrm{ev})(\tau\otimes \mathrm{id}_{M})(\mathrm{id}_{N}\otimes\sigma )
\\ \mathrm{id}_{\Omega^{1}} \otimes \mathrm{coev}&=(\tau\otimes \mathrm{id}_{M})(\mathrm{id}_{N}\otimes\sigma )(\mathrm{coev} \otimes \mathrm{id}_{\Omega^{1}})
\end{align*}
From the above equations, it is easy to check that the morphism $(\mathrm{id}_{M\otimes\Omega^{1}}\otimes \mathrm{ev}) (\mathrm{id}_{M}\otimes \tau \otimes \mathrm{id}_{M}) (\mathrm{coev} \otimes \mathrm{id}_{\Omega^{1}\otimes M})$ becomes the inverse of $\sigma$. Conversely, if $M$ is right fgp with left dual $\pr{M}$ and $\sigma$ is invertible, we define $(\pr{M},\sigma^{\sharp} )$ by
\begin{equation}
\sigma^{\sharp}= (\mathrm{ev} \otimes \mathrm{id}_{\Omega^{1}\otimes \pr{M}})(\mathrm{id}_{\pr{M}}\otimes \sigma^{-1} \otimes \mathrm{id}_{\pr{M}})(\mathrm{id}_{\pr{M}\otimes\Omega^{1}}\otimes \mathrm{coev})
\end{equation}
so that $(\pr{M},\sigma^{\sharp} )$ is left dual to $(M,\sigma)$ in $\prescript{}{A}{\mathcal{M}}^{\Omega^{1}}_{A}$, via $\mathrm{coev}$ and $\mathrm{ev}$.\end{proof}
For $\prescript{}{A}{\mathcal{IM}}_{A}^{\Omega^{1}}$ to be representable, we need the additional requirement for $\Omega^{1}$ to be left fgp as well as right fgp, with its right dual bimodule denoted by $\mathfrak{Y}^{1} $. Let $\underline{\mathrm{coev}}:A\rightarrow \mathfrak{Y}^{1}  \otimes\Omega^{1}$ and $\underline{\mathrm{ev}}: \Omega^{1}\otimes \mathfrak{Y}^{1} \rightarrow A $ denote the respective coevaluation and evaluation maps and denote $\underline{\mathrm{coev}}(1)= \sum_{j} y_{j}\otimes \rho_{j}$. Parallel to Section \ref{SInter}, we consider $\Omega^{1}\otimes_{\field}\mathfrak{Y}^{1} $ as an $A^{e}$-bimodule via (\ref{EqYFrak}), so that $T_{A^{e}}(\Omega^{1}\otimes_{\field}\mathfrak{Y}^{1} )$-modules have the structure of $A$-bimodules $M$ with a bimodule map $\Omega^{1}\otimes M\rightarrow M\otimes\Omega^{1}$. We denote this category by $\prescript{\Omega^{1}}{A}{\ct{M}}_{A}$ and observe that $\prescript{}{T_{A^{e}}(\Omega^{1}\otimes_{\field}\mathfrak{Y}^{1} )}{\mathcal{M}}\cong \prescript{\Omega^{1}}{A}{\mathcal{M}}_{A}$. This can be proved in a completely symmetric manner to the arguments in Section \ref{SInter}. Consequently, the bialgebroid whose module category is isomorphic to $\prescript{}{A}{\mathcal{IM}}_{A}^{\Omega^{1}}$ is a quotient of the free product of algebras $B(\Omega^{1} )$ and $T_{A^{e}}(\Omega^{1}\otimes_{\field}\mathfrak{Y}^{1} )$ by an ideal which imposes the induced intertwinings with $\Omega^{1}$ to be inverses. 

For a bimodule $M$, when necessary we distinguish bimodule morphisms $\Omega^{1}\otimes M\rightarrow M\otimes\Omega^{1}$ and $M\otimes \Omega^{1}\rightarrow \Omega^{1}\otimes M$ by referring to them by \emph{left} and \emph{right} $\Omega^{1}$-intertwinings. Otherwise, we refer to both morphisms as $\Omega^{1}$-intertwinings and the domain and codomain of morphisms will be clear from context. 

Let $\mathfrak{Z}:=(\XA^{1}\otimes_{\field}\Omega^{1})\oplus (\Omega^{1}\otimes_{\field}\mathfrak{Y}^{1} )$ as a vectorspace and $R:=T_{A^{e}}\mathfrak{Z}$ as an algebra, where the $A^{e}$ bimodule structure of $\mathfrak{Z}$ is defined as follows
\begin{align}
a\ov{a'}(x,\omega) b\ov{b'}&= (axb, b'\omega a')\label{EqMfrak}
\\ a\ov{a'}(\rho,y) b\ov{b'}&= (a\rho b, b'y a')\label{EqYFrak}
\end{align}
where $a,a',b,b'\in A$, $(x,\omega )\in \XA^{1}\otimes_{\field}\Omega^{1}$ and $(\rho ,y)\in\Omega^{1}\otimes_{\field}\mathfrak{Y}^{1} $. It is easy to check that the bialgebroid structures of $T_{A^{e}}(\XA^{1}\otimes_{\field}\Omega^{1})$ and its symmetric counterpart $T_{A^{e}}(\Omega^{1}\otimes_{\field}\mathfrak{Y}^{1} )$, lift to $R$ multiplicatively. Alternatively, we can view $R$ as the free product of $A^{e}$-algebras $T_{A^{e}}(\XA^{1}\otimes_{\field}\Omega^{1})$ and $T_{A^{e}}(\Omega^{1}\otimes_{\field}\mathfrak{Y}^{1} )$. From this point of view, it is easy to see that the free product of two $A^{e}$-algebras with $A$-bialgebroid structures will have a natural $A$-bialgebroid structure: modules over the free product algebra are simply $A$-bimodules with actions from both algebras and the tensor of two such bimodules over $A$ will have an induced action from both biaglebroids, which induces an action of the free product algebra. Ultimately, the coproduct and counit induced on the free product algebra, from the categorical point of view, extend the coproduct and counit of each bialgebroid to the free product algebra, multiplicatively. 

We define $IB(\Omega^{1} )$ as the quotient of algebra $R$ by the set of relations 
\begin{align}
  (\omega_{i},y)\sbt (x_{i},\omega )&=\ov{\underline{\mathrm{ev}}(\omega\otimes y)} \label{EqRelInv1}
\\   (x,\rho_{j})\sbt(\omega, y_{j})&=\mathrm{ev}(x\otimes\omega)\label{EqRelInv2}
\end{align}
for any $x\in X , \omega\in\Omega^{1} , y\in\mathfrak{Y}^{1} $. 
\begin{lemma}\label{LInvBialg} The bialgebroid structure of $R$ descends to a well defined bialgebroid structure on $IB(\Omega^{1} )$. 
\end{lemma}
\begin{proof} Since the bialgebroid structure on $R$ is defined by multiplicatively, we only need to check that the comultiplication and counit are well defined on its quotient $IB(\Omega^{1})$. To do this we look at the relations generating the ideal quotiented from $R$. For relation (\ref{EqRelInv1}) we demonstrate this by the following calculations
\begin{align*}
\Delta ( (\omega_{i},y)&\sbt (x_{i},\omega ))= (\omega_{i},y_{j})\sbt (x_{i},\omega_{k} )\otimes (\rho_{j},y)\sbt (x_{k},\omega )
\\ &=\ov{\underline{\mathrm{ev}}(\omega_{k}\otimes y_{j})} \otimes (\rho_{j},y)\sbt (x_{k},\omega )= 1\otimes\underline{\mathrm{ev}}(\omega_{k}\otimes y_{j}) \sbt (\rho_{j},y)\sbt (x_{k},\omega ) 
\\&=1\otimes(\omega_{k},y)\sbt (x_{k},\omega ) = 1 \otimes \ov{\underline{\mathrm{ev}}(\omega\otimes y)}= \Delta \left(\ov{\underline{\mathrm{ev}}(\omega\otimes y)}\right)
\end{align*} 
and 
$$\epsilon \left( (\omega_{i},y)\sbt (x_{i},\omega )\right) = \underline{\mathrm{ev}} (\omega_{i} \mathrm{ev}(x_{i}\otimes\omega )\otimes y)=\underline{\mathrm{ev}} (\omega\otimes y)=\epsilon \left(\ov{\underline{\mathrm{ev}}(\omega\otimes y)}\right) $$
where $y\in \mathfrak{Y}^{1} $ and $\omega\in\Omega$. The morphisms $\Delta,\epsilon$ being well defined for relation (\ref{EqRelInv2}), follows similarly and is left to the reader.\end{proof}
\begin{thm}\label{TIBXA} There is an isomorphism of categories $\prescript{}{IB(\Omega^{1} )}{\mathcal{M}}\cong \prescript{}{A}{\mathcal{IM}}_{A}^{\Omega^{1}}$.
\end{thm}
\begin{proof} For a $IB(\Omega^{1})$-module $M$, we can obtain left and right $\Omega^{1}$-intertwinings $\sigma,\tau$ on $M$ by restriction of scalars to subalgebras $T_{A^{e}}(\XA^{1}\otimes_{\field}\Omega^{1})$ and $T_{A^{e}}(\Omega^{1}\otimes_{\field}\mathfrak{Y}^{1} )$:
$$ \sigma (m\otimes\omega)= \omega_{i}\otimes (x_{i},\omega)m, \quad \tau(\omega\otimes m )= (\omega, y_{j})m\otimes \rho_{j} $$
For any $m\otimes\omega\in M\otimes\Omega^{1}$,
\begin{align*}
\tau\sigma (m\otimes\omega)&=\tau \left( \omega_{i}\otimes (x_{i},\omega)m\right)=  (\omega_{i}, y_{j})\sbt (x_{i},\omega ) m\otimes \rho_{j}
\\ &=  \ov{\underline{\mathrm{ev}}(\omega,y_{j})}m\otimes \rho_{j}=  m\underline{\mathrm{ev}}(\omega,y_{j})\otimes \rho_{j} =m\otimes \omega 
\end{align*}
holds by relation (\ref{EqRelInv1}). Similarly, $\sigma\tau =\mathrm{id}_{M\otimes\Omega^{1}}$ follows from relation (\ref{EqRelInv2}). The converse statement follows by looking at the induced actions of $T_{A^{e}}(\Omega^{1}\otimes_{\field}\mathfrak{Y}^{1} )$ and $T_{A^{e}} (\XA^{1}\otimes_{\field} \Omega^{1})$ on the underlying $A$-bimodule of any object $(M,\sigma)$ in $\prescript{}{A}{\mathcal{IM}}_{A}^{\Omega^{1}}$. This gives rise to an action of $R$ on $M$ and by the calculation above relations (\ref{EqRelInv1}) and (\ref{EqRelInv2}) annihilate $M$, making $M$ an $IB(\Omega^{1})$-module. \end{proof}
We can obtain the left bialgebroid $IB\XA^{1}$ whose module category recovers left bimodule connections with invertible $\Omega^{1}$-intertwinings, as the quotient of the free product of $T\XA^{1}\star IB(\Omega^{1} )$, by the relations (\ref{EqAX}), (\ref{EqXA}), (\ref{EqBXA}). 
\begin{rmk}\label{RTXop} By symmetry, we can describe the category of right connections, $\ct{E}_{A}$, as left modules over the algebra $T\mathfrak{Y}^{1} _\sbt$ which is defined as the quotient of the algebra $A^{op}\star T\mathfrak{Y}^{1} $ by relations 
$$\ov{a}\sbt y= ya, \hspace{0.6cm} y\sbt \ov{a} = ay+ \ov{\underline{\mathrm{ev}}(da \otimes y)} $$ 
for $y\in \mathfrak{Y}^{1} $ and $\ov{a}\in A^{op}$. In Lemma 3.70 of \cite{beggs2019quantum}, it is noted that a left bimodule connection $(M,\cnc ,\sigma )$ with invertible $\sigma$, has an induced right bimodule connection structure with $(M,\sigma^{-1}\cnc ,\sigma^{-1})$. We can view this as $T\mathfrak{Y}^{1} _\sbt$ being isomorphic to the subalgebra of $IB\XA^{1}$ generated by
$$ y\longmapsto  (\omega_{i},y)\sbt x_{i} ,\hspace{0.6cm} \ov{a}\longmapsto \ov{a}$$
for $y\in\mathfrak{Y}^{1} $ and $\ov{a}\in A^{op}$. 
\end{rmk}

As explained in Theorem \ref{TIBXA}, the relations (\ref{EqRelInv1}) and (\ref{EqRelInv2}) imply that the intertwining map on a $IB(\Omega^{1})$-module $M$ defined via $\sigma (m,\omega )=\sum_{i}\omega_{i} \otimes (x_{i},\omega )m$, is invertible. To do this we had to add a number of generators to the algebra ($\Omega^{1}\otimes_{\field}(\Omega^{1})^{\vee}$) and impose some minor relations, (\ref{EqRelInv1}) and (\ref{EqRelInv2}), on their interaction with the previous generators. However, as mentioned before $\sigma$ being invertible for a right fgp bimodule, does not make $\sigma^{\sharp}$ invertible and $\prescript{}{A}{\mathcal{IM}}_{A}^{\Omega^{1}}$ does not lift the closed structure of $\bim$. We need a suitable subcategory where $\sigma^{\sharp}: \pr{M}\otimes \Omega^{1} \rightarrow \Omega^{1}\otimes \pr{M}$ is invertible as well. Hence, we need to translate this condition to $(\sigma^{\sharp})^{\vee}: M\otimes (\Omega^{1} )^{\vee}\rightarrow (\Omega^{1})^{\vee}\otimes M$ being invertible. We can impose this condition on the bialgebroid $IB(\Omega^{1})$, by adding generators of the form $(\Omega^{1}) ^{\vee}\otimes_{\field}(\Omega^{1})^{\vee\vee} $ and similar relations to (\ref{EqRelInv1}) and (\ref{EqRelInv2}). On the other hand, $(\sigma^{\sharp})^{\sharp}$ will not necessarily be invertible, and we will have to repeat the process infinitely. Instead, in the next section we focus on the case where $\Omega^{1}\cong (\Omega^{1})^{\vee\vee}$ so that all the genrators required already exist in $R$ and by imposing the correct relations the arguments mentioned become cyclic. 
\subsection{Pivotal Modules} \label{SPiv}
\begin{defi}\label{DpivM} We say a bimodule $M$ is a \emph{pivotal bimodule} if there exists a bimodule isomorphism $\pr{M}\cong M^{\vee}$, or equivalently $M\cong M^{\vee\vee}$. 
\end{defi}
Many familiar examples of differential calculi are pivotal bimodules. In the classical case, if $A$ is commutative and $\Omega^{1}$ has the same left and right $A$-actions, then $\pr{\Omega^{1}}\cong\prescript{}{A}{\mathrm{Hom}(\Omega^{1},A)}$ and $(\Omega^{1})^{\vee}\cong \mathrm{Hom}_{A}(\Omega^{1},A)$ are naturally isomorphic.
\begin{ex}\hspace{0.01cm} [\emph{Quantum Riemannian Metric} \cite{beggs2019quantum}] We say a differential calculus $\Omega^{1}$ on algebra $A$ has a \emph{quantum metric} if $\Omega^{1}$ is self-dual i.e. $\pr{\Omega^{1}}\cong\Omega^{1}\cong (\Omega^{1})^{\vee}$ as an $A$-bimodule with evaluation and coevaluation maps $\mathrm{ev}$, $\mathrm{coev}$ satisfying 
$$(\mathrm{ev}\otimes \mathrm{id}_{\Omega^{1}}) (\mathrm{id}_{\Omega^{1}}\otimes \mathrm{coev}) =\mathrm{id}_{\Omega^{1}}= (\mathrm{id}_{\Omega^{1}}\otimes \mathrm{ev})(\mathrm{coev}\otimes \mathrm{id}_{\Omega^{1}}) $$
In this case, $g=\mathrm{coev}(1)$ is called a \emph{quantum metric} for the calculus. 
\end{ex} 
Of course any free bimodule such as the calculus over $M_{2}(\mathbb{C})$, presented in Example \ref{EM2C} is also pivotal and self dual. 
\begin{ex}\label{ExPivGrph}[Finite Quivers] Any quiver calculus as described in Example \ref{EGrph} is pivotal. Recall that $\XA^{1} =\mathrm{Span}_{\field}\lbrace \ovl{e} \mid e\in E\rbrace$, where $f\ovl{e} g=f(t(e))\ovl{e}g(s(e))$ for any pair $f,g\in \field (V)$. The evaluation and coevaluation maps are given by  
$$\mathrm{coev}(1)= \sum _{e\in E} \ovr{e}\otimes \ovl{e},\hspace{1cm} \mathrm{ev}(\ovl{e_{1}}\otimes \ovr{e_{2}}) = \delta_{e_{1},e_{2}} f_{t(e_{1})}$$
$$ \underline{\mathrm{coev}}(1)= \sum _{e\in E} \ovl{e}\otimes \ovr{e}\hspace{1cm} \underline{\mathrm{ev}}(\ovr{e_{1}}\otimes \ovl{e_{2}})=\delta_{e_{1},e_{2}} f_{s(e_{1})}$$
for any $e_{1},e_{2}\in E$, so that $\XA^{1}$ is both left dual and right dual to $\Omega^{1}$.
\end{ex}
Not every parallelised calculus is pivotal. However, the class of bicovariant calculi over Hopf algebras have this additional property:
\begin{ex}\label{EPivHopf}[Hopf Bimodules] Recall that a Hopf bimodule $\Omega^{1}$ for a Hopf algebra $A$, decomposes as a free right module $\Lambda\otimes_{\field} A$. When the antipode of $A$, $s$, is invertible, we utilise the following isomorphism to move between free right $A$-modules and free left $A$-modules:
\begin{equation*}
\xymatrix@R-27pt{\Phi : \Lambda\otimes_{\field} A\rightarrow A\otimes_{\field} \Lambda &\Phi^{-1}:A\otimes_{\field} \Lambda\rightarrow \Lambda\otimes_{\field} A
\\ \lambda\otimes_{\field}a\mapsto a_{(2)}\otimes_{\field}  s^{-1}(a_{(1)})\triangleright\lambda &a\otimes_{\field}\lambda \mapsto a_{(1)}\triangleright \lambda\otimes_{\field} a_{(2)} }
\end{equation*}
Observe that the left $A$-action translates to $\Phi(b\triangleright (\lambda\otimes_{\field}a) )= ba\otimes_{\field} \lambda$, making $\Omega^{1}$ free as a left $A$-module as well with $\Omega^{1} \cong A\otimes_{\field}\Lambda$ and $(\Omega^{1})^{\vee}\cong \Lambda^{\star}\otimes_{\field}A$. We denote elements of $A\otimes_{\field}\Lambda$ and $\Lambda^{\star}\otimes_{\field}A$ by $\bm{a}\otimes_{\field}\bm{\lambda}$ and $\bm{f}\otimes_{\field} \bm{a} $, respectively. Observe that as bimodules:
\begin{align*}
b( \bm{a}\otimes_{\field}\bm{\lambda}) &= \bm{ba}\otimes_{\field}\bm{\lambda}\hspace{0.7cm} ( \bm{a}\otimes_{\field}\bm{\lambda}) b=( \bm{ab_{(2)}}\otimes_{\field}s^{-1}(b_{(1)})\triangleright \bm{\lambda}) 
\\ (\bm{f}\otimes_{\field} \bm{a})b&=\bm{f}\otimes_{\field} \bm{ab} \hspace{0.7cm} b(\bm{f}\otimes_{\field} \bm{a})= (\bm{f}\triangleleft s^{-1}(b_{(1)})\otimes_{\field} \bm{b_{(2)}a})
\end{align*} 
for $b\in A$. Hence, the evaluation and coevaluation morphisms for $\Omega^{1}$ are calculated as follows 
$$\mathrm{coev}(1)= \sum_{i=1}^{n} (\lambda_{i}\otimes_{\field} 1)\otimes (1 \otimes_{\field} f_{i}), \quad \mathrm{ev}\big((a\otimes_{\field} f)\otimes (\lambda\otimes_{\field}b)\big)= ab f(\lambda ) $$
$$\underline{\mathrm{coev}}(1)= \sum_{i=1}^{n}  (\bm{f}_{i}\otimes_{\field} \bm{1})\otimes (\bm{1} \otimes_{\field}\bm{\lambda}_{i}), \quad \underline{\mathrm{ev}}\big((\bm{a}\otimes_{\field}\bm{\lambda})\otimes ( \bm{f}\otimes_{\field} \bm{b})\big)=ab f(\lambda )$$
where $\lambda\in \Lambda$, $f\in \Lambda^{\star}$ and $a,b\in A$. Furthermore, $\Omega^{1}$ is pivotal and the isomorphism between $\pr{(\Omega^{1})}$ and $(\Omega^{1})^{\vee} $ is provided by 
 \begin{align*}
\xymatrix@R-22pt@C-10pt{ \pr{(\Omega^{1})}= A\otimes_{\field} \Lambda^{\star}\longleftrightarrow \Lambda^{\star}\otimes_{\field}A= (\Omega^{1})^{\vee} 
\\ (a\otimes_{\field}f)\quad  \longmapsto \quad \sum_{i=1}^{n} (\bm{f_{i}}\otimes_{\field}\bm{a_{(2)}} )f\big( s\big((\lambda_{i})_{(-1)}a_{(1)}\big)\triangleright (\lambda_{i})_{(0)} \big)
\\\sum_{i=1}^{n} (a_{(2)} \otimes_{\field} f_{i}) f\big( s^{-2}\big( a_{(1)}(\lambda_{i})_{(-1)} \big)\triangleright (\lambda_{i})_{(0)}\big)\quad \ \reflectbox{\ensuremath{\longmapsto}}\  \quad (\bm{f}\otimes_{\field} \bm{a}) }
\end{align*}
where $\lambda_{(-1)}\otimes_{\field} \lambda_{(0)}=\delta_{L}(\lambda)$ denotes the left coaction of $\Lambda$ as a Yetter-Drinfeld modules. The category of Hopf modules over a Hopf algebra $A$, has a natural monoidal structure lifting that of $A$-bimodules. In particular, when the antipode $s$ of $A$ is invertible, the category of Hopf bimodules has a braided monoidal structure and is monoidal equivalent to the category of left Yetter-Drinfeld modules \cite{bespalov1997crossed}. Using this equivalence and the fact that in a braided monoidal category, left and right duals of an object are isomorphic, we obtain the above isomorphism. 
\end{ex}
\subsection{Resulting Hopf Algebroid Structure}\label{SPivHpf}
From this point onwards we assume that $\Omega^{1}$ is a pivotal bimodule and modify our notation from previous sections. We denote evaluation and coevaluation maps as before, but with applying the isomorphism $\XA^{1}\cong \mathfrak{Y}^{1} $ so that  
\begin{align}
\mathrm{coev}: A \rightarrow \Omega^{1}\otimes\XA^{1}, \quad \mathrm{coev}(1)&= \sum_{i} \omega_{i}\otimes x_{i}, \quad \mathrm{ev}: \XA^{1}\otimes\Omega^{1} \rightarrow A
\\ \underline{\mathrm{coev}}: A \rightarrow \XA^{1}\otimes\Omega^{1}, \quad \underline{\mathrm{coev}}(1)&= \sum_{j} y_{j}\otimes \rho_{j}, \quad \underline{\mathrm{ev}}: \Omega^{1}\otimes \XA^{1}\rightarrow A
\end{align}
With this notation we define $H(\Omega^{1} )$ to be the quotient of $IB(\Omega^{1} )$ by the additional relations 
\begin{align}
  (y_{j},\omega)\sbt (\rho_{j},x)&=\ov{\mathrm{ev}(x, \omega)} \label{EqRelHpf1}
\\  (\omega, x_{i})\sbt(x,\omega_{i})&=\underline{\mathrm{ev}}(\omega,x)\label{EqRelHpf2}
\end{align}
for any $x\in \XA^{1}$ and $\omega \in\Omega^{1}$.
\begin{lemma}\label{LPiv} The comultiplication and counit of $IB(\Omega^{1})$, are well-defined on the quotient algebra, $H(\Omega^{1} )$, and give rise to an $A$-bialgebroid structure on $H(\Omega^{1} )$. 
\end{lemma}
\begin{proof} The proof is completely symmetric to that of Lemma \ref{LInvBialg} and is left to the reader. \end{proof}
By Theorem \ref{TIBXA}, an $IB(\Omega^{1})$-module can be viewed as an $A$-bimodule with an invertible $\Omega^{1}$-intertwining $\sigma :M\otimes \Omega^{1} \rightarrow \Omega^{1}\otimes M$. Hence, we can translate the additional relations in $B\XA^{1}$, to the maps 
\begin{align}
(\mathrm{ev}\otimes \mathrm{id}_{M\otimes \XA^{1}})(\mathrm{id}_{\XA^{1}}\otimes\sigma\otimes \mathrm{id}_{\XA^{1}})(\mathrm{id}_{\XA^{1}\otimes M}\otimes \mathrm{coev}): \XA^{1}\otimes M\rightarrow M\otimes\XA^{1} &\label{EqsigX}
\\ (\mathrm{id}_{\XA^{1}\otimes M}\otimes \underline{\mathrm{ev}} )(\mathrm{id}_{\XA^{1}}\otimes\sigma^{-1}\otimes \mathrm{id}_{\XA^{1}})(\underline{\mathrm{coev}}\otimes \mathrm{id}_{M\otimes \XA^{1}}): M\otimes\XA^{1} \rightarrow \XA^{1}\otimes M &\label{EqXsig}
\end{align}
being each others inverses. Notice that when $M$ is right fgp, the second map being invertible is equivalent to $\sigma^{\sharp}$ being invertible, which is what we desire in a closed subcategory of $\prescript{}{A}{\mathcal{M}}^{\Omega^{1}}_{A}$. If $\Omega^{1}$ were not pivotal, we would have to write $\mathfrak{Y}^{1} $ instead of $\XA^{1}$ in the second morphism, and the two morphisms could not be inverses. 
\begin{thm}\label{TPivXA} The category of $H(\Omega^{1})$-modules is isomorphic to the category of $A$-bimodules with invertible $\Omega^{1}$-intertwining maps $\sigma$, such that bimodule maps (\ref{EqsigX}), (\ref{EqXsig}) are inverses. We denote this category by $\prescript{\XA^{1}}{A}{I\mathcal{M}}^{\Omega^{1}}_{A}$. 
\end{thm}
\begin{proof} Under the correspondence described in Theorem \ref{TIBXA}, an $H(\Omega^{1})$-module $M$ has an induced invertible $\Omega^{1}$-intertwining $\sigma$. By recalling the definition of $\sigma$, the morphisms (\ref{EqsigX}) and (\ref{EqXsig}) translate to 
$$ (\mathrm{ev}\otimes \mathrm{id}_{M\otimes \XA^{1}})(x \otimes \sigma (m \otimes \omega_{i})\otimes x_{i})= (x,\omega_{i})m\otimes x_{i}$$ 
and 
$$(\mathrm{id}_{\XA^{1}\otimes M}\otimes \underline{\mathrm{ev}} )(y_{j}\otimes\sigma^{-1}(\rho_{j}\otimes m)\otimes x)=  y_{j}\otimes (\rho_{j},x)m$$
respectively, for any $x\in \XA^{1}$ and $m\in M$. In this form, the morphisms being inverses follows directly from (\ref{EqRelHpf1}) and (\ref{EqRelHpf2}). The converse direction also follow trivially.\end{proof}
In the above paragraph, we already hinted at the fact that the left (right) duals, of right (left) fgp bimodules with $\Omega^{1}$-intertwinings in $\prescript{\XA^{1}}{A}{I\mathcal{M}}^{\Omega^{1}}_{A}$, will have invertible $\Omega^{1}$-intertwinings. We now show that in fact $\prescript{\XA^{1}}{A}{I\mathcal{M}}^{\Omega^{1}}_{A}$ is closed and $H(\Omega^{1})$ is a Schauenburg Hopf algebroid. In fact, $H(\Omega^{1} )$ admits an invertible antipode and has the form of a B{\"o}hm-Szlach{\'a}nyi Hopf algebroid.
\begin{thm}\label{THOmega} The map $S:H(\Omega^{1} )\rightarrow H(\Omega^{1} )$ is defined by 
\begin{align*}
S(a)=\ov{a},\hspace{0.7cm} & S((x,\omega))=(\omega,x)
\\S(\ov{a})=a,  \hspace{0,7cm} & S((\omega,x))=(x,\omega)
\end{align*}
for $a\in A$, $\omega\in \Omega^{1}$ and $x\in \XA^{1}$ and extended anti-multiplicatively to $H(\Omega^{1})$. The map $S$ is a well-defined anti-algebra automorphism of algebra $H(\Omega^{1} )$, with $S^{-1}=S$ and satisfies the conditions in Definition \ref{DHgebroid} (II). 
\end{thm}
\begin{proof} We have defined $S$ on the generators of the algebra, and must verify that $S$ is well-defined by looking at the relations. Notice that relations (\ref{EqMfrak}) and (\ref{EqYFrak}) are symmetric under $S$ and $S$ is well-defined on relation (\ref{EqRelInv1}) due to relation (\ref{EqRelHpf2}): 
\begin{align*}
S\left(  (\omega_{i},x)\sbt (x_{i},\omega )\right)&=  S((x_{i},\omega ))\sbt S((\omega_{i},x)) 
\\ &= (\omega, x_{i})\sbt(x,\omega_{i})=\underline{\mathrm{ev}}(\omega,x)=S\left(\ov{\underline{\mathrm{ev}}(\omega,x)}\right)
\end{align*} 
where $x\in \XA^{1}$ and $\omega\in\Omega^{1}$. Similar arguments apply for the other relations and one can conclude that $S$ is well defined and by definition $S=S^{-1}$. Since the image of the coproduct falls in the Takeuchi $\times$-product, we only need to check the antipode conditions (\ref{EqSCond}) and (\ref{EqSinvCond}) on the generators of the bialgebroid. For generators $(x,\omega)\in\XA^{1}\otimes_{\field} \Omega^{1}$, 
\begin{align*}
S((x,\omega)_{(1)})_{(1)}\sbt & (x,\omega)_{(2)}\diamond S((x,\omega)_{(1)})_{(2)}=(\omega_{i},x)_{(1)}\sbt(x_{i},\omega)\diamond (\omega_{i},x)_{(2)}
\\ &=(\omega_{i},y_{j})\sbt(x_{i},\omega)\diamond (\rho_{j},x)=   \ov{\underline{\mathrm{ev}}(\omega ,y_{j})} \diamond (\rho_{j},x)
\\&= 1\diamond  \underline{\mathrm{ev}}(\omega ,y_{j})\sbt (\rho_{j},x)=1\diamond (\omega,x)=1\diamond S(x,\omega )
\end{align*}
and 
\begin{align*}
S^{-1}((x,\omega)_{(2)})_{(1)}&\diamond  S^{-1}((x,\omega)_{(2)})_{(2)}(x,\omega)_{(1)}=(\omega , x_{i})_{(1)} \diamond(\omega , x_{i})_{(1)}\sbt (x,\omega_{i})
\\ &=(\omega , y_{j})\diamond (\rho_{j} , x_{i})\sbt (x,\omega_{i})= (\omega , y_{j})\otimes  \underline{\mathrm{ev}}(\rho_{j},x)
\\&=\ov{ \underline{\mathrm{ev}}(\rho_{j},x)} \sbt(\omega , y_{j})\diamond 1 = (\omega,x)\diamond 1= S^{-1}((x,\omega))\diamond 1
\end{align*}
hold. A symmetric argument applies for generators of the form $(\omega,x)\in \Omega^{1}\otimes_{\field}\XA^{1}$.\end{proof}
Using the antipode we can describe the closed structure of $\prescript{\XA^{1}}{A}{I\mathcal{M}}^{\Omega^{1}}_{A}$, which lifts that of $\bim$. For a pair of $H(\Omega^{1})$-modules $M$ and $N$, we recover the action of $H(\Omega^{1})$ by (\ref{EqClosAntip}):
\begin{align}
[(x, \omega)f ](m)= (x,\rho_{j}) f((\omega,y_{j})m),  & \hspace{0.7cm} [(\rho, y)f ](m)= (\rho,x_{i}) f((y,\omega_{i})m) 
\\ [(x, \omega)g ](m)= (y_{j},\omega) g((\rho_{j},x)m),  & \hspace{0.7cm} [(\rho, y)g ](m)= (\omega_{i},y) g((x_{i},\rho)m)
\end{align}
for any $m\in M$, $(x,\omega)\in \XA^{1}\otimes_{\field}\Omega^{1}$, $(\rho,y)\in \Omega^{1}\otimes_{\field}\XA^{1}$, $f\in \mathrm{Hom}_{A}(M,N)$ and $g\in \prescript{}{A}{\mathrm{Hom}}(M,N) $.

\textbf{Notation.} We have used the notation $[hf](m)= h_{(+)}f(h_{(-)}m)$ to distinguish between $[hf](m)$, where $[hf]$ is the morphism obtained by $h\in H(\Omega^{1})$ acting on the morphism $f\in \mathrm{Hom}_{A}(M,N)$ and $hf(m)$, where $h$ acts on $f(m)$ as an element of $N$. In what follows, we will continue to adapt this notation. 

Now we look at bimodule connections whose underlying intertwinings belong to $\prescript{\XA^{1}}{A}{I\mathcal{M}}^{\Omega^{1}}_{A}$. At this point it should be clear that to do this we need to take the quotient of $IB\XA^{1}$ by the ideal generated by the set of relations (\ref{EqRelHpf1}) and (\ref{EqRelHpf2}). We denote this algebra by $H\XA^{1}$. From the arguments in Lemmas \ref{LInvBialg} and \ref{LPiv} it follows that the resulting algebra carries down the left $A$-bialgebroid structure of $B\XA^{1}$. Moreover, $H\XA^{1}$ is a Schauenburg Hopf algebroid. Observe that in order to demonstrate this, we only need to prove that the category of $H\XA^{1}$-modules lifts the closed structure of $\bim$. Since we have already described the action of $H(\Omega^{1})$ on $\mathrm{Hom}_{A}(M,N)$ and $\prescript{}{A}{\mathrm{Hom}}(M,N) $, we only need to present a well-defined action of elements of $\XA^{1}$ in $H\XA^{1}$, or in particular a connection on $\mathrm{Hom}_{A}(M,N)$ and $\prescript{}{A}{\mathrm{Hom}}(M,N) $.
\begin{thm}\label{THXAHopf} For $H\XA^{1}$-modules $M,N$, we can extend the actions of $H(\Omega^{1} )$ on the inner homs, to actions of $H\XA^{1}$ by defining the action of elements $x\in \XA^{1}$ by
\begin{align}
[xf](m) =& x f(m) -(x,\rho_{j}) f\big((\omega_{i},y_{j})\sbt x_{i}m\big)
\\ [xg](m)=&   y_{j} g\big( (\rho_{j},x)m\big)- g(y_{j}\sbt(\rho_{j},x)m)
\end{align}
where $m\in M$, $f\in \mathrm{Hom}_{A}(M,N)$ and $g\in \prescript{}{A}{\mathrm{Hom}}(M,N) $, so that the closed monoidal structure of $\bim$ lifts to the category of $H\XA^{1}$-modules. 
\end{thm}
\begin{proof} We must first check that the $H\XA^{1}$-actions defined above are well defined. We then proceed to showing that the units and counits of the adjunctions providing the closed structure of $\bim$, (\ref{EqAdj}), are $H\XA^{1}$-module morphisms. Since the actions of elements in $\Omega^{1}\otimes_{\field}\XA^{1}$ and $\XA^{1}\otimes_{\field}\Omega^{1}$ are lifted from $H(\Omega^{1})$, we only need to check these facts on the generators of the form $x\in\XA^{1}$. In particular, we only need to look at relations (\ref{EqAX}), (\ref{EqXA}) and (\ref{EqBXA}) for the $H\XA^{1}$-action to be well-defined:
\begin{align*}
[(a\sbt x)&f](m) = a\left( x (f(m)) -(x,\rho_{j})f \big((\omega_{i},y_{j})\sbt x_{i}m\big)\right)= [(a x)f](m)
\\ [(x\sbt a)&f](m)= x (af(m)) -(x,\rho_{j}) \left( af\big((\omega_{i},y_{j})\sbt x_{i}m\big)\right)
\\=&(xa)f(m) + \mathrm{ev}(x,da)f(m)-(xa,\rho_{j})  f\big((\omega_{i},y_{j})\sbt x_{i}m\big)
\\=& [(xa)f](m)+[\mathrm{ev}(x,da)f](m)
\\ [(x\sbt \ov{a})&f] (m)=x f(am) -(x,\rho_{j})  f\big(a\sbt(\omega_{i},y_{j})\sbt x_{i}m\big)
\\=& x f(am) -(x,\rho_{j})  f\big((\omega_{i},y_{j})\sbt (x_{i}a) m\big)= x f(am)
\\& -\left((x,\rho_{j})  f\big((\omega_{i},y_{j})\sbt x_{i}\sbt a m\big) +(x,\rho_{j})  f\big((\omega_{i},y_{j})\sbt \mathrm{ev}(x_{i}\otimes da) m\big)\right)
\\ =&[\ov{a}\sbt x f](m) + [(x,da)f](m) 
\end{align*}
where $a\in A$, $x\in \XA^{1}$ and $f\in \mathrm{Hom}_{A}(M,N)$. Similarly for $g\in \prescript{}{A}{\mathrm{Hom}}(M,N) $, we note that the right $A$-action on $M$ arises from the action of $A^{op}\subset H\XA^{1}$.
\begin{align*}
[(a\sbt x)&g](m) = y_{j} g\big((\rho_{j},x)\sbt \ov{a} m\big)-g(y_{j}\sbt(\rho_{j},x)\sbt \ov{a}m)=[(ax)g](m)
\\ [(x\sbt a)&g](m)= y_{j}g\big(\ov{a}\sbt (\rho_{j},x) m\big)-g\big(\ov{a}\sbt y_{j}\sbt(\rho_{j},x)m\big)
\\=&y_{j}g\big((\rho_{j},ax) m\big)-g(y_{j}\sbt(\rho_{j},ax)m)+g\big((y_{j},da)\sbt(\rho_{j},x)m\big)
\\=&[(xa)g](m)+ [\mathrm{ev}(x,da)g](m)
\\ [(x\sbt \ov{a})&g](m)= y_{j}\sbt \ov{a} g\big((\rho_{j},x)m\big)-\ov{a}g\big(y_{j}\sbt(\rho_{j},x)m\big)
\\=& [(\ov{a}\sbt x)g](m)+ (y_{j}, da) g\big((\rho_{j},x)m\big)= [(\ov{a}\sbt x)g](m)+[(x,da)g](m)
\end{align*} 
Hence, the actions of $H(\Omega^{1})$ on the inner homs extend to well-defined actions of $H\XA^{1}$. We now show that the the unit and counit, $\varrho^{M}$ and $\varepsilon^{M}$, of the adjunction $-\otimes M\dashv \mathrm{Hom}_{A}(M,-) $, respect the $H\XA^{1}$-actions. Let $x\in \XA^{1}$, $f\in \mathrm{Hom}_{A}(M,N)$, $m\in M$ and $n\in N$.
\begin{align*} 
[x \varrho^{M}_{N}(n)](m)=& xf_{n}(m) -(x,\rho_{j}) f_{n}\big((\omega_{i},y_{j})\sbt x_{i}m\big)
\\ =&x (n\otimes m) -(x,\rho_{j}) \big(n\otimes(\omega_{i},y_{j})\sbt x_{i}m \big)= xn\otimes m 
\\ &+(x,\omega_{l})n \otimes x_{l}m - (x,\omega_{k})n\otimes (x_{k},\rho_{j})\sbt (\omega_{i},y_{j})\sbt x_{i}m
\\ = &xn \otimes m=f_{xn}(m)=[\varrho^{M}_{N}(xn)](m)
\\ \varepsilon^{M}_{N} \big(x(f\otimes m)\big)=& \varepsilon^{M}_{N}\left( xf \otimes m+ (x,\omega_{i})f\otimes x_{i}m\right)=[xf](m)
\\&+  [(x,\omega_{i})f](x_{i}m) =x f(m)-(x,\rho_{k}) f\big((\omega_{l},y_{k})\sbt x_{l}m \big)
\\& +(x,\rho_{j}) f\big((\omega_{i},y_{j})\sbt x_{i} m\big)= xf(m)=x \varepsilon^{M}_{N} (f\otimes m)
\end{align*}
Similarly, we look at the unit and counit, $\Theta^{M}$ and $\Pi^{M}$, of the adjunction $M\otimes -\dashv \prescript{}{A}{\mathrm{Hom}}(M,-) $ respecting the $H\XA^{1}$-actions. Let $x\in \XA^{1}$, $g\in \prescript{}{A}{\mathrm{Hom}}(M,N)$, $m\in M$ and $n\in N$.
\begin{align*}
[x \Theta^{M}_{N}&(n)](m)= y_{j} g_{n}\big((\rho_{j},x)m\big)-g_{n}\big(y_{j}\sbt(\rho_{j},x)m\big)
\\ =&y_{j}\big( (\rho_{j},x)m \otimes n \big)-y_{j}\sbt (\rho_{j},x)m \otimes n
\\ =& y_{j} (\rho_{j},x)m \otimes n +(y_{j},\omega_{i})\sbt(\rho_{j},x)m \otimes x_{i}n -y_{j}\sbt(\rho_{j},x)m \otimes n
\\ =& \ov{\mathrm{ev}(x,\omega_{i})}m \otimes x_{i}n =m\otimes xn =[\Theta^{M}_{N}(xn)](m)
\\ \Pi^{M}_{N} \big(x&(m\otimes g)\big)= \Pi^{M}_{N} \left(xm\otimes g + (x,\omega_{i})m \otimes x_{i}g\right)
\\=& g(xm)+[x_{i}g]\big((x,\omega_{i})m\big)=g(xm) + y_{j}g\big((\rho_{j},x_{i})\sbt (x,\omega_{i})m\big)
\\& -g\big(y_{j}\sbt(\rho_{j},x_{i})\sbt (x,\omega_{i})m\big)=g(xm)+y_{j} g\big(\underline{\mathrm{ev}}(\rho_{j}\otimes x)m\big)
\\& -g\big(y_{j}\sbt\underline{\mathrm{ev}}(\rho_{j}\otimes x)m\big)=g(xm) +xg(m)-g(xm) 
\\& +\mathrm{ev}(y_{j}\otimes d\underline{\mathrm{ev}}(\rho_{j}\otimes x)) g(m)- g\big(\mathrm{ev}(y_{j}\otimes d\underline{\mathrm{ev}}(\rho_{j}\otimes x)) m\big) 
\\ =&g(xm)= x\Pi^{M}_{N}(m\otimes g) 
\end{align*}\end{proof}
The Hopf algebroid $H\XA^{1}$ is not expected to admit an antipode in general. For the existence of an antipode, we require a linear map $\Upsilon :\XA^{1}\rightarrow A$ satisfying 
\begin{equation}\label{EqSpade}\Upsilon (xa)=\Upsilon (x)a+\mathrm{ev}(x\otimes da), \hspace{1cm} \Upsilon (ax)=a\Upsilon (x)+\underline{\mathrm{ev}}(da \otimes x) 
\end{equation}
for any $x\in \XA^{1}$ and $a\in A$. In fact, the existence of such a map is equivalent to $H\XA^{1}$ admitting an antipode.
\begin{thm}\label{THXaAnti} The Hopf algebroid $H\XA^{1}$ admits an invertible antipode if and only if there exists a linear map $\Upsilon: \XA^{1}\rightarrow A$ satisfying (\ref{EqSpade}). In particular, if such $\Upsilon$ exists, the maps $S$ and $S^{-1}$ defined by 
\begin{align}
S(x)&=-(\omega_{i},x)\sbt x_{i}-\ov{\Upsilon (x) }
\\ S^{-1}(x)&=-\left( y_{j}+\Upsilon (y_{j})\right)\sbt(\rho_{j},x)
\end{align}
for $x\in \XA^{1}$, extend $S$ and $S^{-1}$ from Theorem \ref{THOmega}, to well-defined anti-algebra morphisms on $H\XA^{1}$ and are inverses. Furthermore, they satisfy the conditions in Definition \ref{DHgebroid} (II). 
\end{thm}
\begin{proof} ($\Rightarrow$) First we recall the following elementary fact stated in \cite{krahmer2015lie}: if a Hopf algerboid admits an antipode $S: H\XA^{1} \rightarrow H\XA^{1}$ as defined in Definition \ref{DHgebroid} (II), then
$$ a\triangleleft h= \epsilon (S(h)\sbt a) , \quad a\in A, \ h\in H\XA^{1}$$
defines a right action of the algebra $H\XA^{1}$ on $A$, such that the action $A^{op}\subset H\XA^{1}$ coincides with left multiplication i.e. $a_{1}\triangleleft\ov{a_{2}}=a_{2}a_{1}$ for $a_{1},a_{2}\in A$. Hence, we define the map $\Upsilon :\XA^{1}\rightarrow A$ by $\Upsilon (x) := -\epsilon (S(x))=-1\triangleleft x$. It is then straightforward to check that (\ref{EqSpade}) holds:
\begin{align*}
\Upsilon (ax)=& -\epsilon (S(ax))= -\epsilon(S(x)\sbt \ov{a})=-\epsilon (S(x)\sbt a) = -\epsilon (S(\ov{a}\sbt x))
\\ = & -\epsilon (S(x\sbt \ov{a}))+\epsilon (S((x,da)))=-a\epsilon (S(x))+\epsilon ((da,x))
\\&  = a\Upsilon(x) + \underline{\mathrm{ev}}(da \otimes  x) 
\\ \Upsilon (x)a=&-\epsilon (S(x))a= -\epsilon (\ov{a}\sbt S(x))= -\epsilon (S(x\sbt a))
\\ =& -\epsilon (S(xa)+S(\mathrm{ev}(x\otimes  da)))= \Upsilon (xa)-\mathrm{ev}(x\otimes da)
\end{align*}

($\Leftarrow$) We assume such a map $\Upsilon$ exist. Hence, we have defined $S$ and $S^{-1}$ on the generators of $H\XA^{1}$ and must first check whether they are well defined on $H\XA^{1}$. For the relations present in $H(\Omega^{1})$, this has already been done in the proof of Theorem \ref{THOmega}. Hence, we only have to check relations (\ref{EqAX}), (\ref{EqXA}) and (\ref{EqBXA}). First we demonstrate this for $S$ 
\begin{align*}
S(a\sbt x)&=S(x)\sbt \ov{a}= -(\omega_{i},x)\sbt x_{i}\sbt\ov{a}-\ov{\Upsilon (x) }\sbt\ov{a}
\\ &= -(\omega_{i},ax)\sbt x_{i} -(\omega_{i},x)\sbt (x_{i},da) -\ov{a\Upsilon (x) }
\\ &= -(\omega_{i},ax)\sbt x_{i} -\ov{\Upsilon (ax) }= S(ax)
\\S(x\sbt a) &= \ov{a}\sbt S(x) =  -(\omega_{i},xa)\sbt x_{i}-\ov{\Upsilon (x)a }
\\ &= -(\omega_{i},xa)\sbt x_{i}-\ov{\Upsilon (xa) }+\ov{\mathrm{ev}(x\otimes  da)}= S(xa+\mathrm{ev}(x\otimes  da))
\\S(x\sbt \ov{a})&= a\sbt S(x)= -(a\omega_{i},x)\sbt x_{i}-a\ov{\Upsilon (x) }=-(\omega_{i},x)\sbt x_{i}a-\ov{\Upsilon (x)}a 
\\ &=S(x)\sbt a+(\omega_{i},x)\sbt \mathrm{ev}(x_{i}\otimes da) = S(\ov{a}\sbt x+(x,da))
\end{align*}
and for $S^{-1}$ 
\begin{align*}
S^{-1}(a\sbt x)&=S^{-1}(x)\sbt \ov{a} = -\left( y_{j}+\Upsilon (y_{j})\right)\sbt(\rho_{j},x)\sbt\ov{a}     =S^{-1}(a x)
\\ S^{-1}(x\sbt a) &= \ov{a}\sbt S^{-1}(x)=-\ov{a}\sbt\left( y_{j}+\Upsilon (y_{j})\right)\sbt(\rho_{j},x)
\\ &=-\left( y_{j}+\Upsilon (y_{j})\right)\sbt\ov{a}\sbt(\rho_{j},x) + ( y_{j},da)\sbt(\rho_{j},x)
\\& =S^{-1}(xa) +  \ov{\mathrm{ev}(x\otimes da)}= S^{-1}(xa+ \mathrm{ev}(x\otimes  da))
\\S^{-1}(x\sbt \ov{a})&= a\sbt S^{-1}(x)= -a\sbt\left( y_{j}+\Upsilon (y_{j})\right)\sbt(\rho_{j},x)
\\ &=-\left( ay_{j}+\Upsilon (ay_{j})\right)\sbt(\rho_{j},x)+  \underline{\mathrm{ev}}(da\otimes  y_{j})(\rho_{j},x)
\\ &= S^{-1}(x)\sbt a +(da,x)= S^{-1}(\ov{a}\sbt x+(x,da))
\end{align*}
where $a\in A$ and $x\in \XA^{1}$. We must also check that $S$ and $S^{-1}$ are inverse. Let $x\in \XA^{1}$.
\begin{align*}
S^{-1}S(x)&=-S^{-1}(x_{i})\sbt (x,\omega_{i}) -\Upsilon (x) 
\\&= \left( y_{j}+\Upsilon (y_{j})\right)\sbt (\rho_{j},x_{i})\sbt (x,\omega_{i}) -\Upsilon (x) 
\\&=\left( y_{j}+\Upsilon (y_{j})\right)\sbt \underline{\mathrm{ev}}(x\otimes  \rho_{j})-\Upsilon (x) =x
\\SS^{-1}(x)&= -(x,\rho_{j})\sbt\left( S(y_{j})+\ov{\Upsilon (y_{j})}\right)
\\&= (x,\rho_{j})\sbt \left((\omega_{i},y_{j})\sbt x_{i}+\ov{\Upsilon (y_{j})}\right) -(x,\rho_{j})\ov{\Upsilon (y_{j})}
\\&= (x,\rho_{j})\sbt(\omega_{i},y_{j})\sbt x_{i} =  \mathrm{ev}(x\otimes \omega_{i})x_{i}=x
\end{align*}
Since the coproduct falls in the Takeuchi product, we only need to verify axioms (\ref{EqSCond}) and (\ref{EqSinvCond}) on the generators of the bialgebroid. Let $x\in \XA^{1}$.
\begin{align*}
S(x_{(1)})_{(1)}&\sbt x_{(2)}\diamond S(x_{(1)})_{(2)}= (\omega_{i},x)_{(1)}\sbt x_{i}\diamond (\omega_{i},x)_{(2)}+S(x)_{(1)}\diamond S(x)_{(2)} 
\\ =&(\omega_{i},y_{j})\sbt x_{i}\diamond (\rho_{j},x)-\left((\omega_{i},x)\sbt x_{i}\right)_{(1)}\diamond \left((\omega_{i},x)\sbt x_{i}\right)_{(2)} 
\\&-1\diamond \ov{\Upsilon (x)} = - (\omega_{i},y_{j})\sbt (x_{i},\omega_{k})\diamond (\rho_{j},x)\sbt x_{k}-1\diamond \ov{\Upsilon (x)}
\\ =& -\ov{\underline{\mathrm{ev}}(\omega_{k}\otimes y_{j})}\diamond (\rho_{j},x)\sbt x_{k}-1\diamond \ov{\Upsilon (x)}=1\diamond S(x) &
\end{align*}
\begin{align*} S^{-1}(&x_{(2)})_{(1)}\diamond S^{-1}(x_{(2)})_{(2)}\sbt x_{(1)}= 1\diamond x +  S^{-1}(x_{i})_{(1)}\diamond S^{-1}(x_{i})_{(2)}\sbt (x,\omega_{i})
\\=& 1\diamond x  -\left( y_{j}\right)_{(1)}\sbt(\rho_{j},y_{k})\diamond \left( y_{j}\right)_{(2)}\sbt(\rho_{k}, x_{i})\sbt (x,\omega_{i})
\\& -\Upsilon (y_{j}) \sbt(\rho_{j},y_{k})\diamond (\rho_{k}, x_{i})\sbt (x,\omega_{i})
\\=& 1\diamond x- \Upsilon (y_{j})\sbt(\rho_{j},y_{k})\diamond \underline{\mathrm{ev}}(\rho_{k}\otimes  x)  
-y_{j} \sbt(\rho_{j},y_{k})\diamond \underline{\mathrm{ev}}(\rho_{k}\otimes  x)
\\&  - ( y_{j}, \omega_{l})\sbt(\rho_{j},y_{k})\diamond x_{l}\sbt \underline{\mathrm{ev}}(\rho_{k}\otimes  x)
\\=& -\left( y_{j}+\Upsilon (y_{j})\right)\sbt(\rho_{j},x)\diamond 1  +1\diamond x - \ov{\mathrm{ev}( y_{k}\otimes\omega_{l} )}\diamond x_{l}\  \underline{\mathrm{ev}}(\rho_{k}\otimes x)
\\=&S^{-1}(x)\diamond 1
\end{align*}
In the algebraic manipulations above, both properties of (\ref{EqSpade}) have been used but the additional terms have been omitted.\end{proof}
For any pair of $H\XA^{1}$-modules $M$ and $N$, one can easily check that the induced connections on the inner homs $\mathrm{Hom}_{A}(M,N)$ and $\prescript{}{A}{\mathrm{Hom}}(M,N) $ calculated via the antipode, (\ref{EqClosHpf}), agrees with those presented in Theorem \ref{THXAHopf}. In particular, the terms including $\Upsilon$ cancel out in the calculation of (\ref{EqClosHpf}).

\begin{rmk} In the classical theory of Hopf algebras, if a bialgebra admits an antipode, the antipode is unique. However, as demonstrated by the above theorem, this is not true for Hopf algebroids. In fact, one can add any bimodule morphism $\phi :\XA^{1} \rightarrow A$ to $\Upsilon$ and $\Upsilon +\phi$ will again satisfy (\ref{EqSpade}).  
\end{rmk} 

\subsection{Examples of Hopf Algebroids}\label{SExHopf}
As a corollary of Theorem \ref{THXaAnti}, several of the Hopf algebroids constructed here will admit antipodes. In particular, if the calculus $\Omega^{1}$ is a finitely generated free $A$-bimodule with basis $\lbrace f_{i}\rbrace_{i=1}^{n}$ for $\XA^{1}$, then $\Upsilon (\sum_{i} a_{i}f_{i})=\sum_{i}\mathrm{ev}(da_{i}\otimes f_{i})$ satisfies (\ref{EqSpade}), for any collection of elements $a_{i}\in A$.  
\begin{ex}\label{EHopfDer} [Derivation Calculus] Recall the bialgebroid constructed in Example \ref{EBiaDeriv} for a derivation $d:A\rightarrow A$. To obtain $H\XA^{1}$, a new generator $E = (1,1)\in \Omega^{1}\otimes_{\field}\XA^{1}$ is added and the new relations are equivalent to $F\sbt E=1= E\sbt F$. Hence, $H\XA^{1}=A^{e}.\field\langle D,F,F^{-1}\rangle$ with the commutation relations in Example \ref{EBiaDeriv}. The coproduct, counit and antipode are extended as follows
\begin{align*}
\Delta(F^{-1})=F^{-1}\otimes F^{-1}& , \quad \quad\epsilon (F^{-1})=1 
\\ S(D)=-F^{-1}D \quad S(F)=&F^{-1}\quad S(F^{-1})=F
\end{align*}
\end{ex}
\begin{ex}\label{ExHopfEM2C}[$M_{2}(\mathbb{C})$] For the differential calculus of Example \ref{EM2C}, $\Omega^{1}$ is a free bimodule and the Hopf algebroid $H\XA^{1}$ factorises as $A^{e}.\mathbb{C}\langle f_{i},\omx{i}{\gamma}{j},\omx{i}{\kappa}{j}\mid i,j\in \lbrace \mathsf{s},\mathsf{t}\rbrace\rangle$ with the relations of $B\XA^{1}$ presented in Example \ref{EBiaM2C} and additional relations
\begin{align*} [\omx{i}{\kappa}{j},a\ov{b}]=&0
\\\omx{\mathsf{s}}{\gamma}{i}\sbt \omx{\mathsf{s}}{\kappa}{j}+ \omx{\mathsf{t}}{\gamma}{i}\sbt \omx{\mathsf{t}}{\kappa}{j}=\delta_{i,j}&= \omx{i}{\gamma}{\mathsf{s}}\sbt \omx{j}{\kappa}{\mathsf{s}}+ \omx{i}{\gamma}{\mathsf{t}}\sbt \omx{j}{\kappa}{\mathsf{t}}
\\ \omx{\mathsf{s}}{\kappa}{i}\sbt \omx{\mathsf{s}}{\gamma}{j}+ \omx{\mathsf{t}}{\kappa}{i}\sbt \omx{\mathsf{t}}{\gamma}{j}=\delta_{i,j}&= \omx{i}{\kappa}{\mathsf{s}}\sbt \omx{j}{\gamma}{\mathsf{s}}+ \omx{i}{\kappa}{\mathsf{t}}\sbt \omx{j}{\gamma}{\mathsf{t}}
\end{align*} 
for all $i,j  \in \lbrace \mathsf{s},\mathsf{t}\rbrace$ and $a\ov{b}\in A^{e}$. The coproduct, counit and antipode extend similarly by
\begin{align*}
\Delta ( \omx{i}{\kappa}{j}) =\omx{i}{\kappa}{\mathsf{s}}\otimes \omx{\mathsf{s}}{\kappa}{j}+\omx{i}{\kappa}{\mathsf{t}}\otimes \omx{\mathsf{t}}{\kappa}{j},&\quad \epsilon (\omx{i}{\kappa}{j})=\delta_{i,j}
\\S(f_{i})=-\omx{\mathsf{s}}{\kappa}{i}\sbt f_{\mathsf{s}} -\omx{\mathsf{t}}{\kappa}{i}\sbt f_{\mathsf{t}}, \quad S(\omx{i}{\gamma}{j}) &= \omx{j}{\kappa}{i}, \quad S(\omx{i}{\kappa}{j})=\omx{j}{\gamma}{i}
\end{align*}
for all $i,j  \in \lbrace \mathsf{s},\mathsf{t}\rbrace$.
\end{ex}
\begin{ex}\label{EHopfGrph} [Finite Quiver] For a finite quiver $\Gamma =(V,E)$, we described $B\XA^{1}$ as an extension of the quiver path algebra $\field\Gamma$ in Example \ref{EBiaGrph}. The resulting Hopf algebroid on $\field (V)$ can also be described with relation to the quiver path algebra and additional generators as
$$\field\Gamma\left\langle \ov{f_{p}},\ (\ovl{e_{1}},\ovr{e_{2}}) ,\ (\ovr{e_{1}},\ovl{e_{2}}) \mid p\in S, \ e_{1},e_{2}\in E \right\rangle$$ 
with the relations presented in Example \ref{EBiaGrph} and additional relations
\begin{align*}
 (\ovr{e_{1}},\ovl{e_{2}})\sbt f_{p}\sbt \ov{f_{q}} &= (\ovr{e_{1}},\ovl{e_{2}})\delta_{p,t(e_{1})}\delta_{q,t(e_{2})}
\\ f_{p}\sbt \ov{f_{q}}\sbt(\ovr{e_{1}},\ovl{e_{2}}) &= (\ovr{e_{1}},\ovl{e_{2}})\delta_{p,s(e_{1})}\delta_{q,s(e_{2})}
\\\sum_{e\in E} (\ovl{e_{1}},\ovr{e}) \sbt (\ovr{e_{2}},\ovl{e}) =f_{t(e_{1})} \delta_{e_{1},e_{2}},& \quad \sum_{e\in E} (\ovr{e},\ovl{e_{1}}) \sbt (\ovl{e},\ovr{e_{2}})=\ov{f_{s(e_{1})}} \delta_{e_{1},e_{2}}
\\ \sum_{e\in E} (\ovr{e_{1}},\ovl{e}) \sbt (\ovl{e_{2}},\ovr{e}) = f_{s(e_{1})}\delta_{e_{1},e_{2}},&\quad \sum_{e\in E} (\ovl{e},\ovr{e_{1}}) \sbt (\ovr{e},\ovl{e_{2}})=\ov{f_{t(e_{1})}} \delta_{e_{1},e_{2}}
\end{align*}
for all $e_{1},e_{2}\in E$ and $p,q\in V$. The coproduct and counit of the new generators are given by 
$$ \Delta ( (\ovr{e_{1}},\ovl{e_{2}}))=\sum_{e\in E} (\ovr{e_{1}},\ovl{e})\otimes (\ovr{e},\ovl{e_{2}}),\quad \epsilon ((\ovr{e_{1}},\ovl{e_{2}}))=\delta_{e_{1},e_{2}} f_{s(e_{1})} $$ 
for any $e_{1},e_{2}\in E$. In fact $H\XA^{1}$ admits an antipode since the map $\Upsilon :\XA^{1} \rightarrow A$ defined by $\Upsilon (\ovl{e})=f_{s(e)}-f_{t(e)}$, for $e\in E$, satisfies (\ref{EqSpade}). Translating this data in terms of $\field \Gamma$, the antipode takes the form 
\begin{align*}
S(\ovl{e_{1}}) = -\sum_{e\in E} (\ovr{e},\ovl{e_{1}})\sbt \ovl{e} - \sum_{e\in E} (\ovr{e},\ovl{e_{1}}) - \ov{f_{s(e_{1})}}
\\ S( (\ovr{e_{1}},\ovl{e_{2}}) )= (\ovl{e_{2}},\ovr{e_{1}}), \quad S( (\ovl{e_{1}},\ovr{e_{2}}) )= (\ovr{e_{2}},\ovl{e_{1}})
\end{align*}
for any $e_{1},e_{2}\in E$.
\end{ex}
\begin{ex}\label{EHopfBico} [Bicovariant Calculi] If $A$ is a Hopf algebra and $\Omega^{1}$ a bicovariant calculus over $A$, then as demonstrated in Example \ref{EPivHopf}, $\Omega^{1}$ is free as a left $A$-module so that $\Omega^{1}\otimes_{\field}(\Omega^{1})^{\vee}\cong A^{e}\otimes_{\field} (\Lambda\otimes_{\field} \Lambda^{\star})$ as a left $A^{e}$-module. Hence the Hopf algebroid $H\XA^{1}$ factorises as $A^{e}.T\mathfrak{W}$ where $\mathfrak{W}=\Lambda^{\star}\oplus (\Lambda^{\star}\otimes_{\field} \Lambda )\oplus (\Lambda\otimes_{\field} \Lambda^{\star})$, with the relations present in Example \ref{EBiaHopf} and additional commutation relations
\begin{align*}
(\bm{\lambda}_{j},\bm{f}_{k})\sbt a\ov{b} = a_{(2)}\ov{b_{(2)}}\sbt \big( s^{-1}(a_{(1)}&)\triangleright\bm{\lambda}_{j} , \bm{f}_{j}\triangleleft s^{-1}(b_{(1)})\big)
\\ \sum_{i=1}^{n} (\bm{\lambda}_{i},\bm{f}_{k})\sbt (f_{i},\lambda_{j}) = \delta_{j,k} =&\sum_{i=1}^{n} (f_{j},\lambda_{i})\sbt (\bm{\lambda}_{k},\bm{f}_{i})
\\ \sum_{i=1}^{n} (f_{i},\lambda_{j})\sbt \big(s^{-2}((\lambda_{i})_{(-1)} )\triangleright \bm{(\lambda_{i})_{(0)}} ,\bm{f}_{k}\big) &= f_{k} \big( s^{-2}((\lambda_{j})_{(-1)} )\triangleright (\lambda_{j})_{(0)}\big) 
\\ \sum_{i=1}^{n} (\bm{\lambda}_{j},\bm{f}_{i}) \sbt \big(f_{k},s((\lambda_{i})_{(-1)})\triangleright (\lambda_{i})_{(0)}\big) =&f_{k}\big( s((\lambda_{j})_{(-1)})\triangleright (\lambda_{j})_{(0)} \big)
\end{align*}
for all $1\leq j,k\leq n$ and $a,b\in A$. The coproduct and counit of $B\XA^{1}$ extend to $H\XA^{1}$ by 
$$\Delta ( (\bm{\lambda}_{i},\bm{f}_{j}) )= \sum_{k=1}^{n} (\bm{\lambda}_{i},\bm{f}_{k})\otimes (\bm{\lambda}_{k},\bm{f}_{j}) \quad \epsilon ( (\bm{\lambda}_{i},\bm{f}_{j})) =\delta_{i,j}$$ 
for all $1\leq i,j\leq n$.
\end{ex}
\begin{ex}\label{ExHopfD6} [$\mathbb{C}D_{6}$] For the differential calculus of Example \ref{EBiaD6}, the left coaction on $\Lambda$ is trivial. Hence, the Hopf algebroid $H\XA^{1}$ over the group algebra $\mathbb{C}D_{6}$, factorises as $A^{e}.\mathbb{C}\langle f_{i},\omx{i}{\gamma}{j},\omx{i}{\kappa}{j}\mid i,j\in \lbrace \xi,\tau\rbrace\rangle$ with the relations of $B\XA^{1}$ as presented in Example \ref{EBiaD6} and additional relations
\begin{align*} 
\omx{\xi}{\kappa}{i}\sbt a = \frac{1}{2} (\omx{\xi}{\kappa}{i} +\sqrt{3}\omx{\tau}{\kappa}{i}), &\quad \omx{\tau}{\kappa}{i}\sbt a = \frac{1}{2} (-\sqrt{3}\omx{\xi}{\kappa}{i} +\omx{\tau}{\kappa}{i}) 
\\\omx{i}{\kappa}{\xi}\sbt \ov{a} = \frac{1}{2} (\omx{i}{\kappa}{\xi} -\sqrt{3}\omx{i}{\kappa}{\tau}), &\quad \omx{i}{\kappa}{\tau}\sbt \ov{a} = \frac{1}{2} (\sqrt{3}\omx{i}{\kappa}{\xi} +\omx{i}{\kappa}{\tau})  
\\ \omx{\xi}{\kappa}{i}\sbt b =\omx{\xi}{\kappa}{i}, \quad \omx{\tau}{\kappa}{i}\sbt b =-\omx{\tau}{\kappa}{i} ,&\quad  \omx{i}{\kappa}{\xi}\sbt \ov{b}= \omx{i}{\kappa}{\xi},\quad  \omx{i}{\kappa}{\tau}\sbt \ov{b}= -\omx{i}{\kappa}{\tau}
\\\omx{\xi}{\gamma}{j}\sbt \omx{\xi}{\kappa}{j}+ \omx{\tau}{\gamma}{j}\sbt \omx{\tau}{\kappa}{j}=&\delta_{i,j}= \omx{i}{\gamma}{\xi}\sbt \omx{j}{\kappa}{\xi}+ \omx{i}{\gamma}{\tau}\sbt \omx{j}{\kappa}{\tau}
\\ \omx{\xi}{\kappa}{j}\sbt \omx{\xi}{\gamma}{j}+ \omx{\tau}{\kappa}{j}\sbt \omx{\tau}{\gamma}{j}=&\delta_{i,j}= \omx{i}{\kappa}{\xi}\sbt \omx{j}{\gamma}{\xi}+ \omx{i}{\kappa}{\tau}\sbt \omx{j}{\gamma}{\tau}
\end{align*}
for $i,j\in \lbrace \xi,\tau\rbrace$. The coproduct and counit extend as
\begin{align*}
\Delta ( \omx{i}{\kappa}{j}) =\omx{i}{\kappa}{\xi}\otimes \omx{\xi}{\kappa}{j}+\omx{i}{\kappa}{\tau}\otimes \omx{\tau}{\kappa}{j},&\quad \epsilon (\omx{i}{\kappa}{j})=\delta_{i,j}
\end{align*}
for $i,j\in \lbrace \xi,\tau\rbrace$.
\end{ex}
\section{Flat Bimodule Connections}\label{SFlat}
Classically, the curvature on connections is defined using the Lie bracket on vector fields or alternatively, the exterior derivative from the space of 1-forms to the space of 2-forms. In this section, we assume $d:A\rightarrow \Omega^{1}$ is part of a dga $\Omega^{\sbt}$. However, we only require the bimodule $\Omega^{2}$ and linear maps $d:\Omega^{1}\rightarrow \Omega^{2}$ and $\wedge: \Omega^{1} \otimes\Omega^{1} \rightarrow \Omega^{2}$, satisfying the relevant properties, as additional data. We briefly recall the definitions of curvature, flat connections and the sheaf of differential operators from \cite{beggs2019quantum}.

If $(M,\cnc )$ is a left connection, then the \emph{curvature} of $\cnc$ is a map $R_{M}:M\rightarrow \Omega^{2}\otimes M$ defined by 
$$R_{M}=(d \otimes \mathrm{id}_{ M} - \mathrm{id}_{\Omega^{1}} \wedge \cnc ) \cnc $$
We say $(M,\cnc)$ is a \emph{flat left connection}, if $R_{M}=0$ and denote the subcategory of flat left connections in $\prescript{}{A}{\ct{E}}$, by $\prescript{}{A}{\ct{F}}$. 

In Chapter 6 of \cite{beggs2019quantum}, the category $\prescript{}{A}{\ct{F}}$ is shown to be isomorphic to the category of modules over an algebra, $\ct{D}_{A}$, when $\Omega^{2}$ is right fgp. We denote the left dual bimodule of $\Omega^{2}$ by $\XA^{2}$ and denote the respective coevaluation and evaluation maps by $\fr{coev}$ and $\fr{ev}$ and denote $\fr{coev}(1)= \sum_{i} x^{2}_{i}\otimes\omega^{2}_{i}$. In this case, $\prescript{}{A}{\ct{F}}\cong \prescript{}{\ct{D}_{A}}{\ct{M}}$, where $\ct{D}_{A}$ is the algebra obtained as the quotient of $T\XA^{1}_{\sbt}$ by the ideal generated by relations
\begin{equation}\label{EqFlat}
 \fr{ev}(x^{2}\otimes d\omega_{i})\sbt x_{i} - \mathfrak{ev}(x^{2}\otimes \omega_{j}\wedge\omega_{k}) \sbt x_{k}\sbt x_{j} =0
\end{equation}
for all $x^{2}\in \XA^{2}$. It is easy to verify that the mentioned ideal annihilating a $T\XA^{1}_{\sbt}$-module is equivalent to the induced connection on the module being flat [Corollary 6.24 \cite{beggs2019quantum}].

Although one could quotient out the algebra $B\XA^{1}$ by the same relations, (\ref{EqFlat}), and discuss bimodule connection which have a flat left connection, the tensor product of two such connections will not have zero curvature. To discuss a monoidal category of flat bimodule connections, we must assume the bimodule connections are \emph{extendable}. We say an $\Omega^{1}$-intertwining map $\sigma :M\otimes\Omega^{1} \rightarrow \Omega^{1} \otimes M$ is extendable if there exists an $\Omega^{2}$-intertwining map $\sigma_{2}:M\otimes\Omega^{2} \rightarrow \Omega^{2} \otimes M$ such that the equation
\begin{equation}\label{EqExtend}
(\wedge \otimes \mathrm{id}_{M})(\mathrm{id}_{\Omega^{1}}\otimes \sigma )(\sigma \otimes \mathrm{id}_{\Omega^{1}}) = \sigma_{2}(\mathrm{id}_{M}\otimes \wedge )
\end{equation}
holds as an equality of bimodule morphisms with domain $M\otimes \Omega^{1}\otimes\Omega^{1}$ and codomain $\Omega^{2}\otimes M$. An additional condition is required when the calculus is not surjective. The equation
\begin{equation}\label{EqFlExtraa}
(\wedge \otimes M)[ (\mathrm{id}_{\Omega^{1}}\otimes\sigma)(\cnc \otimes \mathrm{id}_{\Omega^{1}}) +( \mathrm{id}_{\Omega^{1}}\otimes \cnc)\sigma]= (d \otimes \mathrm{id}_{M})\sigma - \sigma_{2}(\mathrm{id}_{M}\otimes d)
\end{equation} 
must hold for linear maps with domain $M\otimes \Omega^{1}$ and codomain $\Omega^{2}\otimes M$. This condition appears implicitly in Lemma 4.12 of \cite{beggs2019quantum} and is said to be equivalent to the curvature being a right module morphism. However, if the calculus is not surjective, this is an additional condition. The subcategory of left bimodule connections which are flat, extendable and satisfy condition (\ref{EqFlExtraa}), is a monoidal subcategory of $\biml$ and is denoted by $\prescript{l}{A}{\mathcal{F}}_{A}$. This is discussed in Section 4.5.1 of \cite{beggs2019quantum}. To obtain the bialgebroid whose category of modules is isomorphic to $\prescript{l}{A}{\mathcal{F}}_{A}$, we must adjoin additional generators of the form $\XA^{2}\otimes_{\field}\Omega^{2}$ to $B\XA^{1}$, to induce $\Omega^{2}$-intertwinings and quotient out the corresponding relations for flatness (\ref{EqFlat}), extendability (\ref{EqExt1}) and the additional condition (\ref{EqFlExt}). However, the category $\prescript{l}{A}{\mathcal{F}}_{A}$ will again not lift the closed structure of $\bim$. Instead, we will look at the relevant closed monoidal subcategory of flat bimodule connections in $\prescript{}{H\XA^{1}}{\ct{M}} $, and the construction of the relevant bialgebroid for $\prescript{l}{A}{\mathcal{F}}_{A}$ will also be implicitly present in our work. 

\subsection{Hopf Algebroid $\ct{D}\XA$ in Flat Case}\label{SDXA}
The closed subcategory of flat bimodule connections with extendable $\Omega^{1}$-intertwining which we would like to consider, should lift the closed structure of $\bim$. Since the extendability condition adds an underlying $\Omega^{2}$-intertwining to our connection, the underlying $\Omega^{2}$-intertwining of such bimodules must belong to the appropriate closed subcategory of $\Omega^{2} $-intertwinings. Hence, as for $\Omega^{1}$-intertwinings in Section \ref{SPivHpf}, we assume $\Omega^{2}$ is left and right fgp and pivotal as an $A$-bimodule. We denote the relevant coevaluation and evaluation maps between $\Omega^{2}$ and its left and right dual $\XA^{2}$, by $\fr{coev}$, $\underline{\fr{coev}}$, $\mathfrak{ev}$ and $\underline{\mathfrak{ev}}$. We utilise the following notation $\fr{coev}(1)= \sum_{i}x^{2}_{i}\otimes \omega^{2}_{i} \in \XA^{2}\otimes\Omega^{2}$ and $\underline{\fr{coev}}(1)= \sum_{i}\rho^{2}_{i}\otimes y^{2}_{i} \in \Omega^{2}\otimes \XA^{2}$.

Additionally, we require $\wedge$ to be a \emph{pivotal bimodule morphism} i.e. for any $x^{2}\in \XA^{2}$, the equation 
\begin{equation}\label{EqWdgPiv}
\mathfrak{ev} (x^{2} \otimes \omega_{i}\wedge \omega_{j}) x_{j}\otimes x_{i} = y_{i}\otimes y_{j} \underline{\mathfrak{ev}}(\rho_{j}\wedge \rho_{i} \otimes x^{2} ) 
\end{equation}
holds for elements of $\XA^{1}\otimes\XA^{1}$. Since both $\Omega^{1}$ and $\Omega^{2}$ are both pivotal, $\wedge$ provides two bimodule morphisms from $\XA^{2}$ to $ \XA^{1}\otimes\XA^{1}$, presented on either side of the equation above, and condition (\ref{EqWdgPiv}) requires these two bimodule morphisms to be equal. 

We note that the free product of two Hopf algebroids over an algbera $A$, as $A^{e}$-algebras will again be a Hopf algebroid over $A$. Since modules over the free product are just $A$-bimodules with additional actions of each algebra, the action of both Hopf algebroids on tensor products and inner homs simply lift to the category of modules over the free product. Hence, We obtain a new Hopf algebroid by considering the free product of $A\otimes_{\field}A^{op}$-algebras $H\XA^{1}$ and $H(\Omega^{2})$ and denote it by $F$. We define $\ct{D}\XA$ as the quotient of $F$ by the ideal generated by relations (\ref{EqFlat}) and 
\begin{align}
\mathfrak{ev}( x^{2}\otimes \omega_{i}\wedge \omega_{j}) [x_{j}\sbt(x_{i},\omega)+(x_{j},\omega)\sbt x_{i}] =& \mathfrak{ev}(x^{2}\otimes d\omega_{i} )(x_{i},\omega ) - (x^{2},d\omega) \label{EqFlExt}
\\ \mathfrak{ev}(x^{2}\otimes \omega_{i}\wedge \omega_{j}) \sbt (x_{j},\rho)\sbt(x_{i},\omega)=& (x^{2},\omega\wedge \rho) \label{EqExt1}
\\ \ov{\underline{\mathfrak{ev}}(\rho_{i}\wedge \rho_{j}\otimes x^{2})}\sbt(\omega ,y_{i})\sbt(\rho, y_{j} ) =&(\omega\wedge\rho , x^{2})\label{EqExt2}
\end{align}
for all $x^{2}\in\XA^{2}$ and $\omega, \rho\in \Omega^{1}$. 

Firstly, note that an $F$-module $M$ is an $A$-bimodules with a left bimodule connection $(M,\cnc,\sigma)$, so that $(M,\sigma)$ lies in $\prescript{\XA^{1}}{A}{I\mathcal{M}}^{\Omega^{1}}_{A}$ and an invertible $\Omega^{2}$-intertwining $\sigma_{2}$ such that $(M,\sigma_{2})$ lies in $\prescript{\XA^{2}}{A}{I\mathcal{M}}^{\Omega^{2}}_{A}$. By constructing $\sigma$ and $\sigma^{2}$ for an $F$-module, as described in Theorem \ref{TBOmga}, we can deduce that the annihilation of the module by relations (\ref{EqExt1}) and (\ref{EqExt2}) is equivalent to $\sigma$ and $\sigma^{-1}$ extending to $\sigma_{2}$ and $\sigma_{2}^{-1}$, respectively. Since $\wedge$ is a pivotal morphism and (\ref{EqWdgPiv}) holds, relations (\ref{EqExt1}) and (\ref{EqExt2}) are equivalent to relations 
\begin{align}
  (y_{i},\rho)\sbt(y_{j},\omega)\sbt\underline{\mathfrak{ev}}(\rho_{j}\wedge \rho_{i} \otimes x^{2} ) =& (x^{2},\omega\wedge \rho) \label{EqExt3}
\\ (\omega ,x_{i})\sbt(\rho, x_{j} )\sbt\ov{ \mathfrak{ev} (x^{2} \otimes \omega_{i}\wedge \omega_{j}) } =&(\omega\wedge\rho , x^{2})\label{EqExt4}
\end{align}
respectively. Recall that any $H(\Omega^{1})$-module has a pair of induced $\XA^{1}$-intertwinings, (\ref{EqXsig}) and (\ref{EqsigX}), which are inverses. Relations (\ref{EqExt3}) and (\ref{EqExt4}) annihilating an $F$-module, are equivalent to the induced $\XA^{1}$-intertwinings on the module, extending to the corresponding $\XA^{2}$-intertwinings. 

Relation (\ref{EqFlExt}) annihilating an $F$-module, is equivalent to the induced bimodule connection and intertwinings of the $F$-module satisfying the additional condition (\ref{EqFlExtraa}). We previously noted that, when the calculus in question is surjective i.e. $\Omega^{1}$ is generated by elements of the form $bda$ where $a,b\in A$, then relation (\ref{EqFlExt}) follows from (\ref{EqFlat}) and (\ref{EqExt1}):
\begin{align*}
 0=&0\sbt \ov{ba}= \big(\mathfrak{ev}(x^{2}\otimes d\omega_{i})\sbt x_{i} - \mathfrak{ev}(x^{2}\otimes \omega_{j}\wedge\omega_{k}) \sbt x_{k}\sbt x_{j} \big)\sbt \ov{ba}
\\ =& \big[\ov{a}\mathfrak{ev}(x^{2}\otimes  d\omega_{i})\sbt x_{i} + \mathfrak{ev}(x^{2}\otimes  d\omega_{i})\sbt ( x_{i},da) - \ov{a}\mathfrak{ev}(x^{2}\otimes \omega_{j}\wedge\omega_{k}) \sbt x_{k}\sbt x_{j}
 \\   &-\mathfrak{ev}(x^{2}\otimes \omega_{j}\wedge\omega_{k}) \sbt [x_{k}\sbt (x_{j},da) + (x_{k},da)\sbt x_{j}] \big]\sbt \ov{b}  
 \\ =& \big[\mathfrak{ev}(x^{2}\otimes  d\omega_{i}) ( x_{i},da)-\mathfrak{ev}(x^{2}\otimes \omega_{j}\wedge\omega_{k})  [x_{k}\sbt (x_{j},da) + (x_{k},da)\sbt x_{j}] \big]\sbt \ov{b}
 \\   =& \mathfrak{ev}(x^{2}\otimes  d\omega_{i}) ( x_{i},bda)-\mathfrak{ev}(x^{2}\otimes \omega_{j}\wedge\omega_{k})  [x_{k}\sbt (x_{j},bda) + (x_{k},bda)\sbt x_{j}] 
 \\ &+ \mathfrak{ev}(x^{2}\otimes \omega_{j}\wedge\omega_{k})(x_{k},da)\sbt (x_{j},db) 
 \\ &\Longrightarrow \mathfrak{ev}(x^{2}\otimes \omega_{j}\wedge\omega_{k})  [x_{k}\sbt (x_{j},bda) + (x_{k},bda)\sbt x_{j}]  
 \\ &\hspace{3cm} = \mathfrak{ev}(x^{2}\otimes  d\omega_{i}) ( x_{i},bda)- (x^{2}, db\wedge da )
\end{align*}
for any $x^{2}\in \XA^{2}$ and $a,b\in A$.
\begin{rmk}\label{RWdSplit} If $\wedge : \Omega^{1} \otimes\Omega^{1} \rightarrow \Omega^{2}$ splits as a bimodule map, we do not need to add additional generators to $H\XA^{1}$ to capture the intertwining map extending. In other words, when $\wedge$ is surjective, the relations imposed in $\ct{D}\XA$, describe the additional generators of $H(\Omega^{2})$ in terms of elements of $H\XA^{1}$. Additionally when $\wedge$ splits, the extendabillity conditions would simply be equivalent to relations
$$\mathfrak{ev}(x^{2}\otimes \omega_{i}\wedge \omega_{j}) \sbt (x_{j},\rho)\sbt(x_{i},\omega)=0=\ov{\underline{\mathfrak{ev}}(\rho_{i}\wedge \rho_{j}\otimes x^{2})}\sbt(\omega ,y_{i})\sbt(\rho, y_{j} ) $$ 
on $H\XA^{1}$, for all $\omega\wedge \rho\in \mathrm{ker} (\wedge)$. 
\end{rmk}

\textbf{Notation.} From this point onwards, whenever the action of elements of $A$ and $A^{op}$ agrees with a module action on elements in the algebras constructed, we avoid writing $\sbt$ for brevity. For example, for elements $a\in A$, $x\in\XA^{1}$ and $\omega\in \Omega^{1}$, we simply write $ax$ and $a(x,\omega)$ instead of $a\sbt x$ and $a\sbt (x,\omega) $, respectively.

\begin{thm}\label{TDXAbia} The algebra $\ct{D}\XA$ inherits the bialgebroid structure of $F$. 
\end{thm}
\begin{proof} We only need to check that the comultiplication and counit of $F$ are well defined on its quotient $\ct{D}\XA$. We first look at the comultiplication and the extendibility relations. Let $x^{2}\in \XA^{2}$ and $\omega,\rho\in \Omega^{1}$ and consider relation (\ref{EqExt1}):
\begin{align*}
\Delta \big(\mathfrak{ev}&(x^{2}\otimes \omega_{i}\wedge \omega_{j}) \sbt (x_{j},\rho)\sbt(x_{i},\omega)\big) 
\\=&\mathfrak{ev}(x^{2}\otimes \omega_{i}\wedge \omega_{j}) \sbt (x_{j},\omega_{l})\sbt(x_{i},\omega_{t})\otimes (x_{l},\rho)\sbt(x_{t},\omega)
\\=& (x^{2}, \omega_{t}\wedge \omega_{l})\otimes (x_{l},\rho)\sbt(x_{t},\omega)= (x^{2},\omega_{i}^{2})\otimes  \mathfrak{ev}(x_{i}^{2}\otimes \omega_{t}\wedge \omega_{l})(x_{l},\rho)\sbt(x_{t},\omega)
\\=&(x^{2},\omega_{i}^{2})\otimes (x_{i}^{2},\omega\wedge \rho)= \Delta\big((x^{2},\omega\wedge \rho)\big)
\end{align*}
The computations for relation (\ref{EqExt2}) are completely symmetric and are left to the reader. We now look at relation (\ref{EqFlat}) and see that the additional condition (\ref{EqFlExt}) is essential for the comultiplication to be well-defined for flat bimodule connections: 
\begin{align*}
\Delta \big(  \mathfrak{ev}(&x^{2}\otimes \omega_{j}\wedge\omega_{k}) \sbt x_{k}\sbt x_{j}\big) = \mathfrak{ev}(x^{2}\otimes \omega_{j}\wedge\omega_{k}) \big[  x_{k}\sbt x_{j}\otimes 1 +  
\\ &+ x_{k}\sbt (x_{j},\omega_{l})\otimes x_{l}+ (x_{k},\omega_{l})\sbt x_{j}\otimes x_{l} +  (x_{k},\omega_{l})\sbt (x_{j},\omega_{m})\otimes x_{l}\sbt x_{m}\big]
\\ =& \mathfrak{ev}(x^{2}\otimes d\omega_{i})\sbt x_{i}\otimes 1 +(x^{2}, \omega_{m}\wedge\omega_{l})\otimes x_{l}\sbt x_{m}
\\ &+ \big[\mathfrak{ev}(x^{2}\otimes d\omega_{i} )(x_{i},\omega_{l} ) - (x^{2},d\omega_{l}) \big]\otimes x_{l}
\\=&\Delta \big( \mathfrak{ev}(x^{2}\otimes d\omega_{i})x_{i}\big) + (x^{2}, \omega^{2}_{t})\otimes \big[ \mathfrak{ev}(x_{t}^{2}\otimes \omega_{m}\wedge\omega_{l})x_{l}\sbt x_{m} 
\\ & - \mathfrak{ev}(x_{t}^{2}\otimes d\omega_{l}) x_{l}\big]= \Delta \big( \mathfrak{ev}(x^{2}\otimes d\omega_{i})x_{i}\big)
\end{align*}
where $x^{2}\in\XA^{2}$. To check relation (\ref{EqFlat}) itself, let $x^{2}\in\XA^{2}$ and $\omega\in \Omega^{1}$: 
\begin{align*}
\Delta \big(&\mathfrak{ev}(x^{2}\otimes \omega_{i}\wedge \omega_{j}) [x_{j}\sbt(x_{i},\omega)+(x_{j},\omega)\sbt x_{i}]\big) 
\\ =& \mathfrak{ev}(x^{2}\otimes \omega_{i}\wedge \omega_{j})\big[ x_{j}\sbt (x_{i},\omega_{t})\otimes (x_{t},\omega) +  (x_{j},\omega_{l})\sbt (x_{i},\omega_{t})\otimes x_{l}\sbt (x_{t},\omega)
\\ &+ (x_{j},\omega_{l})\sbt x_{i}\otimes (x_{l},\omega) + (x_{j},\omega_{l})\sbt (x_{i},\omega_{t})\otimes (x_{l},\omega)\sbt x_{t} \big] 
\\ =&(x^{2}, \omega_{t}\wedge \omega_{l}) \otimes \big[x_{l}\sbt (x_{t},\omega)+(x_{l},\omega)\sbt x_{t} \big] +\big[\mathfrak{ev}(x^{2}\otimes d\omega_{i} )(x_{i},\omega_{t} ) 
\\ &- (x^{2},d\omega_{t}) \big]\otimes (x_{t},\omega) = (x^{2}, \omega^{2}_{l}) \otimes\big[ \mathfrak{ev}(x_{l}^{2}\otimes d\omega_{i} )(x_{i},\omega ) - (x_{l}^{2},d\omega)\big] 
\\& +\big[\mathfrak{ev}(x^{2}\otimes d\omega_{i} )(x_{i},\omega_{t} ) - (x^{2},d\omega_{t}) \big]\otimes (x_{t},\omega)
\\=&\Delta\big( \mathfrak{ev}(x^{2}\otimes d\omega_{i} )(x_{i},\omega ) - (x^{2},d\omega)\big)
\end{align*}
For the counit to be well-defined, all computations follow in a straightforward manner. Let $x^{2}\in\XA^{2}$ and $\omega, \rho\in \Omega^{1}$:
\begin{align*}
\epsilon &\big( \mathfrak{ev}(x^{2}\otimes d\omega_{i})\sbt x_{i}\big) =0 =\epsilon\big( \mathfrak{ev}(x^{2}\otimes \omega_{j}\wedge\omega_{k}) \sbt x_{k}\sbt x_{j}\big)
\\\epsilon&\big (\mathfrak{ev}( x^{2}\otimes \omega_{i}\wedge \omega_{j}) [x_{j}\sbt(x_{i},\omega)+(x_{j},\omega)\sbt x_{i}]\big) 
\\ &=  \mathfrak{ev}( x^{2}\otimes \omega_{i}\wedge \omega_{j}) \mathrm{ev}(x_{j}\otimes d \mathrm{ev}(x_{i}\otimes\omega)) + 0 
\\& = \mathfrak{ev}( x^{2}\otimes \omega_{i}\wedge d \mathrm{ev}(x_{i}\otimes\omega)) = -\mathfrak{ev}( x^{2}\otimes (d\omega_{i}) \mathrm{ev}(x_{i}\otimes\omega))+ \mathfrak{ev}(x^{2}\otimes d\omega)
\\&= \epsilon \big( \mathfrak{ev}(x^{2}\otimes d\omega_{i} )(x_{i},\omega ) - (x^{2},d\omega)\big)
\\\epsilon & \big(\ov{\underline{\mathfrak{ev}}(\rho_{i}\wedge \rho_{j}\otimes x^{2})}\sbt(\omega ,y_{i})\sbt(\rho, y_{j} )\big) =  \underline{ev}(\omega \underline{\mathrm{ev}} (\rho\otimes y_{j} )\otimes y_{i})\underline{\mathfrak{ev}}(\rho_{i}\wedge \rho_{j}\otimes x^{2})
\\& = \underline{\mathfrak{ev}} (\omega\wedge\rho \otimes x^{2})=\epsilon \big( (\omega\wedge\rho , x^{2})\big)
\\\epsilon &\big( \mathfrak{ev}(x^{2}\otimes \omega_{i}\wedge \omega_{j}) \sbt (x_{j},\rho)\sbt(x_{i},\omega)\big)=  \mathfrak{ev}(x^{2}\otimes \omega_{i}\wedge \omega_{j})\mathrm{ev}(x_{j}\otimes \mathrm{ev}(x_{i}\otimes\omega)\rho)
\\& = \mathfrak{ev}(x^{2},\omega\wedge \rho)= \epsilon \big( (x^{2},\omega\wedge \rho)\big) \qedhere
\end{align*}\end{proof}
To prove that $\ct{D}\XA$ has a Hopf algebroid structure we need to describe some additional nontrivial relations which hold in $\ct{D}\XA$.
\begin{lemma}\label{LRel} The following additional relations hold in $\ct{D}\XA$:
\begin{equation}\label{EqI}
\ov{\underline{\mathfrak{ev}}(d\rho_{j}\otimes x^{2}) }(\omega_{i},y_{j})\sbt x_{i} + \ov{\underline{\mathfrak{ev}}(\rho_{m}\wedge \rho_{n}\otimes x^{2})}(\omega_{k},y_{m})\sbt x_{k}\sbt (\omega_{l},y_{n})\sbt x_{l}=0
\end{equation}
\begin{align}\label{EqII} \ov{\underline{\mathfrak{ev}}(\rho_{m}\wedge \rho_{n}\otimes x^{2})}[ (\omega_{t},y_{m})\sbt x_{t}\sbt (\omega,y_{n})&+ (\omega,y_{m})\sbt(\omega_{l},y_{n})\sbt x_{l}] 
\\ =&(d\omega ,x^{2}) -  \ov{\underline{\mathfrak{ev}}(d\rho_{i} \otimes x^{2})}(\omega, y_{i})\nonumber
\end{align}
\begin{align}\label{EqIII}(\omega , &x_{i})\sbt y_{t}\sbt (\rho_{t}, x_{j})\ov{\mathfrak{ev} (x^{2}\otimes \omega_{i}\wedge \omega_{j}) } + (\omega, x_{i}) \ov{\mathfrak{ev}(x^{2}\otimes d\omega_{i})}
\\ =& \underline{\mathrm{ev}}(\omega\otimes \mathfrak{ev}(y_{i}^{2}\otimes d\omega_{l})x_{l})(\rho_{i}^{2},x^{2})- \underline{\mathrm{ev}}(\omega\otimes \mathfrak{ev}(y_{i}^{2}\otimes \omega_{j}\wedge \omega_{k})x_{k})x_{j}\sbt(\rho_{i}^{2},x^{2}) \nonumber
\\ &+\underline{\mathrm{ev}}\big( \omega \otimes \mathrm{ev}[y_{t} \otimes d \underline{\mathrm{ev}} (\rho_{t} \otimes \mathfrak{ev} (y_{i}^{2}\otimes \omega_{j}\wedge \omega_{k})x_{k})] x_{j} \big)(\rho_{i}^{2},x^{2})\nonumber 
\end{align}
for all $x^{2}\in \XA^{2}$ and $\omega\in\Omega^{1}$.
\end{lemma} 
\begin{proof} Let $x^{2}\in \XA^{1}$. We prove identity (\ref{EqI}) holds, using relations (\ref{EqRelInv1}) in $H(\Omega^{2})$, (\ref{EqExt1}), (\ref{EqRelInv1}) in $H(\Omega^{1})$, (\ref{EqFlExt}) and (\ref{EqFlat}), respectively: 
\begin{align*}
&\ov{\underline{\mathfrak{ev}}(d\rho_{j}\otimes x^{2}) }(\omega_{i},y_{j})\sbt x_{i} + \ov{\underline{\mathfrak{ev}}(\rho_{m}\wedge \rho_{n}\otimes x^{2})}(\omega_{k},y_{m})\sbt x_{k}\sbt (\omega_{l},y_{n})\sbt x_{l}
\\& =(\omega_{t}^{2},x^{2} )\sbt\big[ (x_{t}^{2},d\rho_{j})\sbt(\omega_{i},y_{j})\sbt x_{i} + (x_{t}^{2},\rho_{m}\wedge \rho_{n})\sbt (\omega_{k},y_{m})\sbt x_{k}\sbt (\omega_{l},y_{n})\sbt x_{l}\big]
\\& =(\omega_{t}^{2},x^{2} )\sbt\big[  \mathfrak{ev}(x_{t}^{2}\otimes \omega_{i}\wedge \omega_{j}) (x_{j}, \rho_{n})\sbt (x_{i},\rho_{m}) \sbt (\omega_{k},y_{m})\sbt x_{k}
\\ & \quad +(x_{t}^{2},d\rho_{n}) \big]\sbt(\omega_{l},y_{n})\sbt x_{l}
\\&= (\omega_{t}^{2},x^{2} )\sbt\big[  \mathfrak{ev}(x_{t}^{2}\otimes \omega_{i}\wedge \omega_{j}) (x_{j}, \rho_{n})\sbt x_{i}+(x_{t}^{2},d\rho_{n}) \big]\sbt(\omega_{l},y_{n})\sbt x_{l}
\\&= (\omega_{t}^{2},x^{2} )\sbt\big[ \mathfrak{ev}(x_{t}^{2}\otimes d\omega_{t})(x_{t},\rho_{n}) -\mathfrak{ev}(x_{t}^{2}\otimes \omega_{i}\wedge \omega_{j}) x_{j}\sbt (x_{i}, \rho_{n}) \big]\sbt(\omega_{l},y_{n})\sbt x_{l}
\\ & = (\omega_{t}^{2},x^{2} )\sbt\big[\mathfrak{ev}(x_{t}^{2}\otimes d\omega_{t})x_{t} + \mathfrak{ev}(x_{t}^{2}\otimes \omega_{i}\wedge \omega_{j}) x_{j}\sbt x_{i}\big]=0
\end{align*}
Let $\omega\in\Omega^{1}$. Identity (\ref{EqII}) follows from relations (\ref{EqRelInv1}) in $H(\Omega^{2})$ and (\ref{EqExt1}), (\ref{EqRelInv1}) in $H(\Omega^{1})$ and (\ref{EqFlExt}):
\begin{align*}
&\ov{\underline{\mathfrak{ev}}(\rho_{m}\wedge \rho_{n}\otimes x^{2})}[ (\omega_{t},y_{m})\sbt x_{t}\sbt (\omega,y_{n})+ (\omega,y_{m})\sbt(\omega_{l},y_{n})\sbt x_{l}] 
\\ &=(\omega^{2}_{i},x^{2})\mathfrak{ev}(x^{2}_{i}\otimes \omega_{j}\wedge\omega_{k})(x_{k},\rho_{n})\sbt (x_{j},\rho_{m})\sbt \big[(\omega_{t},y_{m})\sbt x_{t}\sbt (\omega,y_{n})
\\ &\quad + (\omega,y_{m})\sbt(\omega_{l},y_{n})\sbt x_{l}]
\\ &= (\omega^{2}_{i},x^{2})\mathfrak{ev}(x^{2}_{i}\otimes \omega_{j}\wedge\omega_{k})\big[(x_{k},\rho_{n})\sbt x_{j}\sbt (\omega ,y_{n}) + x_{k}\mathrm{ev}(x_{j}\otimes \omega)\big]
\\ &=(\omega^{2}_{i},x^{2})\mathfrak{ev}(x^{2}_{i}\otimes \omega_{j}\wedge\omega_{k})\big[(x_{k},\rho_{n})\sbt x_{j}\sbt (\omega ,y_{n}) + x_{k}\sbt \mathrm{ev}(x_{j}\otimes \omega) 
\\&\quad - \mathrm{ev}(x_{k}\otimes d\mathrm{ev}(x_{j}\otimes \omega)) \big]
\\& =(\omega^{2}_{i},x^{2})\mathfrak{ev}(x^{2}_{i}\otimes \omega_{j}\wedge\omega_{k})\big[(x_{k},\rho_{n})\sbt x_{j}+ x_{k}\sbt (x_{j},\rho_{n})\big]\sbt (\omega , y_{n}) 
\\ &\quad - (\omega^{2}_{i},x^{2})\mathfrak{ev}(x^{2}_{i}\otimes \omega_{j}\wedge\omega_{k})\mathrm{ev}(x_{k}\otimes d\mathrm{ev}(x_{j}\otimes \omega))
\\&=(\omega^{2}_{i},x^{2})\big[\mathfrak{ev}(x^{2}_{i}\otimes d\omega_{l}) (x_{l},\rho_{n})- (x^{2}_{i},d\rho_{n}) \big]\sbt (\omega , y_{n}) 
\\ &\quad - (\omega^{2}_{i},x^{2})\mathfrak{ev}(x^{2}_{i}\otimes \omega_{j}\wedge d\mathrm{ev}(x_{j}\otimes \omega)) = -\ov{\underline{\mathfrak{ev}}(x^{2}\otimes d\rho_{n})} (\omega ,y_{n}) 
\\ &\quad + (\omega^{2}_{i},x^{2})\sbt \big[\mathfrak{ev}(x^{2}_{i}\otimes (d\omega_{l})\mathrm{ev}(x_{l},\omega) ) - \mathfrak{ev}(x^{2}_{i}\otimes \omega_{l}\wedge d\mathrm{ev}(x_{l}\otimes \omega)) \big] 
\\&= -\ov{\underline{\mathfrak{ev}}(x^{2}\otimes d\rho_{n})} (\omega ,y_{n}) + (d\omega, x^{2})
\end{align*}
We prove identity (\ref{EqIII}) by a similar manipulation, using relations (\ref{EqRelHpf1}) in $H(\Omega^{2})$, (\ref{EqExt1}), (\ref{EqRelHpf2}) in $H(\Omega^{1})$ and (\ref{EqFlExt}):
\begin{align*}
(\omega &, x_{i})\sbt y_{t}\sbt (\rho_{t}, x_{j})\ov{\mathfrak{ev} (x^{2}\otimes \omega_{i}\wedge \omega_{j}) } + (\omega, x_{i}) \ov{\mathfrak{ev}(x^{2}\otimes d\omega_{i})}
\\= &(\omega , x_{i})\sbt y_{t}\sbt (\rho_{t}, x_{j})\sbt (y^{2}_{l},\omega_{i}\wedge \omega_{j})\sbt (\rho_{l}^{2},x^{2}) +(\omega, x_{i})\sbt (y^{2}_{l},d\omega_{i})\sbt (\rho_{l}^{2},x^{2})
\\=& (\omega , x_{i})\sbt y_{t}\sbt (\rho_{t}, x_{j})\fr{ev}(y^{2}_{l}\otimes\omega_{m}\wedge \omega_{n})(x_{n},\omega_{j})\sbt(x_{m},\omega_{i}) \sbt(\rho_{l}^{2},x^{2})
\\&+ (\omega , x_{i})\sbt(y^{2}_{l},d\omega_{i}) \sbt(\rho_{l}^{2},x^{2})
\\=& (\omega , x_{i})\sbt\big[y_{t}\sbt \und{\mathrm{ev}}\big(\rho_{t}\otimes \fr{ev}(y^{2}_{l}\otimes\omega_{m}\wedge \omega_{n})x_{n}\big)\sbt(x_{m},\omega_{i})+ (y^{2}_{l},d\omega_{i})\big]\sbt (\rho_{l}^{2},x^{2})
\\=&(\omega , x_{i})\sbt\big[\fr{ev}(y^{2}_{l}\otimes\omega_{m}\wedge \omega_{n})x_{n}\sbt(x_{m},\omega_{i})+ (y^{2}_{l},d\omega_{i})\big]\sbt (\rho_{l}^{2},x^{2})
\\ &+(\omega , x_{i})\mathrm{ev}\big( y_{t}\otimes d \und{\mathrm{ev}}\big(\rho_{t}\otimes \fr{ev}(y^{2}_{l}\otimes\omega_{m}\wedge \omega_{n})x_{n}\big) \big)(x_{m},\omega_{i})\sbt (\rho_{l}^{2},x^{2})
\\=&(\omega , x_{i})\sbt\big[\fr{ev}(y^{2}_{l}\otimes d\omega_{t})(x_{t},\omega_{i})-\fr{ev}(y^{2}_{l}\otimes\omega_{m}\wedge \omega_{n})(x_{n},\omega_{i})\sbt x_{m}\big]\sbt (\rho_{l}^{2},x^{2})
\\ &+\und{\mathrm{ev}}\big(\omega \otimes \mathrm{ev}\big( y_{t}\otimes d \und{\mathrm{ev}}[\rho_{t}\otimes \fr{ev}(y^{2}_{l}\otimes\omega_{m}\wedge \omega_{n})x_{n}] \big)x_{m}\big)(\rho_{l}^{2},x^{2})
\\ =& \underline{\mathrm{ev}}(\omega\otimes \mathfrak{ev}(y_{i}^{2}\otimes d\omega_{l})x_{l})(\rho_{i}^{2},x^{2})- \underline{\mathrm{ev}}(\omega\otimes \mathfrak{ev}(y_{i}^{2}\otimes \omega_{j}\wedge \omega_{k})x_{k})x_{j}\sbt(\rho_{i}^{2},x^{2}) 
\\ &+\underline{\mathrm{ev}}\big( \omega \otimes \mathrm{ev}[y_{t} \otimes d \underline{\mathrm{ev}} (\rho_{t} \otimes \mathfrak{ev} (y_{i}^{2}\otimes \omega_{j}\wedge \omega_{k})x_{k})] x_{j} \big)(\rho_{i}^{2},x^{2}) \qedhere
\end{align*}\end{proof}
\begin{thm}\label{TDXAHpf} The bialgebroid $\ct{D}\XA$ has a Hopf algebroid structure.
\end{thm}
\begin{proof} We have an induced action of $H\XA^{1}$ and $H(\Omega^{2})$ on the inner homs by Theorems \ref{THOmega} and \ref{THXAHopf}. Hence, we only need to check whether the relations imposed on $F$ fall in the annihilator of the induced actions on the inner homs of $\ct{D}\XA$-modules. If the relations for $\ct{D}\XA$ annihilate the inner homs, since the unit and counits for the adjunctions are $F$-module morphisms and automatically become $\ct{D}\XA$-module morphisms, thereby making $\ct{D}\XA$ a Hopf algebroid. 

We check the above for relation (\ref{EqExt1}) and leave the similar calculation for (\ref{EqExt2}) to the reader. Let $M$ and $N$ be $\ct{D}\XA$-modules, $f\in \mathrm{Hom}_{A}(M,N)$, $x^{2}\in\XA^{2}$ and $\omega,\rho\in \Omega^{1}$. We show that relation (\ref{EqExt1}) is annihalted for the induced action on $\mathrm{Hom}_{A}(M,N)$, by using (\ref{EqExt1}) for $N$ and relation (\ref{EqExt2}) for $M$:
\begin{align*}
\big[\mathfrak{ev}(x^{2}&\otimes \omega_{i}\wedge \omega_{j})  (x_{j},\rho)\sbt(x_{i},\omega)f\big](m)
\\=& \mathfrak{ev}(x^{2}\otimes \omega_{i}\wedge \omega_{j})(x_{j},\rho_{n})\sbt(x_{i},\rho_{m})f\big( (\omega,y_{m})\sbt(\rho,y_{n})m\big)
\\=&(x^{2},\rho_{m}\wedge \rho_{n}) f\big( (\omega,y_{m})\sbt(\rho,y_{n})m\big)
\\=& (x^{2},\rho_{l}^{2}) \ov{\underline{\fr{ev}}( \rho_{m}\wedge \rho_{n} \otimes y_{l}^{2}) }f\big( (\omega,y_{m})\sbt(\rho,y_{n})m\big)
\\ =& (x^{2},\rho_{l}^{2})f\big( \underline{\fr{ev}}( \rho_{m}\wedge \rho_{n} \otimes y_{l}^{2}) (\omega,y_{m})\sbt(\rho,y_{n})m\big)
\\=& (x^{2},\rho_{l}^{2})f\big((\omega\wedge \rho,y_{l}^{2})m\big) =\big[ (x^{2},\omega\wedge \rho) f\big](m)
\end{align*}

What remains to be checked is that for $\ct{D}\XA$-modules $M$ and $N$, the induced connections on $\mathrm{Hom}_{A}(M,N)$ and $\prescript{}{A}{\mathrm{Hom}}(M,N) $ are flat and satisfy the additional condition (\ref{EqFlExtraa}). To show that (\ref{EqFlat}) annihilates $\mathrm{Hom}_{A}(M,N)$, we use the identities (\ref{EqFlat}) and (\ref{EqFlExt}) for $N$ and (\ref{EqI}) for $M$. Let $f\in \mathrm{Hom}_{A}(M,N)$, $m\in M$ and $x^{2}\in\XA^{2}$:
\begin{align*}
\big[\fr{ev}&(x^{2}\otimes d\omega_{i})x_{i}f - \fr{ev}(x^{2}\otimes \omega_{i}\wedge \omega_{j})x_{j}\sbt x_{i} f\big](m) 
\\ =& \fr{ev}(x^{2}\otimes d\omega_{i})x_{i}f(m)- \fr{ev}(x^{2}\otimes d\omega_{i})(x_{i},\rho_{l})f\big( (\omega_{t},y_{l})\sbt x_{t} m\big) 
\\&- \fr{ev}(x^{2}\otimes \omega_{i}\wedge \omega_{j})x_{j}\sbt x_{i} f(m)
\\ &+\fr{ev}(x^{2}\otimes \omega_{i}\wedge \omega_{j})\big( x_{j} \sbt (x_{i},\rho_{l})+ (x_{j},\rho_{l})\sbt x_{i}\big) f\big( (\omega_{t},y_{l})\sbt x_{t}m\big)
\\ &- \fr{ev}(x^{2}\otimes \omega_{i}\wedge \omega_{j}) (x_{j},\rho_{l})\sbt  (x_{i},\rho_{m})f\big( (\omega_{n},y_{m})\sbt x_{n}\sbt(\omega_{t},y_{l})\sbt x_{t}m\big)
\\ =& -(x^{2},d\rho_{l})f\big( (\omega_{t},y_{l})\sbt x_{t}m\big) - (x^{2},\rho_{m}\wedge \rho_{l}) f\big( (\omega_{n},y_{m})\sbt x_{n}\sbt(\omega_{t},y_{l})\sbt x_{t}m\big)
\\ = &-(x^{2},\rho_{i}^{2})f\big( \ov{\underline{\mathfrak{ev}}(d\rho_{l}\otimes y_{i}^{2}) }(\omega_{t},y_{l})\sbt x_{t}m \big)
\\ &+ (x^{2},\rho_{i}^{2})f\big(\ov{\underline{\mathfrak{ev}}(\rho_{m}\wedge \rho_{l}\otimes x^{2})}(\omega_{n},y_{m})\sbt x_{n}\sbt(\omega_{t},y_{l})\sbt x_{t}m\big) =0
\end{align*}
To show that (\ref{EqFlExt}) annihilates $\mathrm{Hom}_{A}(M,N)$, we use the identities (\ref{EqFlExt}) for $N$ and (\ref{EqII}) for $M$. Let $f\in \mathrm{Hom}_{A}(M,N)$, $m\in M$, $x^{2}\in\XA^{2}$ and $\omega\in \Omega^{1}$: 
\begin{align*}
\big[&\mathfrak{ev}( x^{2}\otimes \omega_{i}\wedge \omega_{j}) \big(x_{j}\sbt(x_{i},\omega)+(x_{j},\omega)\sbt x_{i}\big) f\big] (m)
\\ =& \mathfrak{ev}( x^{2}\otimes \omega_{i}\wedge \omega_{j}) \big(x_{j}\sbt(x_{i},\rho_{l})+(x_{j},\rho_{l})\sbt x_{i}\big)f \big((\omega,y_{l})m \big)
\\&-\mathfrak{ev}( x^{2}\otimes \omega_{i}\wedge \omega_{j}) (x_{j},\rho_{l})\sbt(x_{i},\rho_{m})f\big( (\omega ,y_{m})\sbt (\omega_{t},y_{l} )\sbt x_{t} m \big)
\\&-\mathfrak{ev}( x^{2}\otimes \omega_{i}\wedge \omega_{j}) (x_{j},\rho_{l})\sbt(x_{i},\rho_{m}) f\big((\omega_{t},y_{m} )\sbt x_{t}\sbt (\omega, y_{l})m\big)
\\ =&\big(\mathfrak{ev}(x^{2}\otimes d\omega_{i} )(x_{i},\rho_{l} ) - (x^{2},d\rho_{l})\big)f \big((\omega,y_{l})m \big)
\\ &- (x^{2},\rho_{n}^{2}) \ov{\underline{\mathfrak{ev}}(\rho_{m}\wedge \rho_{l}\otimes y_{n}^{2})}f \big(\big[(\omega ,y_{m})\sbt (\omega_{t},y_{l} )\sbt x_{t}  +(\omega_{t},y_{m} )\sbt x_{t}\sbt (\omega, y_{l})\big]m\big)
\\ =& \big(\mathfrak{ev}(x^{2}\otimes d\omega_{i} )(x_{i},\rho_{l} ) - (x^{2},d\rho_{l})\big)f \big((\omega,y_{l})m \big)
\\ &-(x^{2},\rho_{n}^{2}) f \big( \big[(d\omega ,y_{n}^{2}) -  \ov{\underline{\mathfrak{ev}}(d\rho_{i} \otimes y_{n}^{2})}(\omega, y_{i})\big]m\big) 
\\=&\big[\mathfrak{ev}(x^{2}\otimes d\omega_{i} )(x_{i},\omega )f - (x^{2},d\omega)f\big](m)
\end{align*}
Now we demonstrate that (\ref{EqFlat}) annihilates $\prescript{}{A}{\mathrm{Hom}}(M,N) $. Let $g\in\prescript{}{A}{\mathrm{Hom}}(M,N) $, $m\in M$ and $x^{2}\in \XA^{2}$:
\begin{align*}
\big[\fr{ev}&(x^{2}\otimes d\omega_{i})x_{i}g - \fr{ev}(x^{2}\otimes \omega_{i}\wedge \omega_{j})x_{j}\sbt x_{i} g\big](m) 
\\=& y_{l}g\big( (\rho_{l},x_{i})\ov{\fr{ev}(x^{2}\otimes d\omega_{i})} m\big)- g\big(y_{l}\sbt (\rho_{l},x_{i})\ov{\fr{ev}(x^{2}\otimes d\omega_{i})} m \big)
\\ &- y_{m}\sbt y_{l} g\big( (\rho_{l},x_{i})\sbt (\rho_{m},x_{j})\ov{\fr{ev}(x^{2}\otimes \omega_{i}\wedge \omega_{j})} m\big) 
\\&- g\big(y_{l}\sbt (\rho_{l},x_{i})\sbt y_{m}\sbt (\rho_{m},x_{j})\ov{\fr{ev}(x^{2}\otimes \omega_{i}\wedge \omega_{j})} m\big)
\\ & + y_{l} g\big( [(\rho_{l},x_{i})\sbt y_{m}\sbt (\rho_{m},x_{j})+ y_{m}\sbt (\rho_{m},x_{i})\sbt (\rho_{l},x_{j})]\ov{\fr{ev}(x^{2}\otimes \omega_{i}\wedge \omega_{j})} m\big)
\\ =& y_{l}g\big( \big[(\rho_{l},x_{i})\ov{\fr{ev}(x^{2}\otimes d\omega_{i})} + (\rho_{l},x_{i})\sbt y_{m}\sbt (\rho_{m},x_{j})\ov{\fr{ev}(x^{2}\otimes \omega_{i}\wedge \omega_{j})} \big]m\big)
\\&- g \big( y_{l}\sbt \big[ (\rho_{l},x_{i})\ov{\fr{ev}(x^{2}\otimes d\omega_{i})} + (\rho_{l},x_{i})\sbt y_{m}\sbt (\rho_{m},x_{j})\ov{\fr{ev}(x^{2}\otimes \omega_{i}\wedge \omega_{j})} \big] m\big) 
\\&- y_{m}\sbt y_{l} g\big( (\rho_{l},x_{i})\sbt (\rho_{m},x_{j})\ov{\fr{ev}(x^{2}\otimes \omega_{i}\wedge \omega_{j})} m\big) 
\\ & + y_{l} g\big( y_{m}\sbt (\rho_{m},x_{i})\sbt (\rho_{l},x_{j})\ov{\fr{ev}(x^{2}\otimes \omega_{i}\wedge \omega_{j})} m\big)
\end{align*}
Since $g$ is a left $A$-module morphism, then for any $x\in \XA^{1}$ and $a\in A$,
\begin{equation}\label{Eqg} xg(am)- g((x\sbt a) m) =(xa)g(m) - g((xa) m)
\end{equation}
holds, where the terms $\mathrm{ev}(x\otimes da)g(m)$ cancel each other. Going back to our calculation, we utilise identity (\ref{EqIII}) and relation (\ref{EqFlat}) for $M$:
\begin{align*}
\big[\fr{ev}&(x^{2}\otimes d\omega_{i})x_{i}g - \fr{ev}(x^{2}\otimes \omega_{i}\wedge \omega_{j})x_{j}\sbt x_{i} g\big](m) 
\\=& \mathfrak{ev}(y_{i}^{2}\otimes d\omega_{l})x_{l} g\big((\rho_{i}^{2},x^{2})m\big) - \mathfrak{ev}(y_{i}^{2}\otimes \omega_{j}\wedge \omega_{k})x_{k}g \big( x_{j}\sbt (\rho_{i}^{2},x^{2}) m\big)
\\ &+\mathrm{ev}(y_{t} \otimes d \underline{\mathrm{ev}} [\rho_{t} \otimes \mathfrak{ev} (y_{i}^{2}\otimes \omega_{j}\wedge \omega_{k})x_{k}] )x_{j} g\big(  (\rho_{i}^{2},x^{2})m\big) 
\\ &- g\big(\big[\mathfrak{ev}(y_{i}^{2}\otimes d\omega_{l})x_{l} -\mathfrak{ev}(y_{i}^{2}\otimes \omega_{j}\wedge \omega_{k})x_{k}\sbt x_{j}\big]\sbt (\rho_{i}^{2},x^{2}) m\big)
\\ &-\mathrm{ev}(y_{t} \otimes d \underline{\mathrm{ev}} [\rho_{t} \otimes \mathfrak{ev} (y_{i}^{2}\otimes \omega_{j}\wedge \omega_{k})x_{k}] ) g\big( x_{j}\sbt  (\rho_{i}^{2},x^{2})m\big) 
\\&- y_{l}\sbt y_{m} g\big( (\rho_{m}\wedge\rho_{l},x^{2}) m\big)  + y_{l} g\big( y_{m}\sbt (\rho_{m}\wedge\rho_{l},x^{2}) m\big)
\\=& \fr{ev}(y_{i}^{2}\otimes d\omega_{l})x_{l} g\big((\rho_{i}^{2},x^{2})m\big) - \fr{ev}(y_{i}^{2},\omega_{j}\wedge \omega_{k})x_{k}g \big( x_{j}\sbt (\rho_{i}^{2},x^{2}) m\big)
\\&- y_{l}\sbt y_{m} g\big( \und{\fr{ev}}(\rho_{m}\wedge\rho_{l},y^{2}_{t})(\rho^{2}_{t},x^{2}) m\big)  + y_{l} g\big( y_{m}\sbt \und{\fr{ev}}(\rho_{m}\wedge\rho_{l},y^{2}_{t})(\rho^{2}_{t},x^{2}) m\big)
\\ &+\mathrm{ev}(y_{t} \otimes d \underline{\mathrm{ev}} [\rho_{t} \otimes \mathfrak{ev} (y_{i}^{2}\otimes \omega_{j}\wedge \omega_{k})x_{k}] )x_{j} g\big(  (\rho_{i}^{2},x^{2})m\big) 
\\ &-\mathrm{ev}(y_{t} \otimes d \underline{\mathrm{ev}} [\rho_{t} \otimes \mathfrak{ev} (y_{i}^{2}\otimes \omega_{j}\wedge \omega_{k})x_{k}] ) g\big( x_{j}\sbt  (\rho_{i}^{2},x^{2})m\big) 
\end{align*}
Since $\wedge$ is a pivotal module morphism, (\ref{EqWdgPiv}), then
\begin{align*} y_{l}&\sbt y_{m}\sbt \und{\fr{ev}}(\rho_{m}\wedge \rho_{l} \otimes x^{2} )= y_{l}\sbt \mathrm{ev}\big( y_{m}\otimes d\und{\fr{ev}}(\rho_{m}\wedge \rho_{l} \otimes x^{2} )\big)
\\&+ \mathfrak{ev} (x^{2} \otimes \omega_{i}\wedge \omega_{j}) x_{j}\sbt x_{i}  + \mathrm{ev}\big(y_{m}\otimes d\underline{\mathrm{ev}}[\rho_{m} \otimes\fr{ev}(x^{2}\otimes \omega_{i}\wedge \omega_{j})  x_{j}]\big) x_{i} 
\end{align*}
holds for any $x^{2}\in \XA^{2}$. Using this fact and relation (\ref{EqFlat}) for $N$, we see that all terms in our calculation cancel out:
\begin{align*}
\big[\fr{ev}&(x^{2}\otimes d\omega_{i})x_{i}g - \fr{ev}(x^{2}\otimes \omega_{i}\wedge \omega_{j})x_{j}\sbt x_{i} g\big](m) 
\\=& \fr{ev}(y_{i}^{2}\otimes d\omega_{l})x_{l} g\big((\rho_{i}^{2},x^{2})m\big) - \fr{ev}(y_{i}^{2},\omega_{j}\wedge \omega_{k})x_{k}g \big( x_{j}\sbt (\rho_{i}^{2},x^{2}) m\big)
\\&- \big[ y_{l}\sbt \mathrm{ev}\big( y_{m}\otimes d\und{\fr{ev}}(\rho_{m}\wedge \rho_{l} \otimes y_{t}^{2} )\big)+ \mathfrak{ev} (y_{t}^{2} \otimes \omega_{i}\wedge \omega_{j}) x_{j}\sbt x_{i} \big] g\big((\rho^{2}_{t},x^{2}) m\big)  
\\&- \mathrm{ev}\big(y_{m}\otimes d\underline{\mathrm{ev}}[\rho_{m}\otimes \fr{ev}(y_{t}^{2}\otimes \omega_{i}\wedge \omega_{j}) x_{j}]\big) x_{i} g\big((\rho^{2}_{t},x^{2}) m\big)
\\ & + y_{l}\sbt \mathrm{ev}\big( y_{m}\otimes d\und{\fr{ev}}(\rho_{m}\wedge \rho_{l} \otimes y_{t}^{2} )\big)g \big( (\rho_{i}^{2},x^{2}) m\big) 
\\ &+  \mathfrak{ev} (y_{t}^{2} \otimes \omega_{i}\wedge \omega_{j}) x_{j}g\big( x_{i} \sbt (\rho^{2}_{t},x^{2}) m\big)
\\& + \mathrm{ev}\big(y_{m}\otimes d\mathrm{ev}[\rho_{m} \fr{ev}(y_{t}^{2}\otimes \omega_{i}\wedge \omega_{j}) ] x_{j}\big)g\big( x_{i} \sbt (\rho^{2}_{t},x^{2}) m\big) 
\\ &+\mathrm{ev}(y_{t} \otimes d \underline{\mathrm{ev}} [\rho_{t} \otimes \mathfrak{ev} (y_{i}^{2}\otimes \omega_{j}\wedge \omega_{k})x_{k}] )x_{j} g\big(  (\rho_{i}^{2},x^{2})m\big) 
\\ &-\mathrm{ev}(y_{t} \otimes d \underline{\mathrm{ev}} [\rho_{t} \otimes \mathfrak{ev} (y_{i}^{2}\otimes \omega_{j}\wedge \omega_{k})x_{k}] ) g\big( x_{j}\sbt  (\rho_{i}^{2},x^{2})m\big) 
\\&=\big[ \fr{ev}(y_{i}^{2}\otimes d\omega_{l})x_{l} -  \mathfrak{ev} (y_{i}^{2} \otimes \omega_{i}\wedge \omega_{j}) x_{j}\sbt x_{i} \big]   g\big((\rho_{i}^{2},x^{2})m\big)=0
\end{align*} 
It remains to show that relation (\ref{EqFlExt}) annihilates $\prescript{}{A}{\mathrm{Hom}}(M,N) $. For this computation we use the facts mentioned above about left $A$-module morphisms and $\wedge$ being pivotal, in addition to identity (\ref{EqIII}) annihilating $M$ and relation (\ref{EqFlExt}) annihilating $N$. Let $g\in\prescript{}{A}{\mathrm{Hom}}(M,N) $ and $m\in M$:
\begin{align*}
\big[&\fr{ev}( x^{2}\otimes \omega_{i}\wedge \omega_{j}) \big((x_{j},\omega)\sbt x_{i}+x_{j}\sbt(x_{i},\omega)\big) g-\fr{ev}( x^{2}\otimes d\omega_{l})(x_{l},\omega) g\big] (m)
\\ &= (y_{m},\omega) \sbt y_{n} g\big( (\rho_{n},x_{i})\sbt (\rho_{m},x_{j})\ov{\fr{ev}( x^{2}\otimes \omega_{i}\wedge \omega_{j})}  m\big)
\\ &\quad - (y_{m},\omega)  g\big( y_{n}\sbt (\rho_{n},x_{i})\sbt (\rho_{m},x_{j})\ov{\fr{ev}( x^{2}\otimes \omega_{i}\wedge \omega_{j})}  m\big)
\\& \quad + y_{m} \sbt (y_{n},\omega) g\big( (\rho_{n},x_{i})\sbt (\rho_{m},x_{j})\ov{\fr{ev}( x^{2}\otimes \omega_{i}\wedge \omega_{j})}  m\big)
\\ &\quad - (y_{m},\omega)  g\big(  (\rho_{m},x_{i})\sbt y_{n}\sbt(\rho_{n},x_{j})\ov{\fr{ev}( x^{2}\otimes \omega_{i}\wedge \omega_{j})} m\big)
\\ &\quad - (y_{m},\omega)  g\big( (\rho_{m},x_{i})\ov{\fr{ev}(x^{2}\otimes d\omega_{i})}m\big)
\\&= [(y_{m},\omega) \sbt y_{n}+ y_{m} \sbt (y_{n},\omega)] g\big( \und{\fr{ev}}(\rho_{n}\wedge \rho_{m}\otimes y^{2}_{t}) (\rho_{t}^{2},x^{2})m\big) 
\\ &\quad- (y_{m},\omega)  g\big( y_{n}\sbt \und{\fr{ev}}(\rho_{n}\wedge \rho_{m}\otimes y^{2}_{t}) (\rho_{t}^{2},x^{2})m\big) 
\\ &\quad - (y_{m},\omega)g\big( \und{\mathrm{ev}}(\rho_{m}\otimes \mathfrak{ev}(y_{i}^{2}\otimes d\omega_{l})x_{l})(\rho_{i}^{2},x^{2})m\big) 
\\ &\quad + (y_{m},\omega)g\big(\und{\mathrm{ev}}(\rho_{m}\otimes \mathfrak{ev}(y_{i}^{2}\otimes \omega_{j}\wedge \omega_{k})x_{k})x_{j}\sbt(\rho_{i}^{2},x^{2})  m\big)
\\ &\quad - (y_{m},\omega)g\big(\underline{\mathrm{ev}}\big( \rho_{m}\otimes \mathrm{ev}(y_{t} \otimes d \underline{\mathrm{ev}} [\rho_{t} \otimes \mathfrak{ev} (y_{i}^{2}\otimes \omega_{j}\wedge \omega_{k})x_{k}]) x_{j} \big)(\rho_{i}^{2},x^{2}) m\big)
\\ &= \fr{ev}(y_{i}^{2}\otimes x_{i}\wedge x_{j} )[(x_{j},\omega) \sbt x_{i} + x_{j}\sbt (x_{i},\omega) ]g\big( (\rho_{t}^{2},x^{2})m\big)
\\&\quad + \mathrm{ev}\big( y_{m} \otimes d \und{\mathrm{ev}}(\rho_{m} \otimes\fr{ev}[y_{i}^{2}\otimes x_{i}\wedge x_{j}]x_{j})\big) (x_{i}, \omega) g\big( (\rho_{t}^{2},x^{2})m\big)
\\&\quad -\fr{ev}(y_{i}^{2}\otimes x_{i}\wedge x_{j} )(x_{j},\omega) g\big( x_{i} \sbt (\rho_{t}^{2},x^{2})m\big)
\\ &\quad - \mathfrak{ev}(y_{i}^{2}\otimes d\omega_{l})(x_{l},\omega)g\big(  (\rho_{i}^{2},x^{2})m\big) + \mathfrak{ev}(y_{i}^{2}\otimes \omega_{j}\wedge \omega_{k})(x_{k},\omega)g\big(x_{j}\sbt(\rho_{i}^{2},x^{2})  m\big)
\\ &\quad - \mathrm{ev}\big( y_{t} \otimes d \underline{\mathrm{ev}} [\rho_{t} \otimes \mathfrak{ev} (y_{i}^{2}\otimes \omega_{j}\wedge \omega_{k})x_{k}]\big)(x_{j},\omega)g\big( (\rho_{i}^{2},x^{2}) m\big)
\\ &= \fr{ev}(y_{i}^{2}\otimes x_{i}\wedge x_{j} )[(x_{j},\omega) \sbt x_{i} + x_{j}\sbt (x_{i},\omega) ]g\big( (\rho_{t}^{2},x^{2})m\big)
\\ &\quad - \mathfrak{ev}(y_{i}^{2}\otimes d\omega_{l})(x_{l},\omega)g\big(  (\rho_{i}^{2},x^{2})m\big) 
\\&= -  (y_{i}^{2}, d\omega ) g\big(  (\rho_{i}^{2},x^{2})m\big)= -\big[ (x^{2},d\omega)g \big] (m) \qedhere
\end{align*}\end{proof}
In Theorem \ref{THXaAnti}, we provided a criterion for when $H\XA^{1}$ admits an antipode in the sense of B{\"o}hm and Szlach{\'a}nyi. We now extend this result to $\ct{D}\XA$. 
\begin{thm}\label{TDXAAnti} The Hopf algebroid $\ct{D}\XA$ is a B{\"o}hm-Szlach{\'a}nyi Hopf algebroid, if and only if there exists a linear map $\Upsilon :\XA^{1}\rightarrow A$ satisfying (\ref{EqSpade}) and additional relations 
\begin{align}
\Upsilon (\fr{ev}(x^{2}\otimes d\omega_{i})x_{i} ) +\Upsilon \big(\Upsilon&(\fr{ev}(x^{2}\otimes \omega_{j}\wedge \omega_{k})x_{k} )x_{j}\big)=0\label{EqUpFla}
\\\underline{\fr{ev}}(d\omega\otimes x^{2})  -\underline{\mathrm{ev}}(\omega \otimes \fr{ev} ( x^{2}\otimes d\omega_{l})x_{l} ) & - \underline{\mathrm{ev}}( d\mathrm{ev}(\omega \otimes\fr{ev}(x^{2}\otimes \omega_{j}\wedge \omega_{k})x_{k}) \otimes x_{j})\nonumber
\\ =\underline{\mathrm{ev}}\big[ \omega \otimes \big(\Upsilon ( \fr{ev}(x^{2}\otimes \omega_{j}\wedge \omega_{k})&x_{k} )x_{j}+\fr{ev}(x^{2}\otimes \omega_{j}\wedge \omega_{k})x_{k}\Upsilon (x_{j})\big)\big] \label{EqUpFlEx}
\end{align}
hold for any $x^{2}\in \XA^{2}$ and $\omega\in \Omega^{1}$. 
\end{thm}
\begin{proof} ($\Rightarrow$) The argument is similar to that of Theorem \ref{THXaAnti}. If $\ct{D}\XA$ were to admit an antipode, $S$, we can use it to recover $\Upsilon$ by $\Upsilon(x)=-\epsilon (S(x))$ for $x\in \XA^{1}$. Hence, $\Upsilon$ would satisfy relations (\ref{EqSpade}) and additional relations arising from the flat relation (\ref{EqFlat}) and the additional condition (\ref{EqFlExt}). Let $x^{2}\in \XA^{2}$, then relation (\ref{EqUpFla}) arises directly from relation (\ref{EqFlat}): 
\begin{align*}
0=&-\epsilon \big( S( \fr{ev}(x^{2}\otimes d\omega_{i}) x_{i} - \mathfrak{ev}(x^{2}\otimes \omega_{j}\wedge\omega_{k})  x_{k}\sbt x_{j}  ) \big)
\\ =&\Upsilon (\fr{ev}(x^{2}\otimes d\omega_{i})x_{i} ) - \epsilon \big( -S(x_{j})\sbt S(\mathfrak{ev}(x^{2}\otimes \omega_{j}\wedge\omega_{k}) x_{k})\big) 
\\ =&\Upsilon (\fr{ev}(x^{2}\otimes d\omega_{i})x_{i} ) - \epsilon \big( -S(x_{j})\sbt \ov{\epsilon (S(\mathfrak{ev}(x^{2}\otimes \omega_{j}\wedge\omega_{k}) x_{k}))}\big) 
\\ =&\Upsilon (\fr{ev}(x^{2}\otimes d\omega_{i})x_{i} ) - \epsilon \big( S(x_{j})\sbt \ov{\Upsilon (\mathfrak{ev}(x^{2}\otimes \omega_{j}\wedge\omega_{k}) x_{k})}\big) 
\\ =&\Upsilon (\fr{ev}(x^{2}\otimes d\omega_{i})x_{i} ) +\Upsilon \big(\Upsilon(\fr{ev}(x^{2}\otimes \omega_{j}\wedge \omega_{k})x_{k} )x_{j}\big)
\end{align*}
Relation (\ref{EqUpFlEx}) arises from relation (\ref{EqFlExt}) where $\omega\in\Omega^{1}$:
\begin{align*}
\underline{\mathrm{ev}}(&d\omega \otimes x^{2}) -\underline{\mathrm{ev}}(\omega \otimes\mathfrak{ev}(x^{2}\otimes d\omega_{i} )x_{i}) = -\epsilon \big( S\big(\mathfrak{ev}(x^{2}\otimes d\omega_{i} )(x_{i},\omega ) - (x^{2},d\omega)\big)\big)
\\ =&-\epsilon\big( S\big(\mathfrak{ev}( x^{2}\otimes \omega_{i}\wedge \omega_{j}) [x_{j}\sbt(x_{i},\omega)+(x_{j},\omega)\sbt x_{i}] \big)\big)
\\=& -\epsilon\big( (\omega , x_{i})\sbt S(\mathfrak{ev}( x^{2}\otimes \omega_{i}\wedge \omega_{j}) x_{j})+ S(x_{i})\sbt (\omega ,\mathfrak{ev}( x^{2}\otimes \omega_{i}\wedge \omega_{j}) x_{j}) \big)
\\=& -\epsilon\big( (\omega , x_{i})\sbt \epsilon (S(\mathfrak{ev}( x^{2}\otimes \omega_{i}\wedge \omega_{j}) x_{j}))+ S(x_{i})\sbt \epsilon ((\omega ,\mathfrak{ev}( x^{2}\otimes \omega_{i}\wedge \omega_{j}) x_{j})) \big)
\\ =& \underline{\mathrm{ev}}\big( \omega \otimes \Upsilon ( \fr{ev}(x^{2}\otimes \omega_{i}\wedge \omega_{j})x_{j} )x_{i} \big) +  \Upsilon \big(\underline{\mathrm{ev}}( \omega \otimes\fr{ev}(x^{2}\otimes \omega_{j}\wedge \omega_{k})x_{k}\big)x_{j}\big)
\\ =& \underline{\mathrm{ev}}\big[ \omega \otimes \big(\Upsilon ( \fr{ev}(x^{2}\otimes \omega_{i}\wedge \omega_{j})x_{j} )x_{i}+\fr{ev}(x^{2}\otimes \omega_{i}\wedge \omega_{j})x_{j}\Upsilon (x_{i})\big)\big]
\\& +\underline{\mathrm{ev}}( d\mathrm{ev}(\omega \otimes\fr{ev}(x^{2}\otimes \omega_{j}\wedge \omega_{k})x_{k}) \otimes x_{j})
\end{align*}
($\Leftarrow$) We assume that such a linear map $\Upsilon$ satisfying (\ref{EqSpade}), (\ref{EqUpFla}) and (\ref{EqUpFlEx}) exists. A consequence of (\ref{EqUpFlEx}) which we use during the proof is that for any $x^{2}\in\XA^{2}$:
\begin{align*}
y_{l}\underline{\fr{ev}}(d \rho_{l}\otimes x^{2})  -\fr{ev} ( x^{2}\otimes d\omega_{l})x_{l}  & - y_{l}\underline{\mathrm{ev}}( d\underline{\mathrm{ev}}(\rho_{l} \otimes\fr{ev}(x^{2}\otimes \omega_{j}\wedge \omega_{k})x_{k}) \otimes x_{j})
\\ = \Upsilon ( \fr{ev}(x^{2}\otimes \omega_{j}\wedge \omega_{k})&x_{k} )x_{j}+\fr{ev}(x^{2}\otimes \omega_{j}\wedge \omega_{k})x_{k}\Upsilon (x_{j})
\end{align*}
holds. We extend the antipode $S$ of $H\XA^{1}$ and $H(\Omega^{2})$ as defined in Theorems \ref{THXAHopf} and \ref{THOmega} to $\ct{D}\XA$ and need to show that the antipode $S$ well-defined on $\ct{D}\XA$. In particular, we need to check relations (\ref{EqFlat}) and (\ref{EqFlExt}). Let $x\in \XA^{2}$. We prove that $S$ is well-defined for (\ref{EqFlat}) by first applying the properties of $\Upsilon$, then applying identity (\ref{EqI}) and using the fact that $\wedge$ satisfies (\ref{EqWdgPiv}), as we did in the proof of Theorem \ref{TDXAHpf}: 
\begin{align*}
-&S\big( \fr{ev}(x^{2}\otimes d\omega_{i})x_{i} - \mathfrak{ev}(x^{2}\otimes \omega_{j}\wedge\omega_{k})  x_{k}\sbt x_{j}\big) 
\\ =& (\omega_{l},\fr{ev}(x^{2}\otimes d\omega_{i})x_{i} )\sbt x_{l} + \ov{\Upsilon ( \fr{ev}(x^{2}\otimes d\omega_{i})x_{i} )} + \ov{\Upsilon (x_{j})} \sbt \ov{\Upsilon ( \mathfrak{ev}(x^{2}\otimes \omega_{j}\wedge\omega_{k})  x_{k})}
\\ & + (\omega_{l} ,x_{j})\sbt x_{l}\sbt \ov{\Upsilon ( \mathfrak{ev}(x^{2}\otimes \omega_{j}\wedge\omega_{k})  x_{k})}+\ov{\Upsilon (x_{j}) }\sbt (\omega_{t} , \mathfrak{ev}(x^{2}\otimes \omega_{j}\wedge\omega_{k})  x_{k}) \sbt x_{t}
\\&+(\omega_{l} ,x_{j})\sbt x_{l}\sbt (\omega_{t} , \mathfrak{ev}(x^{2}\otimes \omega_{j}\wedge\omega_{k})  x_{k}) \sbt x_{t}
\\ =& (\omega_{l},\fr{ev}(x^{2}\otimes d\omega_{i})x_{i} )\sbt x_{l} + \ov{\Upsilon ( \fr{ev}(x^{2}\otimes d\omega_{i})x_{i} )} + \ov{\Upsilon ( \Upsilon (\mathfrak{ev}(x^{2}\otimes \omega_{j}\wedge\omega_{k})  x_{k})x_{j})} 
\\ &- \ov{ \underline{\mathrm{ev}}(d\Upsilon (\mathfrak{ev}(x^{2}\otimes \omega_{j}\wedge\omega_{k})  x_{k}) \otimes x_{j})}+(\omega_{l} ,x_{j}) (x_{l},d\Upsilon ( \mathfrak{ev}(x^{2}\otimes \omega_{j}\wedge\omega_{k})x_{k}))
\\ &+ \big(\omega_{l} ,\Upsilon ( \mathfrak{ev}(x^{2}\otimes \omega_{j}\wedge\omega_{k})  x_{k})x_{j}\big)\sbt x_{l}+\big(\omega_{l} , \mathfrak{ev}(x^{2}\otimes \omega_{j}\wedge\omega_{k})  x_{k}\Upsilon (x_{j} )\big) \sbt x_{l}
\\&+(\omega_{l} ,x_{j})\sbt x_{l}\sbt (\omega_{t} , \mathfrak{ev}(x^{2}\otimes \omega_{j}\wedge\omega_{k})  x_{k}) \sbt x_{t}
\\ =&(\omega_{l}, y_{i}\underline{\fr{ev}}(d\rho_{i}\otimes x^{2}))\sbt x_{l}  - \big(\omega_{l} ,y_{t}\underline{\mathrm{ev}}( d\underline{\mathrm{ev}}(\rho_{t} \otimes\fr{ev}(x^{2}\otimes \omega_{j}\wedge \omega_{k})x_{k}) \otimes x_{j})\big)\sbt x_{l}
\\&+(\omega_{l} ,x_{j})\sbt x_{l}\sbt (\omega_{t} , \mathfrak{ev}(x^{2}\otimes \omega_{j}\wedge\omega_{k})  x_{k}) \sbt x_{t} 
\\ =& -\ov{\underline{\mathfrak{ev}}(\rho_{m}\wedge \rho_{n}\otimes x^{2})}(\omega_{k},y_{m})\sbt x_{k}\sbt (\omega_{l},y_{n})\sbt x_{l} 
\\&- (\omega_{t}, x_{j})(x_{t},d\underline{\mathrm{ev}}(\rho_{l} \otimes\fr{ev}(x^{2}\otimes \omega_{j}\wedge \omega_{k})x_{k}))\sbt(\omega_{l} ,y_{l}) \sbt x_{l}
\\ &+ (\omega_{l} ,x_{j})\sbt x_{l}\sbt (\omega_{t} , \mathfrak{ev}(x^{2}\otimes \omega_{j}\wedge\omega_{k})  x_{k}) \sbt x_{t} = 0 
\end{align*}
Let $\omega\in\Omega^{1}$. We prove $S$ is well-defined for relation (\ref{EqFlExt}) by using the properties of $\Upsilon$, then applying identity (\ref{EqII}) and using the fact that $\wedge$ satisfies (\ref{EqWdgPiv})
\begin{align*}
&-S\big(\mathfrak{ev}( x^{2}\otimes \omega_{i}\wedge \omega_{j}) [x_{j}\sbt(x_{i},\omega)+(x_{j},\omega)\sbt x_{i}] - \mathfrak{ev}(x^{2}\otimes d\omega_{i} )(x_{i},\omega )   \big) 
\\=&  (\omega, x_{i}) \sbt \ov{\Upsilon(\mathfrak{ev}( x^{2}\otimes \omega_{i}\wedge \omega_{j}) x_{j})}+\ov{\Upsilon(x_{i})}\sbt(\omega,\mathfrak{ev}( x^{2}\otimes \omega_{i}\wedge \omega_{j}) x_{j})
\\ &+(\omega,\mathfrak{ev}(x^{2}\otimes d\omega_{i} )x_{i}) +(\omega, x_{i}) \sbt (\omega_{l},\mathfrak{ev}( x^{2}\otimes \omega_{i}\wedge \omega_{j}) x_{j})\sbt x_{l}
\\&+(\omega_{l},x_{i})\sbt x_{l}\sbt(\omega,\mathfrak{ev}( x^{2}\otimes \omega_{i}\wedge \omega_{j}) x_{j})
\\=&  \big(\omega ,y_{l}\underline{\mathrm{ev}}( d\underline{\mathrm{ev}}(\rho_{l}\otimes\fr{ev}(x^{2}\otimes \omega_{j}\wedge \omega_{k})x_{k}) \otimes x_{j})\big)+ ( \omega, y_{l} \underline{\fr{ev}}(d \rho_{l},x^{2}) ) 
\\ & +(\omega,y_{m}\underline{\mathfrak{ev}}(\rho_{m}\wedge \rho_{n}\otimes x^{2}))\sbt(\omega_{l},y_{n})\sbt x_{l}+(\omega_{l},x_{i})\sbt x_{l}\sbt(\omega,\mathfrak{ev}( x^{2}\otimes \omega_{i}\wedge \omega_{j}) x_{j})) 
\\=&  -(\omega_{t}, x_{j})\sbt(x_{t},d\underline{\mathrm{ev}}(\rho_{l} \otimes\fr{ev}(x^{2}\otimes \omega_{j}\wedge \omega_{k})x_{k})\sbt(\omega ,y_{l}\big) +(d\omega ,x^{2}) 
\\ &-(\omega_{l},y_{m}\underline{\mathfrak{ev}}(\rho_{m}\wedge \rho_{n}\otimes x^{2}))\sbt x_{l}\sbt (\omega,y_{m})+(\omega_{l},x_{i})\sbt x_{l}\sbt(\omega,\mathfrak{ev}( x^{2}\otimes \omega_{i}\wedge \omega_{j}) x_{j})) 
\\ =&(d\omega ,x^{2})  =S((x^{2},d\omega))= -S(-(x^{2},d\omega))
\end{align*}
We also need to check relations (\ref{EqFlat}) and (\ref{EqFlExt}) for the inverse of the antipode $S^{-1}$. We prove that $S^{-1}$ is well-defined for (\ref{EqFlat}) by using identity (\ref{EqIII}) and relations (\ref{EqExt4}), (\ref{EqFlat}) and $\wedge$ satisfies (\ref{EqWdgPiv}), respectively: 
\begin{align*}
&-S^{-1}\big( \fr{ev}(x^{2}\otimes d\omega_{i})x_{i} - \mathfrak{ev}(x^{2}\otimes \omega_{j}\wedge\omega_{k})  x_{k}\sbt x_{j}\big)  
\\=&(y_{l}+\Upsilon (y_{l}))\sbt \big[(\rho_{l}, \fr{ev}(x^{2}\otimes d\omega_{i})x_{i}) +  (\rho_{l},x_{j})\sbt y_{t}\sbt (\rho_{t},\mathfrak{ev}(x^{2}\otimes \omega_{j}\wedge\omega_{k})  x_{k})\big]
\\ &+ (y_{l}+\Upsilon (y_{l}))\sbt(\rho_{l},x_{j}) \sbt \Upsilon (y_{t})\sbt (\rho_{t},\mathfrak{ev}(x^{2}\otimes \omega_{j}\wedge\omega_{k})  x_{k})
\\ =&  \mathfrak{ev}(y_{i}^{2}\otimes d\omega_{l})x_{l}\sbt (\rho_{i}^{2},x^{2})-  \fr{ev}(y_{i}^{2}\otimes\omega_{j}\wedge \omega_{k})x_{k}\sbt x_{j}\sbt(\rho_{i}^{2},x^{2}) 
\\ &+ \mathrm{ev}[y_{t} \otimes d \underline{\mathrm{ev}} (\rho_{t} \otimes \mathfrak{ev} (y_{i}^{2}\otimes \omega_{j}\wedge \omega_{k})x_{k})] x_{j} \sbt (\rho_{i}^{2},x^{2})
\\ &+  \Upsilon (\mathfrak{ev}(y_{i}^{2}\otimes d\omega_{l})x_{l}) (\rho_{i}^{2},x^{2})-  \Upsilon (\mathfrak{ev}(y_{i}^{2}\otimes \omega_{j}\wedge \omega_{k})x_{k}) x_{j}\sbt(\rho_{i}^{2},x^{2}) 
\\ &+\Upsilon\big( \mathrm{ev}[y_{t} \otimes d \underline{\mathrm{ev}} (\rho_{t} \otimes \mathfrak{ev} (y_{i}^{2}\otimes \omega_{j}\wedge \omega_{k})x_{k})] x_{j}\big)  (\rho_{i}^{2},x^{2})
\\ &+ (y_{l}+\Upsilon (y_{l}))\sbt \underline{\fr{ev}}(\rho_{l}\wedge \Upsilon (y_{t})\rho_{t}\otimes y_{i}^{2})\sbt(\rho_{i}^{2},x^{2})
\\ =&\mathrm{ev}[y_{t} \otimes d \underline{\mathrm{ev}} (\rho_{t} \otimes \mathfrak{ev} (y_{i}^{2}\otimes \omega_{j}\wedge \omega_{k})x_{k})] x_{j} \sbt (\rho_{i}^{2},x^{2})
\\ &+  \Upsilon (\mathfrak{ev}(y_{i}^{2}\otimes d\omega_{l})x_{l}) (\rho_{i}^{2},x^{2})-  \Upsilon (\mathfrak{ev}(y_{i}^{2}\otimes \omega_{j}\wedge \omega_{k})x_{k}) x_{j}\sbt(\rho_{i}^{2},x^{2}) 
\\ &+\Upsilon\big( \mathrm{ev}[y_{t} \otimes d \underline{\mathrm{ev}} (\rho_{t} \otimes \mathfrak{ev} (y_{i}^{2}\otimes \omega_{j}\wedge \omega_{k})x_{k})] x_{j}\big)  (\rho_{i}^{2},x^{2})
\\ &+\underline{\mathrm{ev}}\big(\Upsilon(y_{t})\rho_{t}\otimes \fr{ev}(y^{2}_{i}\otimes\omega_{j}\wedge \omega_{k}\big)x_{k} )x_{j} \sbt  (\rho_{i}^{2},x^{2})
\\&+  \Upsilon\big(\underline{\mathrm{ev}}\big(\Upsilon(y_{t})\rho_{t}\otimes \fr{ev}(y^{2}_{i}\otimes\omega_{j}\wedge \omega_{k}\big)x_{k} )x_{j} \big)  (\rho_{i}^{2},x^{2})
\\=& \Upsilon (\mathfrak{ev}(y_{i}^{2}\otimes d\omega_{l})x_{l}) (\rho_{i}^{2},x^{2})+  \Upsilon\big(\Upsilon (\fr{ev}(y^{2}_{i}\otimes\omega_{j}\wedge \omega_{k}\big)x_{k} )x_{j} \big)  (\rho_{i}^{2},x^{2})
\\& -  \Upsilon (\mathfrak{ev}(y_{i}^{2}\otimes \omega_{j}\wedge \omega_{k})x_{k}) x_{j}\sbt(\rho_{i}^{2},x^{2})  +\Upsilon (\fr{ev}(y^{2}_{i}\otimes\omega_{j}\wedge \omega_{k}\big)x_{k} )x_{j} \sbt  (\rho_{i}^{2},x^{2})=0
\end{align*}
Finally, we prove that $S^{-1}$ is well-defined for relation (\ref{EqFlExt}) by using identity (\ref{EqIII}), the manipulation used previously for $\wedge$, (\ref{EqWdgPiv}), and the properties of $\Upsilon$, (\ref{EqSpade}) and (\ref{EqUpFlEx}):
\begin{align*}
&-S^{-1}\big(\mathfrak{ev}( x^{2}\otimes \omega_{i}\wedge \omega_{j}) [x_{j}\sbt(x_{i},\omega)+(x_{j},\omega)\sbt x_{i}] - \mathfrak{ev}(x^{2}\otimes d\omega_{i} )(x_{i},\omega )   \big) 
\\ =&(\omega,\mathfrak{ev}(x^{2}\otimes d\omega_{i} )x_{i})+(\omega,x_{i})\sbt y_{l}\sbt (\rho_{l},\mathfrak{ev}(x^{2}\otimes \omega_{i}\wedge\omega_{j})  x_{j})
\\ &+(\omega\Upsilon (y_{l}),x_{i})\sbt ( \rho_{l},\mathfrak{ev}(x^{2}\otimes \omega_{i}\wedge\omega_{j})  x_{j})\\&+ (y_{l}+\Upsilon(y_{l}))\sbt (\rho_{l},x_{i}) \sbt (\omega,\mathfrak{ev}(x^{2}\otimes \omega_{i}\wedge\omega_{j})  x_{j})
\\ =& \underline{\mathrm{ev}}(\omega\otimes \mathfrak{ev}(y_{i}^{2}\otimes d\omega_{l})x_{l})(\rho_{i}^{2},x^{2})- \underline{\mathrm{ev}}(\omega\otimes \mathfrak{ev}(y_{i}^{2}\otimes \omega_{j}\wedge \omega_{k})x_{k})x_{j}\sbt(\rho_{i}^{2},x^{2}) 
\\ &+\underline{\mathrm{ev}}\big( \omega \otimes \mathrm{ev}[y_{t} \otimes d \underline{\mathrm{ev}} (\rho_{t} \otimes \mathfrak{ev} (y_{i}^{2}\otimes \omega_{j}\wedge \omega_{k})x_{k})] x_{j} \big)(\rho_{i}^{2},x^{2})
\\&+\underline{\fr{ev}}( \omega \wedge \Upsilon (y_{l}) \rho_{l}\otimes y_{i}^{2}) (\rho_{i}^{2},x^{2}) + (y_{l}+\Upsilon(y_{l}))\sbt\underline{\fr{ev}}( \omega \wedge \rho_{l}\otimes y_{i}^{2}) (\rho_{i}^{2},x^{2}) 
\\ =& \underline{\mathrm{ev}}(\omega\otimes \mathfrak{ev}(y_{i}^{2}\otimes d\omega_{l})x_{l})(\rho_{i}^{2},x^{2})- \underline{\mathrm{ev}}(\omega\otimes \mathfrak{ev}(y_{i}^{2}\otimes \omega_{j}\wedge \omega_{k})x_{k})x_{j}\sbt(\rho_{i}^{2},x^{2}) 
\\ &+\underline{\mathrm{ev}}\big( \omega \otimes \mathrm{ev}[y_{t} \otimes d \underline{\mathrm{ev}} (\rho_{t} \otimes \mathfrak{ev} (y_{i}^{2}\otimes \omega_{j}\wedge \omega_{k})x_{k})] x_{j} \big)(\rho_{i}^{2},x^{2})
\\&+\underline{\mathrm{ev}}\big( \omega \otimes \underline{\mathrm{ev}}[ \Upsilon (y_{l}) \rho_{l}\otimes \mathfrak{ev} (y_{i}^{2}\otimes\omega_{j}\wedge \omega_{k})x_{k}] x_{j} \big) (\rho_{i}^{2},x^{2}) 
\\&+ \underline{\mathrm{ev}}(\omega\otimes \mathfrak{ev}(y_{i}^{2}\otimes \omega_{j}\wedge \omega_{k})x_{k})x_{j}\sbt(\rho_{i}^{2},x^{2}) 
\\ &+ \Upsilon\big(\underline{\mathrm{ev}}(\omega\otimes \underline{\mathrm{ev}}(y_{i}^{2},\omega_{j}\wedge \omega_{k})x_{k})x_{j}\big)(\rho_{i}^{2},x^{2})
\\= & \underline{\mathrm{ev}}(\omega\otimes \underline{\mathrm{ev}}(y_{i}^{2}\otimes d\omega_{l})x_{l})(\rho_{i}^{2},x^{2}) + \Upsilon\big(\underline{\mathrm{ev}}(\omega\otimes \mathfrak{ev}(y_{i}^{2}\otimes \omega_{j}\wedge \omega_{k})x_{k})x_{j}\big)(\rho_{i}^{2},x^{2})
\\ &+\underline{\mathrm{ev}}\big( \omega \otimes \mathrm{ev}[y_{t} \otimes d \underline{\mathrm{ev}} (\rho_{t} \otimes \mathfrak{ev} (y_{i}^{2}\otimes \omega_{j}\wedge \omega_{k})x_{k})] x_{j} \big)(\rho_{i}^{2},x^{2})
\\&+\underline{\mathrm{ev}}\big( \omega \otimes \Upsilon (y_{l})\underline{\mathrm{ev}}[  \rho_{l}\otimes \mathfrak{ev} (y_{i}^{2}\otimes\omega_{j}\wedge \omega_{k})x_{k}] x_{j} \big) (\rho_{i}^{2},x^{2}) 
\\= & \underline{\mathrm{ev}}(\omega\otimes \underline{\mathrm{ev}}(y_{i}^{2}\otimes d\omega_{l})x_{l})(\rho_{i}^{2},x^{2}) + \Upsilon\big(\underline{\mathrm{ev}}(\omega\otimes \mathfrak{ev}(y_{i}^{2}\otimes \omega_{j}\wedge \omega_{k})x_{k})x_{j}\big)(\rho_{i}^{2},x^{2})
\\ &+\underline{\mathrm{ev}}\big( \omega \otimes \Upsilon(\mathfrak{ev} (y_{i}^{2}\otimes\omega_{j}\wedge \omega_{k})x_{k})x_{j}\big) 
\\ = &  \Upsilon\big(\underline{\mathrm{ev}}(\omega\otimes \mathfrak{ev}(y_{i}^{2}\otimes \omega_{j}\wedge \omega_{k})x_{k})x_{j}\big)(\rho_{i}^{2},x^{2}) +\underline{\fr{ev}}(d\omega \otimes y_{i}^{2}) (\rho_{i}^{2},x^{2}) 
\\ &- \underline{\mathrm{ev}}(\omega\otimes \mathfrak{ev}(y_{i}^{2}\otimes \omega_{j}\wedge \omega_{k})x_{k})\Upsilon (x_{j}) (\rho_{i}^{2},x^{2}) 
\\ &- \underline{\mathrm{ev}}(d\underline{\mathrm{ev}}(\omega\otimes \mathfrak{ev}(y_{i}^{2}\otimes \omega_{j}\wedge \omega_{k})x_{k}) \otimes x_{j} ) (\rho_{i}^{2},x^{2})
\\ =& (d\omega ,x^{2}) = - S^{-1}\big( -(x^{2},d\omega) \big) \qedhere
\end{align*}
\end{proof}
\begin{rmk}\label{RSujAnti} In the proof of Theorem \ref{TDXAAnti}, it is implicit that $\Upsilon$ arises from a right action of $\ct{D}\XA$ on $A$, where the action of elements of $H(\Omega^{1})$ and $H(\Omega^{2})$ agree with the counit. As mentioned previously if the calculus is surjective, then relation (\ref{EqFlExt}) follows from the flatness relation (\ref{EqFlat}). Hence, to obtain a right action of $\ct{D}\XA$ on $A$, one would only need to check the flatness condition, which translates to (\ref{EqUpFla}) for $\Upsilon$ and condition (\ref{EqUpFlEx}) would follow.  
\end{rmk}
\subsection{Examples}\label{SExFlat}
In this section we calculate $\ct{D}\XA$ explicitly in the cases of finite quivers and the bicovariant calculus of Example \ref{EBiaD6} on $\mathbb{C}D_{6}$. Observe that for any first order calculus there can be several choices for $\Omega^{2}$, but for our construction we require $\Omega^{2}$ to be a pivotal bimodule and $\wedge$ a pivotal bimodule morphism, satisfying (\ref{EqWdgPiv}).
\begin{ex}\label{DXQuiv}[Finite Quivers] For simplicity we assume the finite quiver $\Gamma=(V,E)$, does not have any loops i.e. there no edge $e\in E$ has the same source and target. There are several choices of $\Omega^{2}$ for the calculus on finite quivers $\Gamma=(V,E)$ [Proposition 1.40 \cite{beggs2019quantum}]. Here we take $\Omega^{2}$ to be the quotient of $\Omega^{1}\otimes\Omega^{1}$ by the sub-bimodule spanned by sums 
$ \sum_{s(e_{1})=p, t(e_{2})=q} \ovr{e_{1}}\otimes\ovr{e_{2}} $
corresponding to each pair of vertices $p,q\in V$. The bimodule morphism $\wedge$ is the natural projection $\Omega^{1}\otimes\Omega^{1}\twoheadrightarrow \Omega^{2}$ and the differential $d:\Omega^{1}\rightarrow \Omega^{2}$ is defined by 
$$d\ovr{e_{1}}= \sum_{e\in E} \ovr{e}\wedge \ovr{e_{1}} - \sum_{e\in E} \ovr{e_{1}}\wedge \ovr{e} $$
for any $\ovr{e_{1}}\in \Omega^{1}$. The left and right dual of $\Omega^{2}$ is the quotient of $\XA^{1}\otimes\XA^{1}$, by the same relations $ \sum_{s(e_{1})=p, t(e_{2})=q} \ovl{e_{2}}\otimes\ovl{e_{1}} $
corresponding to each pair of vertices $p,q\in V$. To define a pair of evaluation and coevaluation maps, we nominate a \emph{2-step} $(a_{p,q},b_{p,q})\in E\times E$ for each pair of vertices $p,q\in V$, such that $s(a_{p,q})=p$,  $t(b_{p,q})=q$ and $s(a_{p,q})=t(b_{p,q})$. We denote the set of nominated 2-steps by $N=\lbrace (a_{p,q},b_{p,q}) \in E\times E\mid p,q \in V \rbrace$. Notice that for any pair of vertices $p,q\in V$,
$$\ovr{a_{p,q}}\wedge \ovr{b_{p,q}}= \sum_{\substack{s(e_{1})=p,\  t(e_{2})=q\\ (e_{1},e_{2})\neq (a_{p,q},b_{p,q})}} -\ovr{e_{1}}\wedge \ovr{e_{2}}$$
Thereby $\Omega^{2}$ is spanned by elements $\ovr{e_{1}}\wedge \ovr{e_{2}}$ whose underlying 2-steps are not nominated and lie in $E_{2}= \lbrace (e_{1},e_{2})\in E\times E \mid t(e_{1})=s(e_{2})\rbrace \setminus N$. With this basis, we can describe the coevaluation and evaluation maps by 
$$\fr{coev}(1) = \sum_{(e_{1},e_{2}) \in E_{2} } \ovr{e_{1}}\wedge \ovr{e_{2}}\otimes \ovl{e_{2}}\wedge \ovl{e_{1}},\quad \fr{ev}(\ovl{e_{2}}\wedge \ovl{e_{1}}\otimes \ovr{e_{3}}\wedge \ovr{e_{4}} )=\delta_{e_{1},e_{3}} \delta_{e_{2},e_{4}}f_{t(e_{4})}$$
for any $\ovr{e_{3}}\wedge \ovr{e_{4}}\in \Omega^{2}$ and $\ovl{e_{2}}\wedge \ovl{e_{1}}\in \XA^{2}$. It is trivial to check that $\wedge$ is a pivotal bimodule morphism. By Remark \ref{RWdSplit}, since $\wedge$ is surjective and splits, we do not need to add any additional generators to $H\XA^{1}$ which was constructed in Example \ref{EHopfGrph}. We only need to quotient out the additional relations for extendability. First notice that by (\ref{EqExt1}) the elements defining the action of $\XA^{2}\otimes_{\field}\Omega^{2}$ will be given by 
$$(\ovl{e_{2}}\wedge \ovl{e_{1}}, \ovr{e_{3}}\wedge \ovr{e_{4}}) := (\ovl{e_{2}},\ovr{e_{4}})\sbt (\ovl{e_{1}},\ovr{e_{3}})- (\ovl{b_{p,q}},\ovr{e_{4}})\sbt (\ovl{a_{p,q}},\ovr{e_{3}}) $$
where $p=s(e_{1})$ and $q=t(e_{2})$. By a similar deduction via (\ref{EqExt2}), the extendability relations which we quotient out $H\XA^{1}$ by are 
\begin{align*}
\sum_{s(e_{3})=v, t(e_{4})=w} (\ovl{e_{2}},\ovr{e_{4}})\sbt (\ovl{e_{1}},\ovr{e_{3}})- (\ovl{b_{p,q}},\ovr{e_{4}})\sbt (\ovl{a_{p,q}},\ovr{e_{3}}) &=0
\\\sum_{s(e_{3})=v, t(e_{4})=w} (\ovr{e_{3}},\ovl{e_{1}})\sbt ( \ovr{e_{4}},\ovl{e_{2}}) - (\ovr{e_{3}},\ovl{a_{p,q}})\sbt ( \ovr{e_{4}},\ovl{b_{p,q}}) &=0
\end{align*}
for all pair of vertices $v,w\in V$ and any $\ovl{e_{2}}\wedge \ovl{e_{1}}\in \XA^{2}$, where $p=s(e_{1})$ and $q=t(e_{2})$. The flatness relation (\ref{EqFlat}) reduces to relations $\ovl{e_{2}} -\ovl{b_{p,q}}+\ovl{b_{p,q}}\sbt \ovl{a_{p,q}} = \ovl{e_{2}}\sbt \ovl{e_{1}}$ holding for all $\ovl{e_{2}}\wedge \ovl{e_{1}}\in \XA^{2}$, where $p=s(e_{1})$ and $q=t(e_{2})$. However, this is in  terms of elements of $T\XA^{1}_{\sbt}$ and we have described $H\XA^{1}$ in terms of $\field\Gamma$. In terms of the quiver path algebra, the relations translate to 
$$\ovl{b_{p,q}}\sbt \ovl{a_{p,q}} = \ovl{e_{2}}\sbt \ovl{e_{1}}$$
holding for all $\ovl{e_{2}}\wedge \ovl{e_{1}}\in \XA^{2}$, where $p=s(e_{1})$ and $q=t(e_{2})$. Similarly, The additional condition (\ref{EqFlExt}) reduces to 
\begin{align*}
\ovl{e_{2}}&\sbt (\ovl{e_{1}}, \ovr{e_{3}})+(\ovl{e_{2}}, \ovr{e_{3}})\sbt \ovl{e_{1}} - \ovl{b_{p,q}}\sbt (\ovl{a_{p,q}}, \ovr{e_{3}})-(\ovl{b_{p,q}}, \ovr{e_{3}})\sbt \ovl{a_{p,q}}-(\ovl{b_{p,q}}, \ovr{e_{3}}) 
\\ &=(\ovl{e_{2}}, \ovr{e_{3}}) - \sum_{e\in E}\big[(\ovl{e_{2}},\ovr{e_{3}})\sbt (\ovl{e_{1}},\ovr{e})- (\ovl{b_{p,q}},\ovr{e_{4}})\sbt (\ovl{a_{p,q}},\ovr{e_{3}})\big]
\\&\quad - \sum_{e\in E}\big[ (\ovl{e_{2}},\ovr{e})\sbt (\ovl{e_{1}},\ovr{e_{3}})- (\ovl{b_{p,q}},\ovr{e})\sbt (\ovl{a_{p,q}},\ovr{e_{3}})\big]
\end{align*}
for all $\ovl{e_{2}}\wedge \ovl{e_{1}}\in \XA^{2}$ and $\ovr{e_{3}}\in \Omega^{1}$, where $p=s(e_{1})$ and $q=t(e_{2})$. Not only is $\ct{D}\XA$ a Hopf algebroid, but it also admits an antipode. One can show that $\Upsilon$ as defined for $H\XA^{1}$ in Example \ref{EHopfGrph} satisfies the conditions presented in Theorem \ref{TDXAAnti}. Recall $\Upsilon (\ovl{e}) = f_{s(e)}-f_{t(e)}$ for any $e\in E$. For a 2-step $(e_{1},e_{2})\in E_{2}$ with $p=s(e_{1})$ and $q=t(e_{2})$, condition (\ref{EqUpFla}) translates to 
\begin{align*}
\Upsilon \big(\ovl{e_{2}} -\ovl{b_{p,q}}\big) + \Upsilon \big( \Upsilon (\ovl{e_{2}})\ovl{e_{1}} -  \Upsilon (\ovl{b_{p,q}})\ovl{a_{p,q}}\big) = f_{s(e_{2})}-f_{s(b_{p,q})} - f_{t(e_{1})}+ f_{t(a_{p,q})}
\end{align*}
which is trivially equal to zero. Condition (\ref{EqUpFlEx}) also follows from a straightforward calculation for the four nontrivial cases where $\omega$ is one of $\ovr{e_{1}}$, $\ovr{e_{1}}$, $\ovr{a_{p,q}}$ or $\ovr{b_{p,q}}$.
\end{ex}

\begin{ex}$ $[Finite Groups] The calculus presented for a finite group algebra $\field G$ in Example \ref{EGrpAlg}, can be extended by setting $\Omega^{2}= \bigwedge_{\field}^{2}(\Lambda)\otimes_{\field}\field G$, where $\bigwedge_{\field}^{2}(\Lambda)$ is the exterior power of the vector space $\Lambda$. The differential $d:\Omega^{1}\rightarrow \Omega^{2}$ is defined as $d(\lambda \otimes_{\field} g)= \lambda\wedge \zeta (g)\otimes_{\field} g$. The left action of $\field G$ on $\Omega^{2}$ is the induced the action on the tensor product of left Yetter-Drinfeld modules and is described by 
$$g\triangleright (\lambda_{1}\wedge \lambda_{2} \otimes_{\field} h )= (g\triangleright\lambda_{1})\wedge (g\triangleright\lambda_{2} ) \otimes_{\field} gh $$ 
where $g\in G$ and $\lambda_{1}\wedge \lambda_{2} \otimes_{\field} h\in \bigwedge_{\field}^{2}(\Lambda) \otimes_{\field}\field G$. Since $\Lambda$ is finite dimensional, $\Omega^{2}$ is a finitely generated free right module. Additionally, by construction $\Omega^{2}$ is a Hopf bimodule and by Example \ref{EPivHopf}, it is a pivotal bimodule. It is a straightforward calculation to check that $\wedge$ is a pivotal bimodule morphism. Hence, we can construct $\ct{D}\XA$ for the calculus on the Dihedral group $D_{6}$ described in Example \ref{EBiaD6}. Since $\wedge$ is surjective, the generators of the from $\Omega^{2}\otimes_{\field}\XA^{2}$ are redundant: for $i,j,k,l\in \lbrace \xi,\tau\rbrace$
$$(f_{j}\wedge f_{i}, k\wedge l) =\omx{i}{\gamma}{l}\sbt \omx{j}{\gamma}{k}  -\omx{j}{\gamma}{l}\sbt \omx{i}{\gamma}{k}$$ 
Hence $\ct{D}\XA$ reduces to imposing the relevant extendability relations
\begin{align*}
\omx{i}{\gamma}{l}\sbt \omx{j}{\gamma}{k}&  -\omx{j}{\gamma}{l}\sbt \omx{i}{\gamma}{k} = -\omx{i}{\gamma}{k}\sbt \omx{j}{\gamma}{l}  +\omx{j}{\gamma}{k}\sbt \omx{i}{\gamma}{l}
\\ \omx{i}{\kappa}{l}\sbt \omx{j}{\kappa}{k}&  -\omx{j}{\kappa}{l}\sbt \omx{i}{\kappa}{k} = -\omx{i}{\kappa}{k}\sbt \omx{j}{\kappa}{l}  +\omx{j}{\kappa}{k}\sbt \omx{i}{\kappa}{l}
\end{align*}
for all $i,j,k,l\in \lbrace \xi,\tau\rbrace$, on $H\XA^{1}$ constructed in Example \ref{ExHopfD6}. The flat condition then translates to 
$$ f_{\xi}\sbt f_{\tau}= f_{\tau}\sbt f_{\xi}$$
Since the calculus is surjective, the additional condition (\ref{EqFlExt}) follows directly.
\end{ex}
\subsection{Commutative Case and Lie-Rinehart Algebras}\label{SCommLie}
In this section, we assume that the algebra $A$ is commutative and recover several known Hopf algebroid structures in the commutative setting, as quotients of $H\XA^{1}$ and $\ct{D}\XA$. 

When $A$ is a commutative algebra, the ordinary category of connections $\prescript{}{A}{\ct{E}}$ is well known to have a monoidal closed structure. Since $A^{op}\cong A$, every left $A$-module has a natural $A$-bimodule structure with the right and left actions agreeing. We call such bimodules over a commutative algebra \emph{symmetric} bimodules. If $\Omega^{1}$ is a symmetric bimodule (referred to as the classical case in \cite{andre2001differentielles}), every symmetric $A$-bimodule $M$ has a natural $\Omega^{1}$-intertwining, namely the \emph{flip} map, $\fr{fl}:M\otimes\Omega^{1} \rightarrow \Omega^{1}\otimes M$ defined by $\fr{fl}(m\otimes\omega )=\omega \otimes m$, for all $m\in M$ and $\omega\in \Omega^{1}$. Hence, every left connection has the structure of an invertible left bimodule connection via the flip map. From our point of view, if $\Omega^{1}$ is fgp as a left (or right) module, then the $\XA^{1}$-intertwinings (\ref{EqXsig}) and (\ref{EqsigX}) are inverses for any symmetric bimodule with $\fr{fl}$ as its $\Omega^{1}$-intertwining. Hence, the classical category of connections embeds as a subcategory of $\prescript{}{H\XA^{1}}{\ct{M}}$. This subcategory is of course represented by $T\XA^{1}_{\sbt}$, and the Hopf algebroid structure of $T\XA^{1}_{sbt}$ can be recovered by viewing it as the quotient of $H\XA^{1}$ by relations 
\begin{equation}\label{EqQuotRel}
a=\ov{a},\quad (x,\omega)=\mathrm{ev}(x\otimes \omega), \quad (\omega,x )=\und{\mathrm{ev}}(\omega \otimes x) 
\end{equation}  
where $a\in A$, $x\in \XA^{1}$ and $\omega\in \Omega^{1}$. First, observe that when the calculus is surjective, second pair of relations follow from $a=\ov{a}$ holding for all $a\in A$:
\begin{align*} a=\ov{a}\quad &\Rightarrow (x,da)=[x,\ov{a}]= [x,a]=\mathrm{ev}(x\otimes da) 
\\&\Rightarrow  \underline{\mathrm{ev}}(\omega,y)=\ov{\underline{\mathrm{ev}}(\omega,y)}=(\omega_{i},y)\sbt (x_{i},\omega )=(\omega_{i},y)\mathrm{ev}(x_{i}\otimes\omega )= (\omega ,y) 
\end{align*} 
Secondly, notice that under these relations, the Hopf relations on $H\XA^{1}$, (\ref{EqRelInv1}), (\ref{EqRelInv2}), (\ref{EqRelHpf1}) and (\ref{EqRelHpf2}) all become trivial. We recover the induced action of $T\XA^{1}_{\sbt}$ on the usual tensor product of connections and inner homs, for any pair of left connections $M$ and $N$:
\begin{equation}\label{EqTXHopf} x (m\otimes n) = xm\otimes n+ m\otimes xn,\quad [xf](m)= xf(m)-f(xm)
\end{equation}
where $f\in \mathrm{Hom}_{A}(M,N)$, $m\in M$ and $m\otimes n\in M\otimes N$. Notice that left inner homs and right inner homs agree for symmetric bimodules. 

In Section 2.4 of \cite{andre2001differentielles} the \emph{semi-classical} case is considered, where $A$ is a commutative algebra and $\Omega^{1}$ is a surjective calculus, not necessarily assumed to be symmetric, while the connections are still regarded as symmetric bimodules with invertible bimodule connections. The author then recovers the induced connections on inner homs of this category of connections, by noting that this is possible when the $\Omega$-intertwinings of the bimodule connections in consideration are invertible, Theorem 2.4.2.2 \cite{andre2001differentielles}. One can deduce from the calculations above that for a surjective calculus, the $\Omega^{1}$-intertwining of invertible bimodule connection on symmetric bimodule will be forced to be the flip map. Additionally, the Hopf algebroid representing the category in consideration would be the quotient of $H\XA^{1}$ by relations $a=\ov{a}$ for all $a\in A$. The author of \cite{andre2001differentielles} only needed to discuss invertible bimodule connections on symmetric bimodules since the additional Hopf conditions (\ref{EqRelHpf1}) and (\ref{EqRelHpf2}), hold immediately when quotienting $IB\XA^{1}$ by the relation $a=\ov{a}$. In other words, the quotients of $IB\XA^{1}$ and $H\XA^{1}$ by the relation $a=\ov{a}$ are isomorphic, and produce the Hopf algebroid in question. Observe that when $\Omega^{1}$ is not symmetric, the quotient will not necessarily be isomorphic to $T\XA^{1}_{\sbt}$, but the additional relations $\mathrm{ev}(xa\otimes \omega)=\mathrm{ev}(x\otimes \omega)a$ and $a\underline{\mathrm{ev}}(\omega\otimes x)=\underline{\mathrm{ev}}(\omega a\otimes x)  $ arise from the relations of $IB(\Omega^{1})$. These additional relations can be seen to arise directly when we require the flip map and its inverse, between $\Omega^{1}$ and a symmetric bimodule, to be bimodule maps.

To understand $\ct{D}\XA$ in the commutative setting, we first recall the definition of \emph{Lie-Rinehart} algebras and their associated family of Hopf algebroids from \cite{kowalzig2009hopf,kowalzig2010duality}. A pair $(A,\XA^{1})$ is called a Lie-Rinehart algebra if $A$ is a commutative algebra, $\XA^{1}$ an $A$-module with a linear maps $[,]:\XA^{1}\otimes_{\field} \XA^{1}\rightarrow \XA^{1}$ and $\tau :\XA^{1}\rightarrow Der(A)$ such that  
\begin{enumerate}[(I)]
\item $[,]$ is antisymmetric
\item $[,]$ satisfies the Jacobi identity
\item $[x,y](a)= x(y(a))-y(x(a))$ for all $x,y\in \XA^{1}$ and $a\in A$
\item $(ax)(b) = a\big( x(b)\big)$ for all $x\in \XA^{1}$ and $a,b\in A$
\item $[x,ay]= x(a)y+a[x,y]$ for all $x,y\in \XA^{1}$ and $a\in A$
\end{enumerate}
where for any $x\in \XA^{1}$ and $a\in A$, we abuse notation and denote $\tau(x)(a)$ by $x(a)$. Observe that axioms (I) and (II) make $(\XA^{1},[,])$ a Lie algebra and (III) simply states that $\tau :\XA^{1}\rightarrow Der(A)$ is a Lie algebra morphism, where the Lie bracket on $Der(A)$ is defined by $[\phi,\psi]= \phi\psi-\psi\phi$, for $\phi,\psi\in Der(A)$.

The universal enveloping algebra of a Lie-Rinehart algebra $(A,\XA^{1})$, denoted by  $V(A,\XA^{1} )$, was described by Rinehart in \cite{rinehart1963differential}. Originally, this algebra was defined as the universal enveloping algebra of a Lie structure on $A\oplus \XA^{1}$. Alternatively, one can formulate $V(A,\XA^{1})$ as the quotient of the free algebra $A \star T\XA^{1}$ by relations 
\begin{equation}
a\sbt x = ax \quad, ax = x\sbt a + x(a), \quad x\sbt y = y\sbt x + [x,y] 
\end{equation}
for all $x,y\in \XA^{1}$ and $a\in A$. It is now common knowledge that $V(A,\XA^{1} ) $ admits a Hopf algebroid structure \cite{kowalzig2010duality,moerdijk2008enveloping}, which induces the same actions described in (\ref{EqTXHopf}). The principle geometric example this construction is generalising is the \emph{algebra of differential operators} on a smooth manifold: if $A$ is the algebra of smooth functions on a smooth finite dimensional manifold and $\XA^{1}$ the Lie algebra of smooth vector fields on the manifold, then $V(A,\XA^{1})$ is isomorphic to the algebra of differential operators on the manifold and $(A,\XA^{1})$-modules or equivalently $V(A,\XA^{1} )$-modules are known to be equivalent to the usual notion of \emph{flat connections} \cite{huebschmann1990poisson}.

Let $\XA^{1}$ be a fgp $A$-module with dual $\Omega^{1}$. Since $\XA^{1}$ is a symmetric bimodule, it does not matter, whether we ask $\XA^{1}$ to be left or right fgp and $\Omega^{1}$ will also be symmetric as a bimodule. We have a bijection between linear maps $\tau :\XA^{1}\rightarrow Der(A)$ satisfying (IV) and first order calculi on $\Omega^{1}$ i.e. linear maps $d:A\rightarrow \Omega^{1}$ satisfying the Leibnitz rule  : 
$$x (a) = \mathrm{ev}(x\otimes da)  \quad \longleftrightarrow \quad d(a)= x_{i}(a)\omega_{i} $$
where we denote the evaluation and coevaluation maps as in previous sections. In this setting the linear map $[,]:\XA^{1}\otimes_{\field} \XA^{1}\rightarrow \XA^{1}$ allows one to extend the calculus to $\Omega^{2}= \bigwedge^{2}(\Omega^{1})$, where $\bigwedge^{2}(\Omega^{1})$ is the exterior power of $\Omega^{1}$ as an $A$-module. If $[,]$ is antisymmetric and (V) holds, then define $d:\Omega^{1} \rightarrow \Omega^{2}$ by 
$$d \omega = \sum_{i<j}\left[x_{i}(\mathrm{ev}(x_{j}\otimes\omega))-x_{j}(\mathrm{ev}(x_{i}\otimes\omega))- \mathrm{ev}([x_{i},x_{j}]\otimes\omega ) \right] \omega_{i}\wedge\omega_{j} $$
for $\omega\in \Omega^{1}$. In particular, condition (V) makes $d$ into a well-defined map with $\bigwedge^{2}(\Omega^{1})$ as its codomain. By definition $d$ satisfies the Leibnitz rule, $d(a\omega)=da\wedge \omega + ad\omega$, but $d$ extending the differential of the calculus i.e. $d^{2}=0$, is equivalent to (III) holding. Since $\Omega^{1}$ is a fgp module, then $\bigwedge^{2}(\Omega^{1})$ is also fgp with $\XA^{2}= \bigwedge^{2}(\XA^{1})$ as its dual, and the coevalation and evaluation maps defined by $\fr{coev} (1) = \omega_{i_{1}}\wedge \omega_{i_{2}} \otimes x_{i_{1}}\wedge x_{i_{2}}$ and 
\begin{align*}
\fr{ev}(x\wedge y\otimes \omega\wedge \rho)=  \big[ \mathrm{ev} (x\otimes \omega)\mathrm{ev}(y\otimes\rho) - \mathrm{ev}(y\otimes\omega) \mathrm{ev}(x\otimes \rho)\big]
\end{align*}
respectively, where we use the notation $i_{1}$ and $i_{2}$ to denote the sum over indices such that $i_{1}<i_{2}$. In this setting, observe that the sheaf of differential operators $\ct{D}_{A}$, as defined in \cite{beggs2019quantum}, which is the quotient of $T\XA^{1}_{\sbt}$ by the flat relation (\ref{EqFlat}), is exactly isomorphic to $V(A,\XA^{1})$: for any $x\wedge y\in \XA^{2}$, we can use identities $[x,y]a=x(a)y-[x,ay]$ and $[x,y]a=y(a)x+[xa,y]$ to expand the relation (\ref{EqFlat}) 
\begin{align*}
0&=\fr{ev}(x\wedge y\otimes d\omega_{i})\sbt x_{i} -\fr{ev}(x\wedge y\otimes \omega_{j}\wedge\omega_{k}) \sbt x_{k}\sbt x_{j} 
\\ =&\big[x_{i_{1}}(\mathrm{ev}(x_{i_{2}}\otimes\omega_{l}))-x_{i_{2}}(\mathrm{ev}(x_{i_{1}}\otimes\omega_{l})) \big]\fr{ev}\big(x\wedge y\otimes \omega_{i_{1}}\wedge\omega_{i_{2}}\big) x_{l}
\\ &\hspace{-0.25cm}- \fr{ev}\big(x\wedge y\otimes \omega_{i_{1}}\wedge\omega_{i_{2}}\big) \mathrm{ev}([x_{i_{1}},x_{i_{2}}]\otimes \omega_{l})x_{l}
\\ &\hspace{-0.25cm}- \big[ \mathrm{ev} (x\otimes \omega_{j})\mathrm{ev}(y\otimes\omega_{k}) - \mathrm{ev}(y\otimes\omega_{j}) \mathrm{ev}(x\otimes \omega_{k})\big]x_{k}\sbt x_{j}
\\  =& \big(x(\mathrm{ev}(x_{j}\otimes\omega_{l}))\mathrm{ev}(y\otimes \omega_{j})-y(\mathrm{ev}(x_{i}\otimes\omega_{l}))\mathrm{ev}(x\otimes\omega_{i})\big)x_{l} 
\\&\hspace{-0.25cm}-\big([x_{i_{1}}\mathrm{ev}(x_{i_{1}}\otimes x),x_{i_{2}}\mathrm{ev}(x_{i_{2}}\otimes y)]- [x_{i_{1}}\mathrm{ev}(x_{i_{1}}\otimes y),x_{i_{2}}\mathrm{ev}(x_{i_{2}}\otimes x)]
\\ &\hspace{-0.25cm}+ y(\mathrm{ev}(x\otimes \omega_{i}))x_{i}- x(\mathrm{ev}(y\otimes \omega_{i}))x_{i}\big)- y\sbt x +x\sbt y+y(\mathrm{ev}(x\otimes \omega_{j}))x_{j}
\\ &\hspace{-0.25cm}  -x(\mathrm{ev} (y\otimes \omega_{j}))x_{j}= x\sbt y -y\sbt x- [x,y]
\end{align*}
As mentioned previously, the Hopf algebroid structure of $V(A,\XA^{1})$ induces the same actions described in (\ref{EqTXHopf}). These actions can be recovered from the actions of $\ct{D}\XA$ after quotienting out the relations (\ref{EqQuotRel}). Note that under these relations, the additional relation in $\ct{D}\XA$, (\ref{EqFlExtraa}), holds trivially and by the above calculation, this quotient is exactly isomorphic to $\ct{D}_{A}$ and $V(A,\XA^{1})$. 

Observe the statement of Theorem \ref{TDXAAnti} mirrors that of Proposition 3.11 in \cite{kowalzig2011cyclic}. In \cite{kowalzig2011cyclic}, it is proven that $V(A,\XA^{1})$, as a left Hopf algebroid, admits an antipode in the sense of B{\"o}hm and Szlach{\'a}nyi, if and only if there exists a right action of $V(A,\XA^{1})$ on $A$. In Theorems \ref{THXaAnti} and \ref{TDXAAnti}, we have demonstrated that $\ct{D}\XA$ admits an antipode if and only if $A$ has a right action of $\ct{D}\XA$, and thereby recovered the map $\Upsilon :\XA^{1}\rightarrow A$. In particular, our construction of the antipode presented in Theorem \ref{THXaAnti} agrees with the antipode presented in \cite{kowalzig2011cyclic}, after applying relations (\ref{EqQuotRel}). It is a consequence of the work in \cite{kowalzig2011cyclic} and Theorem 3 of \cite{huebschmann1998lie} that Lie-Rinehart Hopf algebroids admit an antipode when $\XA^{1}$ is finitely generated and projective. It is not clear to the author, whether $\ct{D}\XA$ also admits an antipode for any Lie-Rinehart algebra $(A,\XA^{1})$, although $H\XA^{1}$ admits an antipode with $\Upsilon (x)=x_{i}(\mathrm{ev}(x\otimes\omega_{i})) $. We must also note that all examples of Hopf algebroids which do not admit antipodes, presented in \cite{krahmer2015lie,rovi2015lie}, do not cross over to our work, since $\XA^{1}$ is not finitely generated and projective in these examples. 

Note that while $V(A,\XA^{1} )$ is often constructed as the universal enveloping algebra of a certain Lie algebra, the Jacobi identity (II) holding is not required to define $V(A,\XA^{1})$. From the point of view of differential forms, flat connections only need $\Omega^{1}$ and $\Omega^{2}$ in the dga, to be defined. The Jacobi identity holding is actually equivalent to the calculus extending to $\Omega^{3}=\bigwedge^{3}(\Omega^{1})$. This is additional datum which we do not require for our construction. Furthermore, while in the commutative case the existence of $d:\Omega^{1}\rightarrow \Omega^{2}$ is equivalent to the existence a Lie-like bracket on $\XA^{1}$, in the noncommutative case, this does not necessarily provide an asymmetric map on $\XA^{1}\otimes\XA^{1}$. We refer the reader to Section 6.1 of \cite{beggs2019quantum} for a brief discussion on this.

\textbf{Acknowledgements.} The author would like to thank Shahn Majid for pointing out the original question this work is based on, in Chapter 6 of \cite{beggs2019quantum}, reviewing early drafts of this work and numerous stimulating discussions. We would also like to show gratitude to Edwin Beggs for fruitful conversations and Laiachi El Kaoutit for referring the author to \cite{andre2001differentielles} and useful suggestions which lead to the addition of Section \ref{SCommLie}.
\bibliographystyle{plain}
\bibliography{hopf}

\begin{thebibliography}{10}

\bibitem{andre2001differentielles}
Yves Andr{\'e}.
\newblock Diff{\'e}rentielles non commutatives et th{\'e}orie de {Galois}
  diff{\'e}rentielle ou aux diff{\'e}rences.
\newblock In {\em Annales scientifiques de l'Ecole normale sup{\'e}rieure},
  volume~34, pages 685--739, 2001.

\bibitem{auslander1997representation}
Maurice Auslander, Idun Reiten, and Sverre~O Smalo.
\newblock {\em Representation theory of {Artin} algebras}, volume~36.
\newblock Cambridge university press, 1997.

\bibitem{beggs2014noncommutative}
Edwin~J Beggs and Tomasz Brzezi{\'n}ski.
\newblock Noncommutative differential operators, {Sobolev} spaces and the
  centre of a category.
\newblock {\em Journal of Pure and Applied Algebra}, 1(218):1--17, 2014.

\bibitem{beggs2017spectral}
Edwin~J Beggs and Shahn Majid.
\newblock Spectral triples from bimodule connections and {Chern} connections.
\newblock {\em Journal of Noncommutative Geometry}, 11(2):669--701, 2017.

\bibitem{beggs2019quantum}
Edwin~J Beggs and Shahn Majid.
\newblock {\em Quantum Riemannian Geometry}.
\newblock 2020.

\bibitem{bespalov1997crossed}
Yu~N Bespalov.
\newblock Crossed modules and quantum groups in braided categories.
\newblock {\em Applied categorical structures}, 5(2):155--204, 1997.

\bibitem{BOHM2009173}
Gabriella B{\"o}hm.
\newblock Hopf algebroids.
\newblock In M.~Hazewinkel, editor, {\em Handbook of Algebra}, volume~6 of {\em
  Handbook of Algebra}, pages 173 -- 235. North-Holland, 2009.

\bibitem{bohm2018hopf}
Gabriella B{\"o}hm.
\newblock {\em Hopf Algebras and Their Generalizations from a Category
  Theoretical Point of View}, volume 2226.
\newblock Springer, 2018.

\bibitem{bohm2004hopf}
Gabriella B{\"o}hm and Korn{\'e}l Szlach{\'a}nyi.
\newblock Hopf algebroids with bijective antipodes: axioms, integrals, and
  duals.
\newblock {\em Journal of Algebra}, 274(2):708--750, 2004.

\bibitem{bresser1996non}
Klaus Bresser, Folkert Mueller-Hoissen, Aristophanes Dimakis, and Andrzej
  Sitarz.
\newblock Non-commutative geometry of finite groups.
\newblock {\em Journal of Physics A: Mathematical and General}, 29(11):2705,
  1996.

\bibitem{bruguieres2011hopf}
Alain Bruguieres, Steve Lack, and Alexis Virelizier.
\newblock Hopf monads on monoidal categories.
\newblock {\em Advances in Mathematics}, 227(2):745--800, 2011.

\bibitem{connes1994noncommutative}
Alain Connes.
\newblock {\em Noncommutative geometry}.
\newblock Springer, 1994.

\bibitem{dubois1996first}
Michel Dubois-Violette and Thierry Masson.
\newblock On the first-order operators in bimodules.
\newblock {\em Letters in Mathematical Physics}, 37(4):467--474, 1996.

\bibitem{DUBOISVIOLETTE1996218}
Michel Dubois-Violette and Peter~W. Michor.
\newblock Connections on central bimodules in noncommutative differential
  geometry.
\newblock {\em Journal of Geometry and Physics}, 20(2):218 -- 232, 1996.

\bibitem{fiore2000leibniz}
Gaetano Fiore and John Madore.
\newblock Leibniz rules and reality conditions.
\newblock {\em The European Physical Journal C-Particles and Fields},
  17(2):359--366, 2000.

\bibitem{hai2008tannaka}
Ph{\`u}ng~H{\^o} Hai.
\newblock Tannaka-krein duality for {Hopf} algebroids.
\newblock {\em Israel Journal of Mathematics}, 167(1):193, 2008.

\bibitem{huebschmann1990poisson}
Johannes Huebschmann.
\newblock Poisson cohomology and quantization.
\newblock {\em J. reine angew. Math}, 408:57--113, 1990.

\bibitem{huebschmann1998lie}
Johannes Huebschmann.
\newblock {Lie}-{Rinehart} algebras, {Gerstenhaber} algebras and
  {Batalin}-{Vilkovisky} algebras.
\newblock In {\em Annales de l'institut Fourier}, volume~48, pages 425--440,
  1998.

\bibitem{el2016finite}
Laiachi~El Kaoutit and Jos{\'e} G{\'o}mez-Torrecillas.
\newblock Finite dual of a cocommutative {Hopf} algebroid. application to
  linear differential matrix equations and {Picard}-{Vessiot} theory.
\newblock {\em arXiv preprint arXiv:1607.07633}, 2016.

\bibitem{kowalzig2011cyclic}
N~Kowalzig and H~Posthuma.
\newblock The cyclic theory of {Hopf} algebroids.
\newblock {\em Journal of Noncommutative Geometry}, 5(3):423--476, 2011.

\bibitem{kowalzig2009hopf}
Niels Kowalzig.
\newblock {\em Hopf algebroids and their cyclic theory}.
\newblock PhD thesis, Utrecht Universiry, 2009.

\bibitem{kowalzig2010duality}
Niels Kowalzig and Ulrich Kr{\"a}hmer.
\newblock Duality and products in algebraic (co) homology theories.
\newblock {\em Journal of Algebra}, 323(7):2063--2081, 2010.

\bibitem{krahmer2015lie}
Ulrich Kr{\"a}hmer and Ana Rovi.
\newblock A {Lie}-{Rinehart} algebra with no antipode.
\newblock {\em Communications in Algebra}, 43(10):4049--4053, 2015.

\bibitem{lu1996hopf}
Jiang-Hua Lu.
\newblock Hopf algebroids and quantum groupoids.
\newblock {\em International Journal of Mathematics}, 7(01):47--70, 1996.

\bibitem{mac2013categories}
Saunders Mac~Lane.
\newblock {\em Categories for the working mathematician}, volume~5.
\newblock Springer Science, 2013.

\bibitem{mackenzie2005general}
Kirill~CH Mackenzie.
\newblock {\em General theory of Lie groupoids and Lie algebroids}, volume 213.
\newblock Cambridge University Press, 2005.

\bibitem{majid1992braided}
Shahn Majid.
\newblock Braided groups and duals of monoidal categories.
\newblock In {\em Canad. Math. Soc. Conf. Proc}, volume~13, pages 329--343,
  1992.

\bibitem{majid1998classification}
Shahn Majid.
\newblock Classification of bicovariant differential calculi.
\newblock {\em Journal of Geometry and Physics}, 25(1-2):119--140, 1998.

\bibitem{majid2013noncommutative}
Shahn Majid.
\newblock Noncommutative riemannian geometry on graphs.
\newblock {\em Journal of Geometry and Physics}, 69:74--93, 2013.

\bibitem{majid2018generalised}
Shahn Majid and Wenqing Tao.
\newblock Generalised noncommutative geometry on finite groups and {Hopf}
  quivers.
\newblock {\em Journal of Noncommutative Geometry}, 13(3):1055--1116, 2019.

\bibitem{moerdijk2008enveloping}
Ieke Moerdijk and Janez Mr{\v{c}}un.
\newblock On the enveloping algebra of a {Lie}-{Rinehart} algebra.
\newblock Technical report, US National Science Foundation, Cornell University,
  2008.

\bibitem{rinehart1963differential}
George~S Rinehart.
\newblock Differential forms on general commutative algebras.
\newblock {\em Transactions of the American Mathematical Society},
  108(2):195--222, 1963.

\bibitem{rovi2015lie}
Ana Rovi.
\newblock {\em Lie-{Rinehart} algebras, {Hopf} algebroids with and without an
  antipode}.
\newblock PhD thesis, University of Glasgow, 2015.

\bibitem{schauenburg2000algebras}
Peter Schauenburg.
\newblock Duals and doubles of quantum groupoids ($\times_{R}$-{Hopf}
  algebras).
\newblock In {\em New Trends in Hopf Algebra Theory: Proceedings of the
  Colloquium on Quantum Groups and Hopf Algebras, La Falda, Sierras de
  C{\'o}rdoba, Argentina, August 9-13, 1999}, volume 267, page 273. American
  Mathematical Soc., 2000.

\bibitem{szlachanyi2003monoidal}
Korn{\'e}l Szlach{\'a}nyi.
\newblock The monoidal {Eilenberg}-{Moore} construction and bialgebroids.
\newblock {\em Journal of Pure and Applied Algebra}, 182(2-3):287--315, 2003.

\bibitem{takeuchi1977groups}
Mitsuhiro Takeuchi.
\newblock Groups of algebras over $\ {A}\otimes \ov{A}$.
\newblock {\em Journal of the Mathematical Society of Japan}, 29(3):459--492,
  1977.

\bibitem{watts1960intrinsic}
Charles~E Watts.
\newblock Intrinsic characterizations of some additive functors.
\newblock {\em Proceedings of the American Mathematical Society}, 11(1):5--8,
  1960.

\bibitem{woronowicz1989differential}
Stanis{\l}aw~L Woronowicz.
\newblock Differential calculus on compact matrix pseudogroups (quantum
  groups).
\newblock {\em Communications in Mathematical Physics}, 122(1):125--170, 1989.

\bibitem{xu1998quantum}
Ping Xu.
\newblock Quantum groupoids and deformation quantization.
\newblock {\em Comptes Rendus de l'Acad{\'e}mie des Sciences-Series
  I-Mathematics}, 326(3):289--294, 1998.

\bibitem{xu2001quantum}
Ping Xu.
\newblock Quantum groupoids.
\newblock {\em Communications in Mathematical Physics}, 216(3):539--581, 2001.

\end{thebibliography}
\end{document}